\documentclass[12pt]{amsart}

\usepackage{graphicx}
\usepackage{amsthm}
\usepackage{amsmath, amssymb}
\swapnumbers

\theoremstyle{plain}
\newtheorem{Thm}{Theorem}[section]
\newtheorem{Coro}[Thm]{Corollary}
\newtheorem{Lem}[Thm]{Lemma}

\theoremstyle{definition}
\newtheorem{Def}[Thm]{Definition}

\newcommand{\mT}{\mathcal{T}}
\newcommand{\mS}{\mathcal{S}}

\newcommand{\mP}{\mathcal{P}}

\newcommand{\mD}{\mathcal{D}}
\newcommand{\comment}[1]{}

\begin{document}

\title{Calculating isotopy classes of Heegaard splittings}

\author{Jesse Johnson}
\address{\hskip-\parindent
        Department of Mathematics \\
        Oklahoma State University \\
        Stillwater, OK 74078
        USA}
\email{jjohnson@math.okstate.edu}


\thanks{Research supported by NSF MSPRF grant 0602368}

\begin{abstract}
We show that given a partially flat angled ideal triangulation for a 3-manifold $M$ with boundary (as defined by Lackenby), there is an algorithm to produce a list of Heegaard splittings for $M$ such that below a given bound $g$, each isotopy class appears exactly once.  In particular, this algorithm determines precisely when two almost normal surfaces represent Heegaard splittings that are isotopic in the ambient 3-manifold.  A closely related algorithm determines the smallest genus common stabilization of any two Heegaard splittings on the list.  The methods, in fact, characterize isotopies between Heegaard surfaces in any triangulation, but the existence of infinitely many normal surfaces of the same genus prevents this characterization from being algorithmic in general.
\end{abstract}

\maketitle


\section{Introduction}

Consider a 3-manifold $M$ and a triangulation $\mT$ for $M$.  It is well known that many isotopy classes of surfaces in $M$ can be represented by normal and almost normal surfaces with respect to $\mT$, i.e. surfaces that intersect the tetrahedra in simple pieces of a finite number of types.  In particular, incompressible surfaces and strongly irreducible Heegaard surfaces can be represented in this manner.  Such surfaces can be classified algorithmically using linear programming techniques and knowing that every Heegaard surface appears on a calculable list of representative surfaces is the basis for algorithms to calculate the Heegaard genus of a 3-manifold~\cite{lacknb:1eff},~\cite{li:alg} (the former for hyperbolic 3-manifolds with boundary and the latter for closed atoroidal 3-manifolds.)

When one wants to extend these results to distinguish the different isotopy classes of Heegaard surfaces in $M$, the problem arises that while every isotopy class is represented by some normal or almost normal surface, it may by represented by more than one.  In this paper, we present criteria to determine when two almost normal surfaces are isotopic in the ambient manifold.  This allows one to replace the list of representatives with a list of isotopy classes.  In a partially flat angled ideal triangulation of a 3-manifold with boundary, as defined by Lackenby~\cite{lacknb:1eff} there are finitely many normal and almost normal surfaces of any given genus, so there is a finite algorithm to determine this criteria.

More precisely, we will encode the Heegaard splittings in terms of a complex $\mD = \mD(M, \mT^1)$ determined by the 3-manifold $M$ and a graph $\mT^1 \subset M$.  (In fact, we also allow $\mT^1$ to have closed loop components.)  For the purposes of the algorithm, this graph will be the 1-skeleton of the triangulation $\mT$, though we will develop the machinery in terms of any graph.  

The vertices of $\mD$ are surfaces in $M$ that are incompressible, strongly irreducible or satisfy a third property called index-two with respect to $\mT^1$.  (These terms will be defined with respect to a graph below, but when $\mT^1$ is empty, they correspond to the the standard definitions.)  We fix a specific pair of vertices $v_-, v_+$ depending on $M$.  Edges in $\mD(M, \mT^1)$ are oriented and correspond to compression bodies in $M$ bounded by the union of their endpoints.  Heegaard splittings in $M$ are represented by oriented paths in $\mD(M, \mT^1)$ between two fixed vertices $v_-$, $v_+$.  Every vertex in $\mD(M, \mT^1)$ has an associated genus and from this we define a genus for each path.  We define a sequence of nested, $n$-dimensional subcomplexes $\{\mD^n(M, \mT^1)\}$ of $\mD(M, \mT^1)$ and show that $\mD^1$ and $\mD^2$ encode the structure of the set of incompressible and Heegaard surfaces in $M$ as follows:

\begin{Thm}
\label{mainthm1}
A vertex $v$ in $\mD^1(M, \mT^1)$ represents an incompressible surface if and only if every edge path starting at $v$ in which the genus of the vertices does not increase is constant in genus.  The incompressible surfaces represented by two such vertices will be isotopic if and only if there is an edge path between the two in which the genus is constant.

Every oriented path in $\mD^1(M, \mT^1)$ from $v_-$ to $v_+$ represents a Heegaard splitting for $M$ and every irreducible Heegaard splitting for $M$ is represented by such a path in $\mD^1(M, \mT^1)$.  A path will represent an irreducible Heegaard splitting if and only if there is no sequence of face slides in $\mD^2(M, \mT^1)$, starting with this path, such that the genus is non-increasing, but the final path has lower genus than the starting path.  Two paths represent isotopic Heegaard splittings if and only if there is a sequence of face slides from one to the other in which the genus stays constant.
\end{Thm}

Because $\mD^2$ encodes the isotopy classes of Heegaard splittings for $M$, the algorithm promised above would follow from an algorithm that constructs $\mD^2(M, \mT^1)$ where $\mT^1$ is the 1-skeleton of a triangulation for $M$. This is our second result:

\begin{Thm}
\label{mainthm2}
Given a positive integer $g$, if $\mT$ is partially flat angled ideal triangulation for a 3-manifold $M$ then there is an algorithm that constructs the subcomplex of $\mD^2(M, \mT^1)$ spanned by all vertices representing surfaces of genus less than or equal to $g$.
\end{Thm}

The first of these two theorems is proved using a formulation of thin position in terms of a cell complex whose vertices represent all isotopy classes of surfaces transverse to $\mT^1$.  The description of this complex and the results leading up to the proof of Theorem~\ref{mainthm1} are proved in Sections~\ref{thinpsect} through~\ref{surfderivedsect}.  Sections~\ref{oddsect} and~\ref{catsect} contain a few more results about this version of thin position that, while not directly applicable to Theorem~\ref{mainthm1}, should be useful for future applications.  Section~\ref{gordonsect} contains an exposition of Dave Bachman's proof of the Gordon conjecture in terms of this approach to thin position.  The proof of Theorem~\ref{mainthm2} is described in Sections~\ref{normalsect} through~\ref{thm2sect}.  

The motivation for this paper came from attempting to apply the ideas from Bachman's proof of the Gordon conjecture~\cite{bach:gordon} to Stockings'~\cite{stocking} and Rubinstein's~\cite{rub:alnormal} methods of constructing almost normal surfaces.  Bachman presented some results along these lines in~\cite{bachman:normal}.  I want to thank Robin Wilson for explaining Stocking's paper to me and Dave Bachman for discussing his approach to thin position with me.  I thank Martin Scharlemann and Scott Tayler for making a number of suggestions to improve the exposition, and Marion Moore for convincing me to ignore trivial spheres (which makes things much simpler.)  I thank Daryl McCullough and Sangbum Cho for pointing me towards Cho's work on contractible complexes in~\cite{cho}, which provided the proof of Lemma~\ref{disksetcontractible}. And I thank Trent Schirmer for pointing out a problem with the original definition of the complex of surfaces, which has been fixed in the current version of the paper.

\section{Thin position}
\label{thinpsect}

The notion of ``thin position'' has been defined in a number of ways and in a number of contexts, including (in historical order) Gabai's thin position for knots ~\cite{gabai}, Scharlemann and Thompson's thin position for 3-manifolds~\cite{sch:thin}, Hayashi and Shimokawa's generized thin position for knots~\cite{hayashimo}, which combines the first two of these, Tomova's generalization~\cite{tom:brcompare} to include cut disks, and Taylor and Tomova's further generalization to embedded graphs~\cite{tomovataylor}.  A notion of circular thin position for 3-manifolds with infinite first homology has also been intoduced by Manjerrez-Gutierrez~\cite{manjarrez}, and Stevens~\cite{stevens} defined a form of thin position surfaces containing a given knot.  Bachman, Derby-Talbot and Sedgwick have applied a thin position argument to isotopies of incompressible surfaces relative to a knot~\cite{bdts:dehn}.

Schulten's work with generalized Heegaard splittings~\cite{sch:cross},~\cite{schl:vert} has shown how powerful thin position for 3-manifolds can be for understanding Heegaard splittings, and Tomova~\cite{tom:brcompare} has extended many of these ideas to thin position for knots.  Recently, Bachman~\cite{bach:gordon}~\cite{bach:stabs} has proved a number of interesting results by applying thin position to sequences of generalized Heegaard splittings.  This iteration of the thin position arguments proves to be a very natural approach to many problems.

The connections between these different ideas are somewhat subtle and may not be visible without coming to a fairly deep understanding of the definitions.  In practice, thin position is more of a philosophy than a concrete mathematical object.  One goal of this paper is to present a framework for defining thin position in different contexts.  This will lead to a fairly concrete treatment of some of the more advanced results about thin position that are known by experts, but are rather obscure in the literature.  

In all of the different flavors of thin position, one applies lexicographic/dictionary ordering to a set of presentations for a given object and finds that any presentation can be simplified to a presentation satisfying certain nice properties.  We will define thin position in terms of a cell complex $\mS$ with a partial ordering on its vertices, satisfying a number of axioms that will be introduced throughout the paper.  (A list of the axioms is provided in Appendix~\ref{axiomsect} for reference.)  We consider paths in $\mS$, applying lexicographic ordering to the sets of local maxima.  By choosing the appropriate complex and the appropriate ordering, one can recover any of the existing notions of thin position.

In order to recover Scharlemann and Thompson's thin position for 3-manifolds, one constructs a complex $\mS$ whose vertices represent isotopy classes of separating surfaces in a given manifold.  Edges correspond to compressing disks for surfaces and faces come from disjoint compressing disks.  We define conditions corresponding to incompressible and strongly irreducible surfaces entirely in terms of the combinatorics of the cell complex in such a way that in this complex, the definition agrees with the traditional definitions.  In particular, paths in $\mS(M, \emptyset)$ correspond to generalized Heegaard splittings.  In the abstract setting, any path can be ``thinned'' to a path whose locally maximal vertices are strongly irreducible and whose local minima are incompressible.  This is Scharlemann and Thompson's main theorem in~\cite{sch:thin}.

In order to recover Bachman's work~\cite{bachman} on sequences of generalized Heegaard splittings, we construct a complex $\mP(\mS)$ whose vertices are paths in $\mS$.  The 1-skeleton of this complex is essentially Schulten's width complex~\cite{sch:width}.  A path in $\mP(\mS)$ corresponds to a sequence of paths in $\mS$.  In the context of Scharlemann-Thompson thin position, this is a sequence of generalized Heegaard splittings.  The same techniques used to characterise thin paths in $\mS$ can be used to characterize thin paths in $\mP(\mS)$ that define thin sequences of paths in $\mS$.

We sum up these results by constructing from $\mS(M, \mT^1)$ the complex $\mD(m, \mT^1)$ described above.  The arguments of the previous sections imply that this derived complex contains all the necessary information to understand paths and sequences of paths in $\mS$, as follows:

\begin{Thm}
\label{mainthm}
Every path in $\mS$ can be thinned to produce a path represented in $\mD(\mS)$.  Any sequence of paths in $\mS$ that are related by face slides in $\mS$ can be thinned to a sequence of paths that is represented by face slides in $\mD(\mS)$.
\end{Thm}

This Theorem is proved for any complex $\mS$ satisfying a certain collection of axioms.  The main example used in this paper, the complex of surfaces, is defined in Section~\ref{surfacecomplessect}.  This complex is defined for a 3-manifold $M$ and a graph $\mT^1$ in the 3-manifold.  Using this as motivation, we define our first two axioms in Section~\ref{heightsect}, and define a preliminary notion of thin position in Section~\ref{thinpathsect}.  This primordial form of thin position is useful for understanding incompressible surfaces and for iterating thin position later, but is not quite complete.

Before defining thin position in its final form, we discuss in Section~\ref{ksplittingsect} how certain paths in $\mS(M, \mT^1)$ correspond to compression bodies in $M$.  This leads to a definition of orientations on the edges of the complex of surfaces in Section~\ref{orsect}, and to the definition of our final form of thin position in Section~\ref{thinorsect}.  The main result of this section is that the local maxima of thin paths satisfy a condition that is equivalent to the traditional notion of strongly irreducible.

In traditional thin position, the thin levels (local minima in this setting) are understood by using a result proved by Casson and Gordon~\cite{cass:red}, using a generalization of a theorem of Haken~\cite{haken}.  We present in Section~\ref{cgsect} an axiom that is equivalent to this result, and show that it implies the expected result about local minima of thin paths.   We prove in Sections~\ref{kbodycplxsect},~\ref{diskssect} and~\ref{pathlinksect} that the complex of surfaces satisfies this axiom.  

With the end of Section~\ref{pathlinksect}, we have proved, in the axiomatic setting, the main results of Scharlemann and Thompson's original paper~\cite{sch:thin}.  In Sections~\ref{iteratedsect} and~\ref{projectionsect}, we define the path complex for the complex of surfaces and use this to ``iterate'' the notion of thin position, recovering a more general version of Bachman's sequences of generalized Heegaard splittings.  

In Section~\ref{linksect}, we discuss the difference between ``strongly irreducible'' and ``weakly incompressible'', showing that in many (but not all) cases they are the same.  In Sections~\ref{splittingpathsect} and~\ref{thinheegsect}, we discuss in more depth the process of ``thinning'' and its inverse, then in Section~\ref{heegsect}, we relate this all to Heegaard splittings.  This recovers many of the facts about amalgamation introduced by Schultens~\cite{sch:cross}.  Finally, in Section~\ref{derivedsect} we define the complex $\mD(\mS)$ discussed in Theorems~\ref{mainthm1} and~\ref{mainthm} and prove Theorem~\ref{mainthm}.  The proof of Theorem~\ref{mainthm1} is presented as a sequence of Lemmas in Section~\ref{surfderivedsect}.

\section{The complex of surfaces}
\label{surfacecomplessect}

We begin by considering an example that will be used to provide context for the axioms we introduce throughout the paper.  Let $M$ be a 3-manifold and $\mT^1$ be the union of a (possibly empty) properly embedded graph and a disjoint (possibly empty) link.  The reader may prefer (at least at first) to consider only the case when $\mT^1$ is empty.  This greatly simplifies the definitions and recovers Scharlemann and Thompson's thin position for 3-manifolds.  However, the more general case is necessary for the main application in this paper.

Given surfaces $S$ and $S'$, both transverse to $\mT^1$, we will say that $S$ is \textit{transversely} isotopic to $S'$ if there is an isotopy from $S$ to $S'$ in which each intermediate surface is transverse to $\mT^1$.  In such an isotopy, the number of points of intersection with each edge of $\mT^1$ is preserved, though the points of intersection may move within the edges and loops of $\mT^1$.

A surface $S$ in $M$ will be called \textit{strongly separating} if we can label the components of $M \setminus S$ either $+$ or $-$ so that each component of $S$ is in the closure of one positive component and one negative component.  (Note that a connected separating surface is strongly separating.)  An \textit{oriented, strongly separating surface} is a strongly separating surface along with a choice of labels for the complement.  Because $M$ is connected, every strongly separating surface corresponds to exactly two oriented strongly separating surfaces.  However, if there is an isotopy from the surface to itself that interchanges the two sides, it will define a single isotopy class of oriented, strongly separating surfaces.

Consider a sphere component $R$ of $S \subset M$ disjoint from $\mT^1$. Let $S'$ be the result of attaching an embedded tube from $R$ to a second component of $S$. (The second component need not be a sphere.) Then we will say that $S'$ is the result of a \textit{sphere tubing} on $S$. Note that the meridian disk for this tube is a trivial loop in $S'$. Thus the inverse of a sphere tubing is not a compression in the strict sense. (Compressing disks are usually assumed to have essential boundary.)

The surface $S'$ is homeomorphic to $S$ as an abstract surface, but not necessarily isotopic in $M$.  For the most part, $S$ and $S'$ will be isotopic if and only of $R$ bounds a ball disjoint from $S$ and $\mT^1$. 
\begin{figure}[htb]
  \begin{center}
  \includegraphics[width=3.5in]{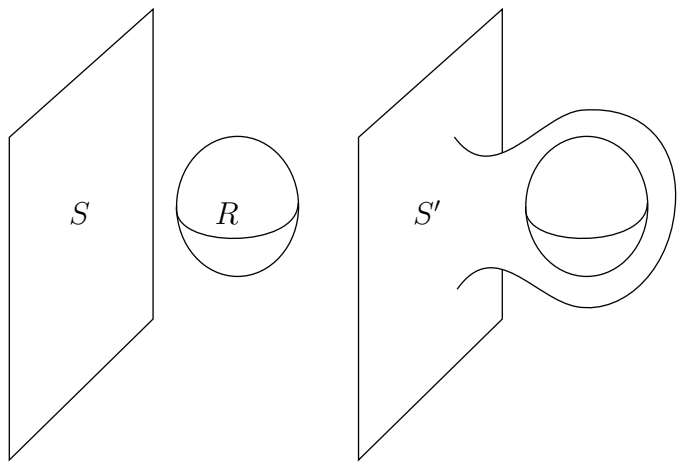}
  \put(-230,90){$S$}
  \put(-175,90){$R$}
  \put(-100,90){$S'$}
  \caption{Tubing a surface across a sphere produces a homeomorphic surface that is not, in general isotopic to the original.}
  \label{spheretubefig}
  \end{center}
\end{figure}

A \textit{trivial sphere} in $S$ is a sphere component of $S$ that bounds a ball in $M$ whose interior is disjoint from $S$ and $\mT^1$.  If $S$ is strongly separating and transversely oriented, then the result of adding or removing a trivial sphere to will also be strongly separating and there is a canonical way to define a transverse orientation for the resulting surface.  We will say that two surfaces $S$, $S'$ are \textit{sphere-blind isotopic}, or just \textit{blind isotopic} if they are related by a sequence of isotopies, sphere tubings, and adding or removing trivial sphere components.

A surface $S$ will be called \textit{bridge compressible} with respect to $\mathcal{T}^1$ if there is a disk $D$ with interior disjoint from $S$ and $\mathcal{T}^1$ such that $\partial D$ consists of an arc in $S$ and an arc in $\mathcal{T}^1$ disjoint from the vertices of $\mT^1$, as in Figure~\ref{edgesfig}.  This disk $D$ will be called a \textit{bridge disk}.

A \textit{tree disk} for $S$ is a disk $D$ containing a simply connected component of $\mT^1 \setminus S$ with a single vertex such that $\partial D$ is contained in $S$ and bounds a disk $E \subset S$ with interior disjoint from $\mT^1$.  We also require that $D \cup E$ bounds a ball in $M$ with interior disjoint from $\mT^1$.  A surface $S$ will be called \textit{tree compressible} with respect to $\mathcal{T}^1$ if there is a tree disk for $S$.

We will say that $S \setminus \mT^1$ is \textit{compressible (in the complement of $\mT^1$)} if there is a disk embedded in $M$ with interior disjoint from $S \cup \mT^1$ and whose boundary is an essential loop in $S \setminus \mT^1$.  (Note that the boundary of such a disk may be trivial in $S$ but essential in $S \setminus \mT^1$.)  This disk will be called a \textit{compressing disk}.  A \textit{K-disk} for $S$ is either a bridge disk, tree disk or a compressing disk for $S$.  

Recall that for a transversely oriented surface $S$, the components of $M \setminus S$ are labeled $+$ or $-$.  We will say that a K-disk is \textit{on the positive/negative side of $S$} if its interior is contained in the positive/negative complement of $S$.  A connected, strongly separating surface $S$ is \textit{K-compressible} if it has a K-disk and is called \textit{K-bicompressible} if there is both a K-disk on the positive side of $S$ and a K-disk on the negative side of $S$.

If a surface $S$ is K-compressible, we can form a new surface from $S$ as follows:  Let $D$ be a K-disk for $S$ and assume without loss of generality that $D$ is on the positive side of $S$.  Let $N$ be a closed regular neighborhood of $D$ in the positive complement of $S$.  Define $S'$ to be the surface that results from removing $N \cap S$ from $S$ and attaching $\partial N \setminus S$.  In the case of a tree disk, this construction creates a trivial sphere, which we can remove while preserving the blind isotopy class.  The surface $S'$ will be called a \textit{K-compression} of $S$, or we will say that $S'$ results from K-compression.  The three types of K-compressions are shown in Figure~\ref{edgesfig}.
\begin{figure}[htb]
  \begin{center}
  \includegraphics[width=4.5in]{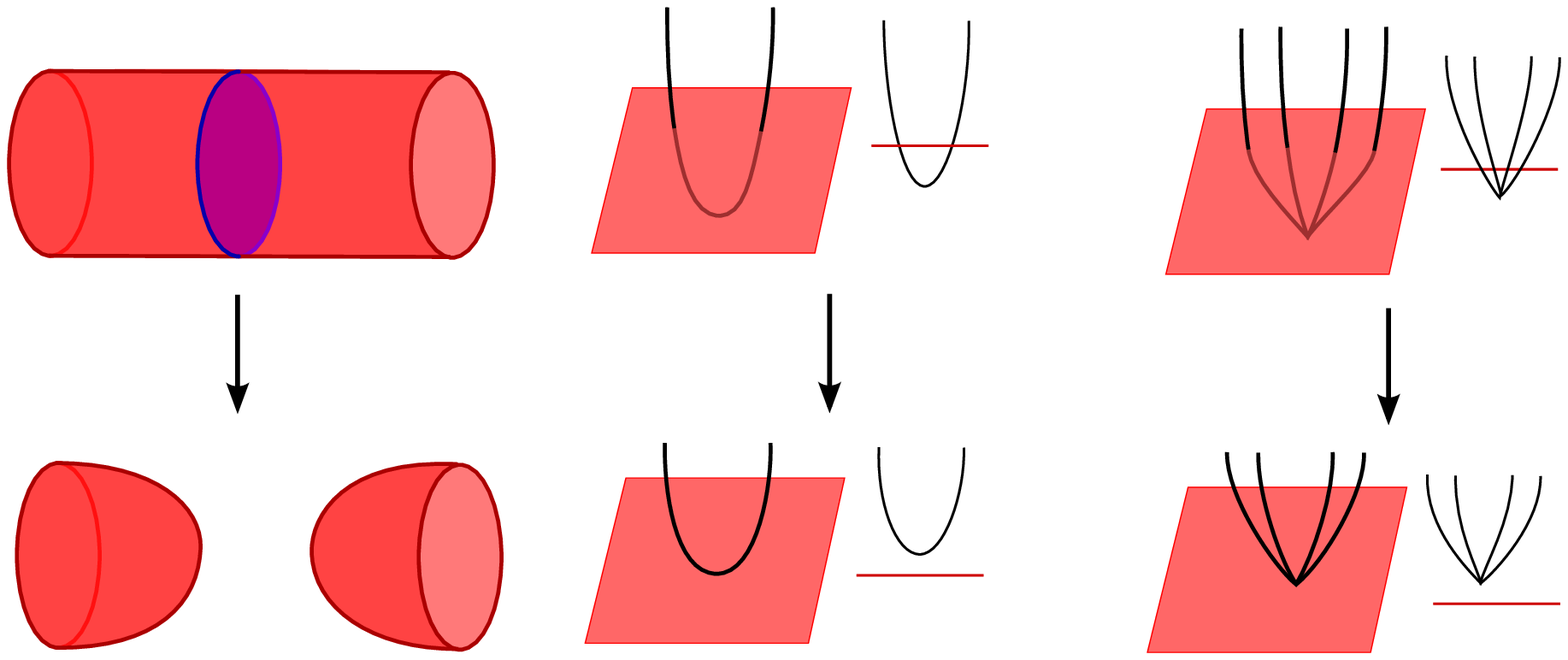}
  \caption{Edges of $\mS(M, \mT^1)$ are defined by K-compressions.}
  \label{edgesfig}
  \end{center}
\end{figure}

If the K-disk is a compressing disk then K-compression is what is normally just called compression:  We remove the annulus $N \cap S$ from $S$ and replace it with the two disks $\partial N \setminus S$.  If $D$ is a bridge disk or tree disk then $S'$ is isotopic in $M$ to $S$, or rather is the result of isotoping $S$ across part of the graph $\mT^1$ so as to reduce the number of intersections.  We will call this a \textit{bridge compression} or \textit{tree compression}, respectively.

The \textit{complex of surfaces} $\mS(M,\mT^1)$ will be the simplicial complex defined by surfaces and K-disks as follows: \\

\noindent
\textbf{Vertices}:  Every vertex is a sphere-blind isotopy class of oriented, strongly separating surfaces embedded in $M$ transverse to $\mT^1$.  This includes two vertices representing the empty surface:  One in which $M$ is positive and one in which $M$ is negative.  \\  

\noindent
\textbf{Edges}: Given a vertex $v$ in the complex, choose a surface $S$ representing $v$.  Given a K-disk $D$ for $S$, let $S'$ be the surface that results from K-compressing $S$ across $D$.  Let $v'$ be the vertex representing $S'$.  For every K-disk $D$ for $S$ with one exception, we will include in $\mS(M, \mT^1)$ an edge connecting $v$ to $v'$.  The one exception is that we will not include edges representing tree disks on the negative side of $S$.  We will say that this edge is \textit{below} $v$.

If $D'$ is a K-disk on the same side of $S$ as $D$ and $D \cap S$ is isotopic to $D' \cap S$ relative to $\mT^1 \cap S$ then K-compressing $S$ across $D'$ produces a surface that is blind isotopic to $S'$.  (If $D$ and $D'$ are not isotopic in the complement of $S$, the surfaces may not be isotopic.  However, there will be a collection of separating spheres between along which we can sphere tube one the resulting surfaces to get the other.)  Thus we will include a single edge representing all the K-disks whose intersection with $S$ is isotopic to that of $D$.  Similarly, if $D$ and $D'$ are tree disks on the same side of $S$ whose boundaries bound disks $E, E' \subset S$ and a regular neighborhood of $E$ is isotopic in $S \setminus \mT^1$ to a regular neighborhood of $E'$ then tree compressing across $D$ or $D'$ produces the same surface $S'$.  We will similarly include a single edge representing all such tree disks.

Note that the construction depends on a choice of representative for $v$.  If we choose a different representative $S^*$ for $v$, then there is a blind isotopy from $S^*$ to $S$ that sends each K-disk for $S$ to a K-disk for $S^*$.  The blind isotopy class of the surface that results from K-compressing is preserved, so the construction of $\mS(M, \mT^1)$ does not depend on the representative we choose for $v$.

However, there is one subtle point: If there are two inequivalent isotopies from $S$ to $S^*$ then there will be two distinct ways to identify the edges determined by one representative with the edges determined by the other.  This does not affect the construction of $\mS(M, \mT^1)$, but we must be careful when we interpret edges later on. \\

\noindent
\textbf{2-cells}: Given a surface $S$ representing a vertex $v$, assume $e_1$ and $e_2$ are distinct edges below $v$.  Let $v_1$, $v_2$ be the second endpoints of $e_1$, $e_2$, respectively, represented by surfaces $S_1$, $S_2$.  The edges $e_1$, $e_2$ represent K-disks $D_1$, $D_2$ for $S$ and we will include faces based on how $D_1 \cap S$ and $D_2 \cap S$ sit relative to each other.  

First consider the case when $D_1$, $D_2$ are disjoint.  If $D_2$ is a bridge disk or tree disk for $S$ then it will be a bridge disk or tree disk for the surface $S_1$ that results from K-compressing $S$ across $D_1$.  If $D_2$ is a compressing disk for $S$ then its boundary will be contained in $S_1$ and it will be a compressing disk for $S_1$ if and only if its boundary is essential in $S_2$.  The disk $D_1$ has a similar relationship with the surface $S_2$ that results from K-compressing $S$ across $D_2$.

If $D_1$ and $D_2$ are K-disks for $S_2$, $S_1$, respectively then they define an edge $e'_1$ descending from $v_2$ and an edge $e'_2$ descending from $v_1$.  Performing the K-compressions in either order produces the same surface so the four edges $e_1$, $e_2$, $e'_1$, $e'_2$ form a loop as in the top left of Figure~\ref{facesfig}.  

In fact, the edges $e'_1$ and $e'_2$ are not uniquely determined because there may be more than one way to isotope the surface that results from compressing along $D_1$, or $D_2$ to the chosen representative for $S_1$ or $S_2$. So there are potentially infinitely many such sets of edges. We will include in $\mS(M, \mT^1)$ a face, called a \textit{diamond} bounded by each of these loops and say that such a diamond is \textit{below} $v$. Note that all these diamonds are in some sense topologically equivalent, and related by a symmetry of the surfaces involved.
\begin{figure}[htb]
  \begin{center}
  \includegraphics[width=4.5in]{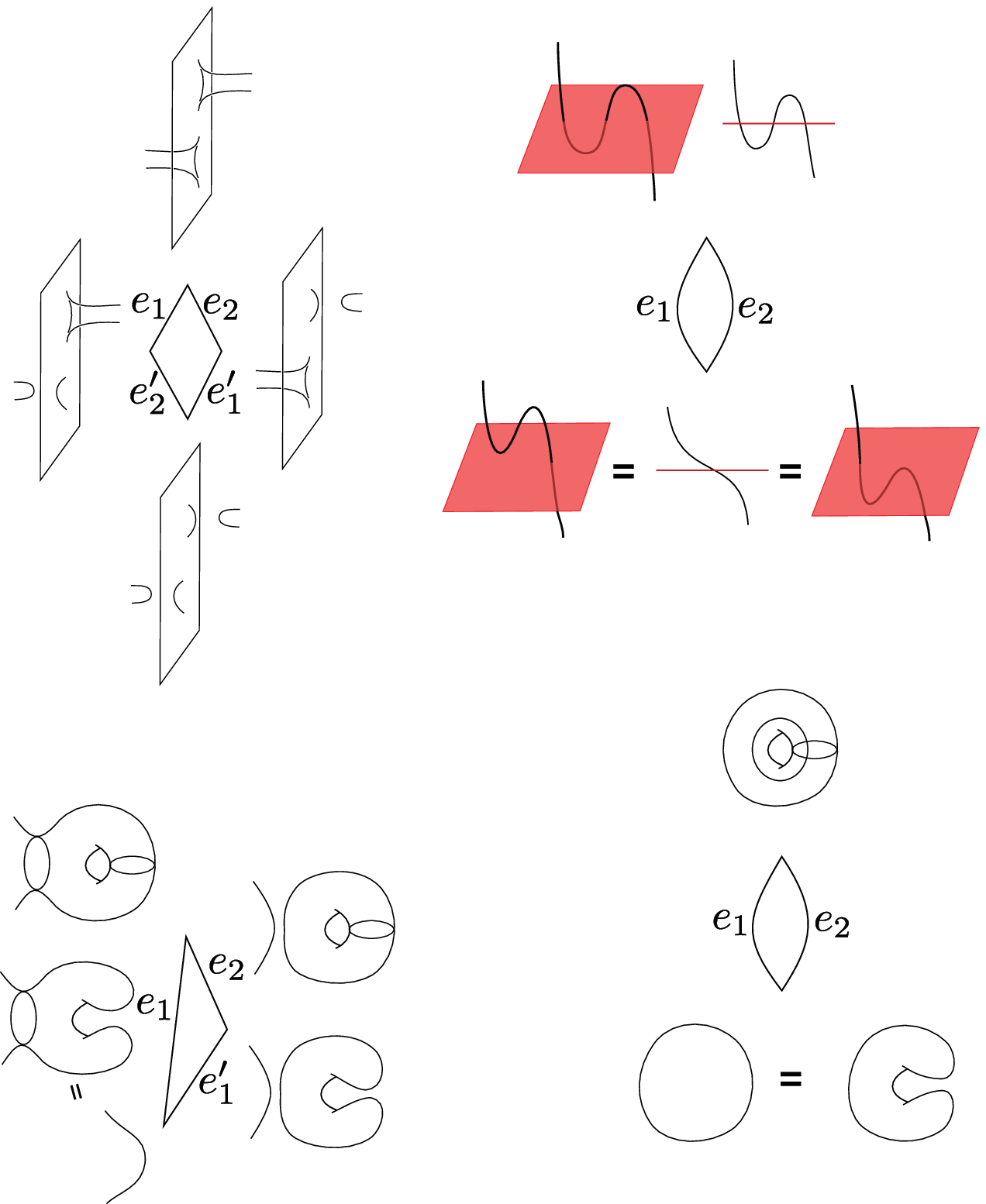}
  \caption{Faces of $\mS(M, \mT^1)$ are defined by pairs of K-compressions.}
  \label{facesfig}
  \end{center}
\end{figure}

There are two cases in which either $D_1$ or $D_2$ has trivial boundary in $S_2$ or $S_1$, respectively.  In the first case, $D_1$ and $D_2$ are disks on opposite sides of $S$ with parallel boundary.  Compressing $S$ across one or the other produces two surfaces which are not, in general, isotopic.  However, the two disks and the annulus between them comprise a sphere $R$.  We can turn the surface produced by compressing across $D_1$ into the surface produced by $D_2$ by tubing the surface to $R$ twice.  Thus the surfaces are blind isotopic so the edges determined by both $D_1$ and $D_2$ share a second endpoint and form a loop.  We will include a 2-cell in $\mS(M, \mT^1)$ bounded by this loop.  This face will be called a \textit{bigon} and we will say that this bigon is \textit{below} $v$.

In the second case, one of the disks, say $D_1$, is a compressing disk whose boundary separates from $S$ a planar subsurface containing $D_2 \cap S$ such that compressing $S$ across $D_1$ produces a surface $S_1$ in which $D_2$ is a K-disk for a sphere component.  Let $e'_2$ be the edge below $v_1$ defined by $D_2$.  K-compressing $S_1$ across $D_2$ produces a surface with a sphere component.  If this sphere component is trivial then the surface that results from removing this trivial sphere is isotopic to the surface $S_2$ that results from K-compressing $S$ across $D_2$.  If the sphere is non-trivial then we can sphere tube the surface to a parallel sphere to make the sphere component trivial, then remove it.  Thus the edges $e_1$, $e_2$, $e'_2$ form a loop, as in the bottom left of Figure~\ref{facesfig}.  As in the case of a diamond, there may be multiple choices for the edge $e'_2$ and we will include in $\mS(M, \mT^1)$ a face, called a \textit{triangle} bounded by every such loop.  We will say that this triangle is \textit{below} $v$.  

Next consider the case when $D_1$ and $D_2$ are bridge compressing disks that intersect in a single point of $\mT \cap S$, or compressing disks on opposite sides of $S$ that intersect in a point and are contained in a torus component of $S$ disjoint from $\mT^1$.  Both K-compressions will produce the same surface, so the edges $e_1$, $e_2$ will share a second vertex besides $v$.  In this case, we will include in $\mS(M, \mT)$ a face called a \textit{bigon}  bounded by the loop formed from these two edges, as on the right in Figure~\ref{facesfig}.  Again we will say that this bigon is \textit{below} $v$.  

Note that every face is defined by a pair of K-disks that either are disjoint or intersect in a single point.  In paricular, two compressing disks that intersect in a single point (sometimes call a stabilizing pair) will determine a bigon if and only if they are contained in a torus component.  For higher genus components, two compressing disks must be disjoint to determine a face below that vertex.  This is precisely the condition for when two vertices in the curve complex determine an edge.  We will see later that there is a strong connection between this definition and the curve complex (or, rather, the disk complex.)

Before defining the 3-cells in $\mS(M, \mT^1)$, we note the following property of the 2-cells:  Each 2-cell $c$ we have defined below a given vertex $v$ has two edges, $e_1$ and $e_2$ below $v$ ending at vertices $v_1$ and $v_2$.  Any remaining edges in the boundary of $c$ are below $v_1$ or $v_2$.  We can think of such a face as the result of taking a diamond and shrinking zero, one or both of the lower edges to a point.  The edge $e'_2$ below $v_1$ is defined by the image in $S_1$ of the K-disk $D_2$.  Thus we can think of the edge $e'_2$ as a projection of $e_2$ down to the vertex $v_1$.  For any pair of edges $e_1$, $e_2$ that determine a 2-cell in $\mS(M, \mT^1)$, we will define the \textit{projection of} $e_2$ \textit{across} $c$ as the edge of $c$ below $v_1$ if it exists.  Otherwise, we will say that the projection of $e_2$ is the vertex $v_1$.  Define the projection of $e_1$ across $c$ symmetrically.

\begin{Lem}
\label{surfaceprojectionlem}
Given edges $e_1$, $e_2$, $e_3$ below a vertex $v$ such that each pair of these edges are contained in a face below $v$, the union of the projections of any two edges across the third either is either a vertex, an edge or a pair of edges that define a face.
\end{Lem}

\begin{proof}
Two edges below a vertex $v$ determine a 2-cell below $v$ if and only if the boundaries of the corresponding K-disks are disjoint.  If we compress across a disk disjoint from both then the boundaries of the disks will be disjoint in the new surface and will either determine a face below the new vertex, will be parallel disks and thus determine the same edge, or one or both will be trivial and project to a vertex.
\end{proof}

\noindent
\textbf{3-cells}: 
We will define the 3-cells of $\mS(M, \mT^1)$ by using projections of edges to build cube-like cells.  Let $e_1$, $e_2$, $e_3$ be edges below a vertex $v$ such that each pair of edges define a 2-cell below $v$. Moreover, assume we have chosen a specific 2-cell for each pair of edges. By Lemma~\ref{surfaceprojectionlem}, each pair of edges project across the cell they cobound with the third edge $e_i$ to either a point, a single edge or a pair of edges that determine a face below $v_i$.  

Consider the set of 3-tuples with entries in $\{0,1\}$, which we will think of as the vertices of a cube.  Identify $v$ with $(1,1,1)$ and identify $v_1$, $v_2$, $v_3$ with $(0,1,1)$, $(1,0,1)$ and $(1,1,0)$, respectively.  Identify each face of the cube with one of the three faces below $v$ so that the three edges adjacent to $(1,1,1)$ is sent to the appropriate edge below $v$.  This identification is continuous, though it may send edges of the squares to vertices in $\mS(M, \mT^1)$.  

Let $e_{ij}$ be the projection of $e_i$ across the cell it cobounds with $e_j$. For each $i \neq j \neq k$, if $e_{ij}$ and $e_{kj}$ project to a single edge or a vertex, then we will let $c'_j$ be this edge or vertex. Otherwise, we may have a number of choices of 2-cell $c'_j$ below $v_j$. For each choice of 2-cell $c'_j$ below each $v_j$, let $e_{ijk}$ be the edge that results from projecting $e_i$ across the cell it cobounds with $e_j$, then across $c'_k$. Because compressing $S$ along $D_j$ and $D_k$ regardless of order, we can choose the 2-cells $c'_1$, $c'_2$, $c'_3$ so that $e_{ijk} = e_{ikj}$. (This choice is not unique.)

Identify each of the three remaining faces of the cube with $c'_1$, $c'_2$, $c'_3$ chosen above (including the case when on of these is the empty set or a single edge). Because we can map the boundary of the cube continuously into $\mS(M, \mT^1)$ in this way, the image in $\mS(M, \mT^1)$ is a sphere.  We will include in $\mS(M, \mT^1)$ a 3-cell bounded by these 2-cells for each compatible choice of $c'_j$s, and we will say that this 3-cell is \textit{below} $v$.

If the three K-disks that determine a 3-cell are pairwise disjoint and do not co-bound a planar subsurface of $S$ then all three of the faces are diamonds and each face will project to a diamond.  These six diamonds determine a cube as in Figure~\ref{3cellsfig}.  The figure also shows some of the other 3-cells that may arise.  Figure~\ref{odd3cellsfig} shows two situations that will produce the 3-cells on the far right of Figure~\ref{3cellsfig}.   Figure~\ref{odd3cells2fig} shows two situations that produce the the middle 3-cell on the bottom row of Figure~\ref{3cellsfig}.  The remaining two cases occur when two of the three edges determine a bigon or triangle, but the third edge determines a diamond with each of the first two.  \\
\begin{figure}[htb]
  \begin{center}
  \includegraphics[width=3.5in]{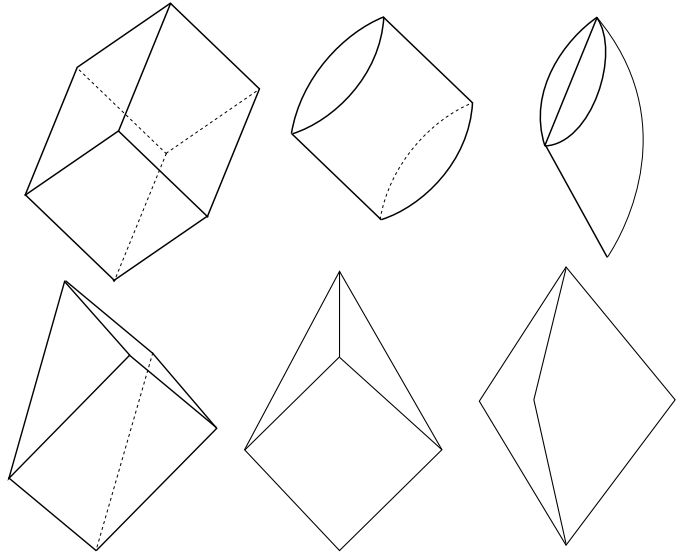}
  \caption{The three-dimensional cells of of $\mS(M, \mT^1)$ are defined by triples of edges descending from a given vertex.}
  \label{3cellsfig}
  \end{center}
\end{figure}

\begin{figure}[htb]
  \begin{center}
  \includegraphics[width=3.5in]{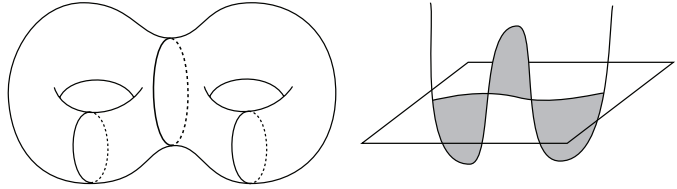}
  \caption{The disks in the two examples induce the 3-cells on the far right of Figure~\ref{3cellsfig}.}
  \label{odd3cellsfig}
  \end{center}
\end{figure} 

\begin{figure}[htb]
  \begin{center}
  \includegraphics[width=2.5in]{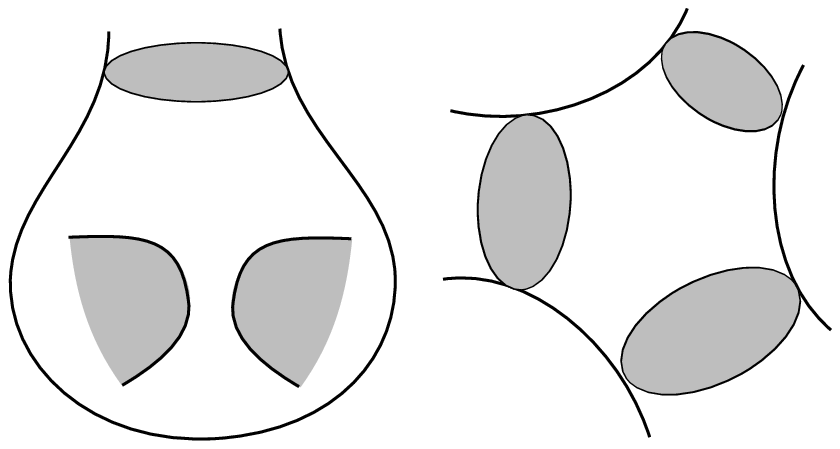}
  \caption{The disks in these examples cobound a planar surface and define the 3-cells in the middle and right of the bottom row of Figure~\ref{3cellsfig}.}
  \label{odd3cells2fig}
  \end{center}
\end{figure}

Note that there is a 3-cell for every triple of edges such that each pair determines a face.  This implies that given four edges below $v$ such that any three of them determine a 3-cell below $v$, any three edges will project across the fourth to a 3-cell, a 2-cell, an edge or a vertex.  Thus we can continue the construction for higher dimensional cells. \\

\noindent
\textbf{Higher-dimensional cells}:  The 3-cells of $\mS(M, \mT^1)$ are defined by mapping the boundary of a 3-cube into the 2-skeleton of $\mS(M, \mT^1)$, based on the projections defined by the 2-cells.  This generalizes directly to a construction for defining 4-cells based on the cells of dimension 3 and lower.  We can the continue by induction, defining the $n$-cells by mapping the boundaries of $n$-cubes into the $n-1$-skeleton of $\mS(M, \mT^1)$.  In particular, a set $e_1,\dots,e_n$ of edges below a vertex $v$ will determine an $n$-cell below $v$ if and only if every pair of edges determines a 2-cell below $v$.  In this paper, we will not make use of these higher dimensional cells, so we will not dwell on this construction any further here.  However, the reader should be able to reconstruct the details if desired.

\section{Height complexes}
\label{heightsect}

We will define thin position for any cell complex $\mS$ satisfying certain axioms and paired with certain additional information.  First, we require a complexity function from the vertices of $\mS$ to some partially ordered set.  (This is really a partial ordering on the vertices of $\mS$, but we find it useful to think in terms of a complexity function.)

To define such a complexity on the complex of surface $\mS(M, \mT^1)$, consider a (possibly disconnected) surface $S$ representing some vertex $v$ in $\mS(M,\mT^1)$.  The \textit{complexity} $c(S)$ will be the number of components of $S$ minus the Euler characteristic of $S \setminus \mT^1$, plus the number of sphere components.  Because the Euler characteristic of a sphere is two, adding or removing a sphere component from $S$ does not change its complexity.  A bridge compression or tree compression always increases the Euler characteristic.  A compression increases the Euler characteristic of $S \setminus \mT^1$ by at least two and increases the number of components by at most one, so any K-compression will reduce the complexity.  In particular, if an edge $e$ is below a vertex $v$ then the complexity of the second endpoint will be lower than that of $v$.  

In general, consider a cell complex $\mS$ with a complexity $c(S)$ into a partially ordered set. \\

\noindent
\textbf{The Morse axiom}: The complexities of the endpoints of any edge in $\mS$ are comparable and distinct.  Every 2-cell in $\mS$ is a diamond, a triangle or a bigon as in the construction of $\mS(M, \mT^1)$.  Given three edges such that any two bound a 2-cell, the projection of any two across the third will determine a face, an edge or a vertex.  Every $n$-cell is defined by mapping the boundary of an $n$-cube in the $n-1$-skeleton via projections. \\

This axiom implies that every cell in a height complex has a maximal vertex and a minimal vertex.  As in the definition of $\mS(M, \mT^1)$, if $v$ is the maximal vertex of a cell $C$ then we will say that $C$ is \textit{below} $v$. \\

\noindent
\textbf{The net axiom}: For any vertex $v \in \mS$, there is an integer $\ell(v)$ such that every edge path starting at $v$, along which the complexity strictly decreases, has length at most $\ell(v)$. \\

\begin{Def}
A cell complex satisfying the Morse axiom and the net axiom is called a \textit{height complex}.
\end{Def}

This terminology was chosen because we will think of the complexity function $c$ as a height function on $\mS$.  By construction, and by Lemma~\ref{surfaceprojectionlem}, the complex of surfaces satisfies the Morse axiom.  To check the second axiom, we need to be a little more careful.  Because the complexity of every vertex in $\mS(M, \mT^1)$ is a non-negative integer, the length of any decreasing path from a given vertex $v$ is at most $c(v)$, so $\mS(M, \mT^1)$ also satisfies the Net axiom.  (Note that later, we will define a height complex with a complexity that is not a non-negative integer, so we will need this axiom in its full generality.)

\section{Thin paths}
\label{thinpathsect}

In this section we consider paths in a height complex $\mS$.  We can take $\mS$ to be the complex of surfaces defined in the previous section, but the arguments will rely only on the axioms of a height complex, so the Lemmas we prove will apply in general.

Recall that the \textit{link} of a vertex $v$ in a simplicial complex $\mS$ is the boundary of a regular neighborhood of $v$.  (Here, we are thinking of the complex $\mS$ as a topological space.)  Every edge with an endpoint in $v$ determines a vertex in the link and the higher dimensional cells of $\mS$ determine a cell decomposition of $L_v$.  Recall that an edge $e$ is below one of its endpoints $v$ if other endpoint has complexity less than that of $v$.  

\begin{Def}
The \textit{descending link} $L_v$ of $v$ is the simplicial quotient of the subcomplex of the link spanned by the vertices corresponding to edges below $v$.
\end{Def}

By the \textit{simplicial quotient}, we mean the simplicial complex that results from identifying any two simplices in the link that have the same boundary. This is necessary because, for example, a pair of vertices in the link may be spanned by an infinite number of edges.

The descending link of a vertex is made up of the ``corners'' of the cells below that vertex.  Because each $n$-cell in a height complex $\mS$ is defined by $n$ edges below a given vertex of $\mS$, the descending link is a simplicial complex, i.e. each cell in the descending link is a convex hull of its vertices.  Recall that a \textit{flag complex} is a simplicial complex such that for any finite set of vertices in which each pair bounds an edge, the entire set bounds a simplex.

\begin{Lem}
The descending link of any vertex in a height complex is a flag complex.
\end{Lem}

\begin{proof}
Two vertices in $L_v$ bound an edge if and only if the corresponding edges in $\mS$ determine a face.  Given a set of vertices in $L_v$ such that each pair bounds an edge, each pair of the corresponding edges in $\mS$ will determine a 2-cell in $\mS$, so the Morse axiom implies that the entire set will determine a higher dimensional cell in $\mS$.  This cell determines a cell in $L_v$ bounded by the original set of vertices.
\end{proof}

\begin{Def}
The \textit{index} of a vertex $v \in \mS$ is $n+1$ where $\pi_n(L_v)$ is the first non-trivial homotopy group of the descending link.  If $D_v$ is empty then the index is zero.
\end{Def}

Note that the index of a vertex may be undefined.  If $\mS$ happens to be the complex of surfaces $\mS(M, \mT^1)$ defined above, then edges descending from a vertex $v$ correspond to K-disks for a surface $S$ representing $v$.  An index-zero vertex has no descending edge, so the corresponding surface has no K-disks, and is thus K-incompressible.  Moreover, in the case when $\mT^1 = \emptyset$, the descending link of $v$ is precisely the disk complex for $S$, and the index of $v$ is the topological index of $S$ as defined by Bachman~\cite{bachman}.

If we think of the complexity $c$ as a height function, then we can think of the vertices with well defined index as critical points of this function.  In particular, note that if $x$ is an index-$n$ critical point of a Morse function $f$ on a smooth manifold $N$ and $B$ is a small ball around $x$ then the set of points $\{y \in \partial B | f(y) < f(x)$ has exactly one non-trivial homotopy group, namely its dimension $n-1$ homotopy group.

\begin{Def}
We will say that a vertex $v$ is \textit{rigid} if it has a well defined index.  Otherwise, we will say that $v$ is \textit{floppy}.
\end{Def}

This terminology will be justified in the applications at the end of this paper.  We avoid the term ``topologically minimal'' introduced by Bachman because a rigid vertiex will often be the minimum or maximum of a path, which would make it a minimal maximum or a minimal minimum.

An \textit{edge path} is a sequence of directed edges $E = e_1,\dots,e_n$ in $\mS$ such that each edge shares its first endpoint with the previous edge and its second endpoint with the next edge.  The endpoints of the edges in $E$ determine a sequence of vertices in $\mS$.  We will label these vertices $v_0,\dots,v_n$ such that $v_i$ is the vertex shared by $e_i$ and $e_{i+1}$. The \textit{endpoints} of the path $E$ will be the first vertex of the first edge ($v_0$) and second vertex of the last edge ($v_n$).  Note that two vertices in $\mS$ may be connected by more than one edge, so we will always keep track of a path by its edges, rather than by its vertices.

A \textit{maximum} in an edge path $E$ is a vertex $v$ in the path such that $v$ is not an endpoint of the path and the edges before and after $v$ are both below $v$.  A \textit{minimum} is a vertex $v$ in the path such that $v$ is not an endpoint of the path and neither of the vertices before or after $v$ is below $v$.

We will define the \textit{complexity} $c(E)$ of a path $E$ to be the $k$-tuple of complexities of the maximal vertices in the path, ordered in a non-increasing fashion.  This determines an ordering on edge paths by applying lexicographic ordering to these $k$-tuples.  In other words, we compare the highest complexity maxima of each path, then the second highest and so on.  If the complexities of the first $i-1$ maxima of $E$ and $E'$ are the same, but $i$th maximum of $E$ is higher than that of the $i$th maximum of $E'$ then $c(E) > c(E')$.  If all the pairs of components match, then the paths will have the same complexity.

If $v_i$ is a maximum of a path $E$ then the edges $e_i$ and $e_{i+1}$ are below $v$ and define points in the link of $v$.  If these points are in the same component of the path link then we can define a new path that is thinner than $E$ as follows:  

Because the two points are in the same component of $L_v$, there is a sequence of edges connecting them in $L_v$.  These correspond to a sequence of faces $c_1,\dots,c_m$ below $v$ such that $c_1$ contains $e_i$ in its boundary, $c_m$ contains $e_{i+1}$ and each $c_i$ shares one edge below $v$ with $c_{i-1}$ and its other edge below $v$ with $c_{i+1}$.  Because $\mS$ satisfies the Morse axiom, every vertex in $c_i$ other than $v$ has lower complexity than $v$.  The lower edges of the 2-cells $\{c_i\}$ form a path whose endpoints are the same as the path consisting of $e_i, e_{i+1}$.  Let $E'$ be the path constructed from $E$ by removing $e_i$, $e_{i+1}$ and replacing them with this new path.  

All the vertices in this new sub-path have lower complexity than $v$ so all the new maxima in this sub-path are lower than the maximum $v$.  Thus while $E'$ may be longer than $E$, the lexicographic ordering implies that $c(E') < c(E)$.  We can think of this construction as sliding the path $E$ across the faces $c_1,\dots, c_m$.  We will say that $E'$ results from \textit{thinning} $E$.

\begin{Def}
A path $E$ between vertices $v$, $v'$ in $\mS$ will be called \textit{thin} if it cannot be thinned, i.e. it is not possible to create a lower-complexity path by the above construction. 
\end{Def}

Note that the term `thin' has traditionally referred to the lowest complexity position among all positions.  The sense in which it is used here would usually be called locally thin.  However, we will not be interested in globally thin paths, so we will drop the word `locally'.

\begin{Lem}
\label{unorientedthinninglem}
If $E = \{e_i\}$ is a thin path in $\mS$ then every maximum in the path has index one.  For any path $E$ between vertices $v$ and $v'$, there is a thin path that results from thinning $E$.
\end{Lem}

\begin{proof}
Let $E$ be a thin path and let $v$ be a maximum of $E$ between edges $e_i$ and $e_{i+1}$.  Because there are two edges below $v$, its descending link is non-empty and $v$ does not have index zero.  If $e_i$ and $e_{i+1}$ define vertices in the same component of the descending link of $v$ then by the above construction, $E$ can be thinned.  Since $E$ is thin, the descending link must disconnected, so $v$ has index exactly one.

If $E$ is not thin, let $E_0 = E$ and let $E_1$ be the result of thinning $E_0$.  By construction, each vertex in this new path is connected to a vertex in the old path by a strictly increasing path.  If $E_1$ is not thin then there is a path $E_2$ that results from thinning $E_1$ and so on.  If we continue this process, it either ends with a thin path or continues indefinitely, producing an infinite sequence of paths.

Let $G$ be the graph consisting of all the edges in the faces along which we slide to produce the (possibly infinite) sequence of paths $\{E_i\}$.  By induction on $i$, there is an increasing path contained in $G$ from any vertex $v'$ contained in a path $E_j$ to a vertex of $E$.  The net axiom implies that there is a bound on the length of any descending path from a vertex, so the graph $G$ has finite diameter.  Because we thinned along finitely many squares at each step, the graph $G$ has finite valence.  Every finite valence, finite diameter graph is finite, so $G$ must be finite.  This implies that the sequence of paths $\{E_i\}$ is finite and the process must terminate with a thin path $E_i$.  
\end{proof}

\section{K-compression bodies}
\label{ksplittingsect}

Let $N$ be the set $S \times [0,1]$ for a compact, closed, orientable (possibly disconnected) surface $S$.  Let $B_1,\dots,B_k$ be a (possibly empty) collection of balls, each parameterized as $D_i \times [0,1]$ for $D_i$ a disk.  Let $B'_1, \dots B'_{k'}$ be a second (possibly empty) collection of balls.  Identify each annulus $\partial D_i \times [0,1]$ with an annulus in $S \times \{0\}$ such that the annuli in $S \times \{0\}$ are pairwise disjoint.  Let $N'$ be the union of $N$ and $\bigcup B_i$, glued according to this identification.  Identify the boundaries of $B'_1,\dots, B'_{k'}$ with spheres in the boundary of $N'$ that are disjoint from $S' \times \{1\}$, then let $H$ be the union of $N'$ and $\bigcup B_i$, glued along these spheres.  A \textit{compression body} is a 3-manifold homeomorphic to a space $H$ that results from this construction.

The \textit{positive boundary}, $\partial_+ H$ of $H$ will be the image of $\partial N \setminus S \times \{1\}$.  The \textit{negative boundary}, $\partial_- H$ of $H$ is the set $\partial H \setminus \partial_+ H$.  Note that in this definition, we allow compression bodies to be disconnected, to have sphere components in their boundary and to have components homeomorphic to a surface cross an interval.

A \textit{K-compression body} is a pair $(H, K)$ where $H$ is a compression body and $K \subset H$ is a properly embedded graph such that each component of $K$ is either a vertical arc in $F \times [0,1]$ or is an arc or one-vertex tree that can be isotoped into the positive boundary $\partial_+ H$ while fixing its intersection with $\partial H$.  (By a \textit{properly embedded} graph, we mean a graph $K$ such that $K \cap \partial H$ consists of valence-one vertices of $K$.)

Given $M$ and $\mT^1$ as above, we will say that a compression body $H \subset M$ is a \textit{K-compression body (with respect to $\mT^1$)} if $K = H \cap \mT^1$ makes $(H, K)$ a K-compression body.

\begin{Lem}
\label{edgekbodylem}
Let $v$ and $v'$ be the endpoints of an edge $e$ in $\mS(M, \mT^1)$ such that $e$ is below $v$.  Then $e$ defines a K-compression body with respect to $\mT^1$ whose positive boundary is a representative for $v$ and whose negative boundary if a representative for $v'$.
\end{Lem}

\begin{proof}
Let $S$ be a surface representing $v$ and let $D$ be a K-disk representing $e$.  Without loss of generality, assume $D$ is on the positive side of $S$.  Let $S^*$ be the result of transversely isotoping a copy of each component of $S$ into the component of the positive complement to which it is adjacent. We will assume that the isotopy fixes $\mT^1$ setwise, but not pointwise.  We can extend this isotopy to an ambient isotopy of $(M, \mT^1)$ and let $D^*$ be the image of $D$ after the isotopy.  Then $D^*$ is a K-disk fir $S^*$.  Let $S'$ be the result of K-compressing $S^*$ along $D^*$.  

Because $S^*, D^*$ are isotopic to $S, D$, the surface $S'$ represents $v'$.  Between $S$ and $S^*$, there is a submanifold homeomorphic to $S \times [0,1]$ that intersects $\mT^1$ in a collection of vertical arcs.  We will parameterize it so that $S \times \{1\} = S$ and $S \times \{0\} = S^*$.  The submanifold $H$ of $M$ between $S$ and $S'$ is the union of this $S \times [0,1]$ and a ball $B$ that intersects $S \times [0,1]$ in one of three ways:

If $D$ is a compressing disk, then $B$ can be parameterized as $D \times [0,1]$ so that $B \cap (S \times [0,1]) = \partial D \times [0,1]$.  By definition, the union is a compression body.  Its intersection with $\mT^1$ is the same as that of $S \times [0,1]$, a collection of vertical arcs, so this $H$ is a K-compression body with respect to $\mT^1$ such as the top object in Figure~\ref{critsincompfig}.
\begin{figure}[htb]
  \begin{center}
  \includegraphics[width=2.5in]{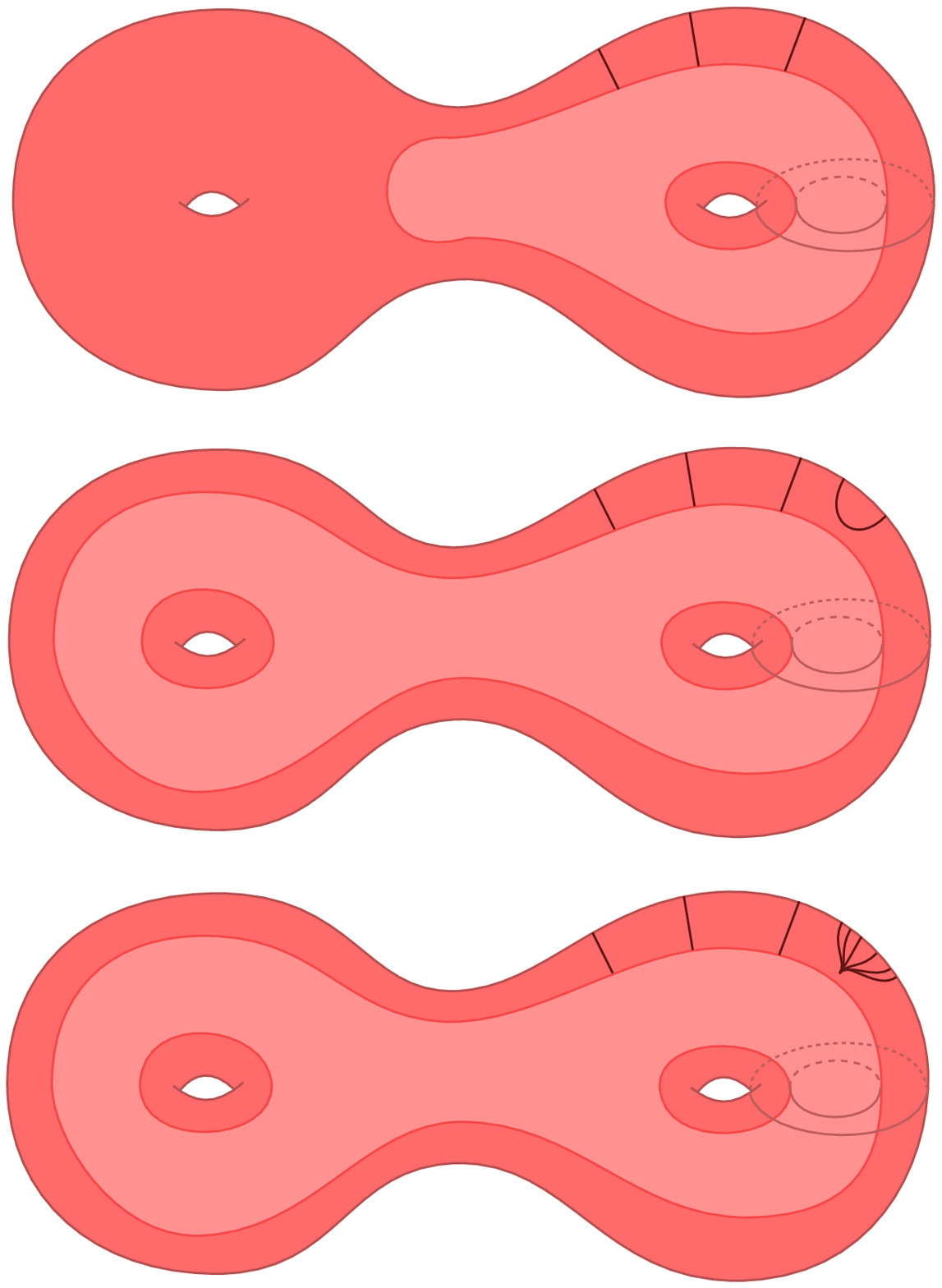}
  \caption{Edges of $\mS(M, \mT^1)$ define K-compression bodies.}
  \label{critsincompfig}
  \end{center}
\end{figure}

If $D$ is a bridge disk then $B$ intersects $S \times [0,1]$ in a disk in $\partial B$, so the union $H$ is itself homeomorphic to $S \times [0,1]$.  The ball $B$ intersects $\mT^1$ in a single unknotted arc whose endpoints connect to vertical arcs in $S \times [0,1]$.  Thus $H$ is a K-compression body with respect to $\mT^1$ with a single horizontal arc as in the middle of Figure~\ref{critsincompfig}.

If $D$ is a tree disk then $B$ intersects $S \times [0,1]$ in a disk in $\partial B$, so the union $H$ is again homeomorphic to $S \times [0,1]$.  The ball $B$ intersects $\mT^1$ in an unknotted tree whose endpoints connect to vertical arcs in $S \times [0,1]$.  Thus $H$ is a K-compression body with respect to $\mT^1$ with a single boundary parallel tree as shown at the bottom of Figure~\ref{critsincompfig}.
\end{proof}

\section{Oriented paths}
\label{orsect}

An edge path in $\mS(M,\mT^1)$ defines a sequence of K-compression bodies in $M$.  We would like to put these K-compression bodies together to form something like the generalized Heegaard splittings defined by Scharlemann and Thompson~\cite{sch:thin}.  To do this, we need to keep track of which side of each surface the next compression body is being added to.  For this purpose, we will add directions to the edges of our height complex.

Every edge $e$ in $\mS(M, \mT^1)$ descends from a vertex $v$ representing some surface $S$ and $e$ corresponds to a K-disk for $S$ or a sphere component for $S$.  Because $S$ is transversely oriented, each component of $M \setminus S$ is labeled positive or negative.  If the K-disk is on the positive side of $S$ then we will choose an orientation for $e$ pointing away from $v$.  Otherwise, we will have the edge $e$ point towards $v$.  With this definition, an edge always points towards the surface that is in the positive complement of the other surface.

In a height complex, by definition every face is either a diamond, a triangle or a bigon.  A triangle or bigon can be thought of as the result of taking a diamond and crushing one or both of the lower edges (i.e. those not adjacent to the maximum) down to a point.  We will say that a diamond is \textit{parallel oriented} if opposite edges point in the same direction.  We will say that a triangle or bigon is parallel-oriented if its orientation is induced by crushing one or both of the lower edges of a parallel oriented diamond.

\begin{Def}
In general, an \textit{oriented} height complex will be a height complex $\mS$ that satisfies the following axiom: \\
\end{Def}

\noindent
\textbf{The parallel orientation axiom}: For any 2-cell $q$ in $\mS$, the orientations on the edges of $q$ make it a parallel-oriented diamond, triangle or bigon. \\

Because the higher dimensional cells in a height complex are built from 2-cells, the parallel orientation axiom implies that the orientations on every $n$-dimensional cell will be induced from crushing edges of an oriented $n$-cube.  From the construction of the complex of surfaces, it is immediate that $\mS(M, \mT^1)$ satisfies the parallel orientation axiom.

\begin{Def}
We will say that an edge path $E = \{e_i\}$ in an oriented height complex is \textit{oriented} if the directions of the edges that come from the path and those that come from the complex agree.  The path $E$ will be \textit{reverse oriented} if the orientations coming from the path all disagree with those coming from the complex.
\end{Def}

In the complex of surfaces, the direction on each edge points towards the surface that is in the positive complement of the other, so each $v_{i+1}$ in an oriented path will represent a surface in the positive complement of the surface representing $v_i$.  In other words, an oriented path always moves in the same direction in $M$.  This implies that the vertices of any oriented path can be represented by pairwise disjoint surfaces.  In particular, its endpoints cobound a submanifold of $M$.  For a directed path that is monotonic in the complexity, one can say precisely what this submanifold looks like:

\begin{Lem}
\label{decpathkbodylem}
Let $E$ be a strictly decreasing oriented or reverse oriented path in $\mS(M, \mT^1)$ starting at a vertex $v$ and ending at a vertex $v'$.  Then $E$ determines a K-compression body whose positive boundary is a representative for $v$ and whose negative boundary consists of a representative for $v'$.
\end{Lem}

\begin{proof}
Without loss of generality, we will assume the path $E = e_1,\dots,e_n$ is oriented and we will induct on the number of edges in the path.  If $n = 1$ then the path determines a compression body by Lemma~\ref{edgekbodylem}.  

For $n > 1$, assume the Lemma is true for length $n-1$ paths.  By the induction hypothesis, there is a compression body $H$ determined by the path $e_1,\dots,e_{n-1}$.  The negative boundary component of $H$ represents the first endpoint of $e_1$ and the positive boundary $\partial_+ H$ corresponds to the last endpoint of $e_{n-1}$, which is the first endpoint of $e_n$.  

As in the proof of Lemma~\ref{edgekbodylem}, the surface representing the second endpoint of $e_n$ comes from $\partial_+ H$ by attaching a ball $B$ contained in the positive component of the complement.  The interior of $H$ is on the negative side of $\partial_+ H$, so $H$ and $B$ have disjoint interiors.  In each of the cases considered in Lemma~\ref{edgekbodylem}, the submanifold $H \cup B$ is a compression body that, along with its intersection with $\mT^1$ defines a K-compression body.
\end{proof}

\begin{Lem}
\label{alignkbodieslem}
An oriented path $E$ in $\mS(M, \mT^1)$ defines a sequence of K-compression bodies with pairwise disjoint interiors, one for each interval between consecutive minima and maxima, that coincide alternately along their positive and negative boundaries.
\end{Lem}

\begin{proof}
Assume the initial edge in the path is increasing.  (If it's not, we will skip to the second step in the proof.)  Let $v_0$ be the initial vertex in the path $E$ and let $v_1$ be the first maximum.  Because the initial edge is increasing, the segment of $E$ from $v_0$ to $v_1$ will be strictly increasing so Lemma~\ref{decpathkbodylem} implies that there is a compression body $H_1$ whose positive boundary is a representative for $v_1$ and whose negative boundary is a representative for $v_0$.  If any trivial sphere in $\partial_- H_1$ bounds a ball disjoint from $H_1$ then we will remove this sphere from $\partial_- H_1$ and redefine $H_1$ to include this ball.

Let $v_2$ be the first minimum in $E$ after $v_1$.  Again Lemma~\ref{decpathkbodylem} implies that there is a compression body $H_2$ whose positive boundary represents $v_1$ and whose negative boundary represents $v_2$.  Because $\partial_+ H_1$ and $\partial_+ H_2$ represent the same vertex $v_1$, they are related by isotopies and by tubing to spheres.  If we sphere tube $\partial_+ H_2$ to a sphere, we can extend the construction to the entire compression body as follows:

Let $S$ be the sphere to which we would like to add a tube.  This sphere is disjoint from $\partial_+ H$ by assumption.  Because the negative boundary of a compression body is incompressible inside the compression body, we can isotope $S$ disjoint from $\partial_- H$.  If $S$ is then contained in $H$, there is a sphere blind isotopy of $\partial_- H$ after which $S$ is parallel into $\partial_- H$, so we can assume that $S$ is completely disjoint from $H$.  

Let $\alpha$ be the arc defining the sphere tubing of $\partial_+ H$.  If both endpoints of $\alpha$ are in $\partial_+ H$ then $\alpha \cap \partial H_+$ will consist of two points and $\alpha \cap \partial_- H$ will consist of an even number of points.  If $\alpha$ approaches its endpoints from outside $H$, then we will extend $\alpha$ into $H$ so that its endpoints are in $\partial_- H$.  After possibly extending $\alpha$ in this way, $\alpha$ will define a sequence of sphere tubings on the surface $\partial_- H \cup \partial_+ H$, that take $H$ to a homeomorphic compression body and take $\partial_+ H$ to the surface we wanted.  If we start with only one endpoint of $\alpha$ in $\partial_+ H$, then after possibly extending this endpoint to $\partial_- H$, we can perform a similar sequence of sphere tubings on $\partial_- H \cup \partial_+ H$ to get the compression body and the surface that we want.

By repeating this process, we can arrange so that $\partial_+ H_2 = \partial_+ H_1$.  If a trivial sphere in $\partial_- H_2$ bounds a ball disjoint from $H_2$, we will add this ball into $H_2$.

Let $v_3$ be the first maximum after $v_2$ and let $H_3$ be a representative for the increasing path from $v_2$ to $v_3$.  By sphere tubing $H_3$ as we did with $H_2$, we can assume that the representative $\partial_- H_2$ for $v_2$coincides with the representative $\partial_- H_3$ for $v_2$.  We will also fill in any trivial sphere component of $\partial_- H_3$ that bounds a ball disjoint from $H_3$.  We can isotope the remaining trivial sphere components of $\partial_- H_2$ and $\partial_- H_3$ to be disjoint from $\partial H_3$ and $\partial H_2$, respectively.

If $\partial_- H_2$ contains a trivial sphere, then the ball $B$ bounded by this trivial sphere intersects $H_2$ and in fact contains a component of $H_2$.  If $B$ is disjoint from $H_3$ then we can isotope $H_3$ so that it contains this ball.  We will then remove $B$ from $H_3$ so that $\partial B$ is a component of $\partial_- H_3$.  If we perform this operation for every trivial component of $\partial_- H_2$ and a similar operation for every trivial component of $\partial_- H_3$, we will ensure that $\partial_- H_2 = \partial_- H_3$.

We can continue in this fashion for each consecutive maximum and minimum of the path $E$, to find the sequence of compression bodies promised by the Lemma.
\end{proof}

This sequence of K-compression bodies is analogous to the generalized Heegaard splittings defined in~\cite{sch:thin}.  In particular, if $\mT^1$ is the empty graph and the surfaces represented by $v^-$, $v^+$ are parallel to $\partial M$ then the compression bodies defined by an oriented path between $v^-$ and $v^+$  in $\mS(M,\emptyset)$ form a generalized Heegaard splitting.

\section{Thinning oriented paths}
\label{thinorsect}

We saw in Section~\ref{thinpathsect} that the local maxima of a thin path in an (unoriented) height complex $\mS$ has link index one.  However, if we want to restrict our attention to directed paths in an oriented height complex, this proof doesn't work; the construction in the original proof does not necessarily produce an oriented path.  In this section we define a more refined construction that always produces an oriented path.  

The parallel orientation axiom implies that the boundary of any 2-cell in $\mS$ contains two ``parallel'' paths;  Either, there are two paths from the top vertex to the bottom or there is a path across the top two edges, and a second path across the bottom edges.  Given a path $E$ and a 2-cell $f$ such that one of these paths in the boundary of $f$ is a sub-path of $E$, an \textit{oriented face slide} consists of replacing this subpath of $E$ with the other oriented path in the boundary of $f$.  Because of the parallel orientation axiom, the new path is also oriented.

We will differentiate between the two types of oriented face slides:  A slide that replaces the two edges adjacent to the maximum vertex with the complementary path of edges, or vice versa, will be called a \textit{vertical face slide}.  A face slide that replaces one path from a maximum to a minimum with the other will called a \textit{horizontal face slide}.

We will say that two oriented paths in $\mS$ are \textit{equivalent} if there is a sequence of horizontal face slides that turns one of the paths into the other.  Because a horizontal face slide does not create any maxima or minima, the set of maximal and minimal vertices in two equivalent paths are the same.  (We will see later that in the case when $\mS = \mS(M, \mT^1)$, if two splittings paths are equivalent then they determine the same K-splittings, i.e. the same collections of compression bodies in $M$.)

\begin{Def}
An oriented path will be called \textit{thinnable} if it is equivalent to an oriented path whose complexity can be reduced by a vertical face slide.  An oriented path is \textit{thin} if it is not thinnable.  
\end{Def}

Note that this is different from the definition of a thin path in an unoriented height complex.

\begin{Def}
Given a maximal vertex $v$ in an oriented path $E$, we will define $u^-_E$, $u^+_E$ to be the vertices of the descending link for $v$ determined by the edges in $E$ right before and right after $v$, respectively.  Define  The \textit{negative path link} $L^-_v$ to be the subset of the descending link spanned by the set of vertices $\{u^-_F\ |\ F$ is a path equivalent to $E\}$.  The \textit{positive path link} $L^+_v$ will be the subset spanned by $\{u^+_F\ |\ F$ is a path equivalent to $E\}$.  The subcomplex spanned by positive and negative path links will be called the \textit{path link} $L^E_v$.
\end{Def}

By definition, every vertex in the positive path link corresponds to an edge pointing away from $v$ and every vertex in the negative path link corresponds to an edge pointing towards $v$.  Since each edge points exactly one way, the positive and negative path links are disjoint.

Note that while the path link is contained in the descending link of a given vertex, it may be a proper subset of the descending link.  In particular, the homotopy type of the path link may be different from that of the descending link, so we need the following definition.

\begin{Def}
Given a maximum $v$ in a path $E$, the \textit{path index} of $v$ is the the smallest $i$ such that the $\pi_{i-1}(L^E_v)$ is non-trivial.\
\end{Def}

As with the previously defined index, the path index of a maximum may be undefined.  Vertices with path index one play a special role in thin position and correspond to strongly irreducible surfaces in Scharlemann-Thompson thin position.  In fact, the path link illustrates a subtlety in terminology that is commonly overlooked:  The term weakly incompressible is defined almost identically to strongly irreducible.  The difference is that (translated into index terminology) weakly incompressible means index one, while strongly irreducible means path index one.

If a maximum $v$ does not have path index one then there is an edge in $L^E_v$ connecting the positive and negative path links.  After picking an equivalent path, this edge will connect the vertices in $L^E_v$ corresponding to the edges before and after $v$ and will thus define a 2-cell in $\mS$ containing $e_i$ and $e_{i+1}$.  A vertical face slide across this 2-cell will replace $v_i$ with lower maxima, so we have the following:

\begin{Lem}
\label{thinpathsimaxlem}
A path $E$ in an oriented height complex $\mS$ will be thin if and only if every maximum in the path has path index one.
\end{Lem}

\section{The Casson-Gordon axiom}
\label{cgsect}

So far, we have defined a condition on the maxima of a path $E$ that determines when $E$ is thin.  In the setting of the complex of surfaces and thin position for 3-manifolds, this condition is equivalent to the statement that a generalized Heegaard splitting is thin if and only if every thick surface is strongly irreducible.  

\begin{Def}
We will say that a vertex $v \in \mS$ is \textit{compressible to the negative/positive side} if there is an edge below $v$ that points towards/away from $v$, respectively.  Otherwise, $v$ is \textit{incompressible to the negative/positive side}, respectively.
\end{Def}

In particular, an index-zero surface will be incompressible to both sides.  The second half of Scharlemann-Thompson's Theorem in~\cite{sch:thin} is that the thin surfaces of thin splittings are incompressible, i.e. index-zero.  This relies on a Lemma by Casson and Gordon~\cite{cass:red}.  In the general context, we will state this lemma as a slightly weaker, though more general, axiom: \\

\noindent
\textbf{The Casson-Gordon axiom}: Let $v$ be a maximum in an oriented path $E$, and let $v_-$, $v_+$ be the minima of $E$ right before and after $v$, respectively.  If $v_-$ is compressible to the positive side then either the path link of $v$ is contractible or $v_+$ is compressible to the positive side.  Similarly, if $v_+$ is compressible to the negative side then either the path link of $v$ is contractible or $v_-$ is compressible to the negative side.  \\

In particular, if the path link of a maximum is contractible then all its homotopy groups will be trivial, so the surface will be floppy.  This axiom allows us to control the minima of a path, based on the maxima as follows:

\begin{Lem}
\label{cgapplem}
Assume $\mS$ is an oriented height complex satisfying the Casson-Gordon axiom and let $E$ be an oriented path in $\mS$ such that the initial vertex $v_0$ is incompressible to the negative side, the final vertex $v_n$ is incompressible to the positive side and such that every maximum of $E$ has a well defined path index.  Then every minimum of $E$ has index zero.
\end{Lem}

\begin{proof}
Assume for contradiction that a minimum $v$ in the path does not have index zero.  Then there is an edge $e$ descending from $v$ and this edge either points towards or away from $v$.  Without loss of generality, we will assume this edge points away from $v$.  

Consider the subpath from $v$ to the maximum after $v$, to the following minimum.  Because the maximum has a well defined path index, its path link has a non-trivial $n$-dimensional homotopy group for some $n$.  The Casson-Gordon axiom thus implies that there must be an edge in $\mS$ descending away from the minimum after $v$.  Because there are finitely many minima in the original path, repeating this argument implies that the final vertex in the path has a edge descending away from it.  This contradicts the initial assumption, so every minimum must be index-zero/incompressible.
\end{proof}

In the context of generalized Heegaard splittings, the Casson-Gordon axiom can be interpreted as follows:  Given a thick surface $S$ between thin surfaces $S^-$ and $S^+$, if $S^-$ is compressible towards $S$ then either the compressing disk crashes through $S^-$, defining a compressing disk on the side away from $S$, or it defines a compressing disk for the sub-manifold between $S^-$ and $S^+$, in which case Casson and Gordon's Theorem implies that $S$ is weakly reducible (its path link is connected).  The conclusion from the axiom is more general and it implies the following corollary, which is equivalent to Scharlemann-Thompson's main result for generalized Heegaard splittings:

\begin{Coro}
\label{thinpathcharlem}
Assume $\mS$ is an oriented height complex satisfying the Casson-Gordon axiom and let $E$ be a thin oriented path in $\mS$ whose endpoints have index zero.  Then every maximum of $E$ is strongly irreducible and every minimum of $E$ is incompressible.
\end{Coro}

The next three sections are devoted to proving that $\mS(M, \mT^1)$ satisfies the Casson-Gordon axiom.

\section{Boundary reducible manifolds}
\label{kbodycplxsect}

In this section, we prove the following Lemma: 

\begin{Lem}
\label{cgsteponelem}
Let $M$ be a 3-manifold such that no component of $\partial M$ is a sphere disjoint from $\mT^1$.  If $S$ is an embedded, separating surface and $D$ is a sphere disjoint from $\mT^1$ or a K-disk for $\partial M$ that cannot be made disjoint from $S$ by sphere-blind isotopy then $S$ is floppy.
\end{Lem} 

Let $D$ be such a surface in $M$ and $S$ a surface representing a vertex $v \in \mS(M, \mT^1)$.  By isotoping and compressing the disk $D$, we can assume that $S \cap D$ consists entirely of essential loops and arcs in $S$.  The final disk may not be isotopic to $D$, but it will have the same boundary and will thus represent the same edge in $\mS$.  After this isotopy, we will say that $D$ is \textit{pulled tight} with respect to $S$.

\begin{Lem}
There is, up to blind isotopy, a unique way to pull $D$ tight.  After $D$ is pulled tight, the surface $D$ will either contain a K-disk for $S'$ or be disjoint from $S$.
\end{Lem}

\begin{proof}
Trivial loops of intersection can be removed in any order, as long as nested loops are removed from the inside out.  Switching the order in which two non-nested loops are removed does not change the final surface, so there is a unique way to pull any disk tight.

If $D$ is a compression disk or sphere component then the intersection contains an innermost loop in this component.  This loop is essential in $S$ because $S$ is pulled tight, so the disk in $D$ bounded by this loop is a compression disk for $S$.  If $D$ is a bridge disk then the intersection contains an innermost loop or an outermost arc.  As above, an innermost loop bounds a compression disk.  The endpoints of any arc are in $\mT^1$ so every outermost arc defines a bridge disk for $S$.  In a tree disk component of $S$, there is either an arc that determines a bridge disk or a collection of arcs surrounding the vertex that define a tree disk for $S$ contained in $D$.
\end{proof}

Because $D$ is an embedded K-disk or a sphere disjoint from $\mT^1$, any K-disks contained in $D$ have disjoint boundaries.  Thus the set of K-disks for $S$ in $D$ form a simplex in the descending link.

\begin{Def}
Given a representative $S$ for $v$, the \textit{visible simplex} for $S$ is the simplex of the descending link for $v$ defined by K-disks with boundary in $D \cap S$ after $D$ is pulled tight.
\end{Def}

We will fix a representative $S$ for $v$, and consider different representatives for $D$, each of which will determine a different visible simplex.

Let $B$ be a K-disk for $S$ contained in $H^-$ or $H^+$ and assume we have isotoped and compressed $B$ to remove all loops of intersection with $D$.  Because the interiors of $D$ and $B$ are disjoint from $\mT^1$, the endpoints of any arc in $D \cap B$ are contained in $S$.  

Let $c_1,\dots,c_n$ be the components of $D \cap B$, ordered so that each $c_i$ is an outermost arc with respect to the set of $c_j$ such that $j > i$.  In other words, choose an outermost arc from the set, then remove it and choose an outermost arc in the new collection and so on until we have exhausted the set.  We will call the ordered set $C = \{c_i\}$ a \textit{squeeze list} for $B$.

Given a simplex of the descending link for $v$, let $B_1,\dots,B_n$ be K-disks representing the vertices that define the simplex.  We can choose the disks so that they are pairwise disjoint and do not intersect $D$ in loops.  Let $C_i$ be a squeeze list for each $B_i$ and let $m_i$ be the number of components in each $C_i$.  

Consider the cube $I^n = [0,m_1] \times \dots \times [0,m_n]$ and let $T$ be the cell complex that results from subdividing $I^n$ into smaller cubes by cutting the $i$th axis into $m_i$ segments for each $i$.  Each vertex in $T$ can be described by a vector $(a_1,\dots,a_n)$ where each $a_i$ ranges from $0$ to $m_i$.  We will associate a K-disk $D_a$ with each vertex $a = (a_1,\dots,a_n)$ of the cube as follows:  Let $D_a$ be the result of isotoping $D$ across the first $a_1$ outermost disks in $B_1$, then across the first $a_2$ outermost disks in $B_2$ and so on for each $i$, then pulling the resulting surface tight with respect to $S$.  Because the disks $\{B_i\}$ are disjoint, this construction is well defined and does not depend on the order in which the isotopies are carried out.

For each vertex $a \in T$, consider the visible simplex in the descending link for $S$ defined by $D_a$.  Choose an arbitrary ordering of the vertices of the descending link of $S$ and define a map $\phi$ from the cube into the descending link by sending each vertex of the cube to the lowest ordered vertex of the visible simplex for $D_a$.

\begin{Lem}
\label{cubetolinkmaplem}
The map $\phi$ extends to a continuous map from a triangulation of $I^n$ into the descending link $L_v$.
\end{Lem}

\begin{proof}
Choose a transverse orientation for $D$, i.e. choose one component of $N(D) \setminus D$ to be positive and the other negative, where $N(D)$ is a regular neighborhood in $M$ of $D$.  This determines a transverse orientation for each disk $D_a$.  We will use this orientation to define directions on the edges of $T$.

Consider a cell $q$ in $T$ and let $e$ be an edge in this cube between vertices $a$, $b$.  This edge corresponds to an isotopy from $D_a$ to $D_b$, possibly followed by one or more compressions, such that during the isotopy there is a single tangency with $D$.  This isotopy is determined by a disk $B'_a$ whose boundary consist of an arc in $S$ and an arc in $D_a$.  The reverse isotopy is determined by a second such disk $B'_b$ with boundary in $S$ and $D_b$.  If $B'_a$ is on the positive side of $D_a$ then $B'_b$ will be on the negative side of $D_b$ and vice versa.  If $B'_a$ is on the positive side then we will have the edge point towards $a$.  Otherwise, the edge will point towards $b$.

Orient every edge of $q$ in this way.  Parallel edges represent the same disks and thus point in the same direction.  This implies that there will be a single vertex $a_-$ and a single vertex $a_+$ in $q$ such that every edge adjacent to $a_-$ points away from it and every edge adjacent to $a_+$ points towards it.  

Any directed path $a_- = a_0,a_1,\dots,a_n = a_+$ in the boundary of $q$ determines a sequence of disks $D_{a_0},\dots,D_{a_n}$ isotopic to $D$.  Because the disks defining the isotopies are pairwise disjoint and on the same side of $D_{a_i}$, the disks will be pairwise disjoint.  The intersection of $D_{a_i}$ with $S$ contains the boundary of $\phi(u_i)$, so the points $\{\phi(u_i)\}$ determine a simplex in the descending like of $S$.

Triangulate the cube $q$ so that each such a path determines a simplex in this triangulation.  (This is a triangulation because every point in $q$ will be contained in the convex hull of one of these paths.)  By the above argument, $\phi$ extends to a simplicial map from this triangulation of $q$ into the descending link.  If we consider any face of $q$ as a cell on its own and apply this construction, the resulting triangulation of the face will be the same as that induced by the triangulation of $q$.  Thus the triangulation of $q$ will coincide with the triangulation of any adjacent cell along their intersection and we can extend $\phi$ to every cell of $T$ in this way.
\end{proof}

Consider a simplicial map from a triangulated $n$-sphere $S^n$ into the descending link for $S$.  For each face $\sigma$ of $S^n$, let $\phi_\sigma$ be the map defined above for the image in the descending link of $\sigma$.  Every face of $\sigma$ defines a face of the cube associated to $\sigma$, so we can identify faces of the cubes.  The resulting cell complex is homeomorphic to a cone over $S^n$, i.e. an $n$-ball and the maps $\{\phi_\sigma\}$ define a map from this ball into the descending link for $S$.  We will use this map to prove Lemma~\ref{cgsteponelem}.

\begin{proof}[Proof of Lemma~\ref{cgsteponelem}]
As in the statement of the Lemma, let $S$ be a surface in $M$ and assume the sphere or disk $D$ cannot be made disjoint from $S$ by sphere-blind isotopies.  We will show that every map $\xi : S^n \rightarrow L_v$ from a triangulated $n$-sphere into the descending link for $v$ is homotopy trivial.  

Begin by choosing a representative $B_u$ and a squeeze list for each vertex of $S^n$ and assume that adjacent vertices are represented by disjoint disks. Consider a simplex $\sigma$ spanned by vertices $\{u_1,\dots, u_k\}$.  Let $B_1,\dots,B_k$ be the K-disks representing $u_1,\dots,u_k$ and let $m_1,\dots,m_k$ be the number of arcs in their respective squeeze lists.  Let $q_\sigma$ be the rectangle $[0,m_1] \times \dots \times [0, m_k]$, cut into square cells with integer vertices, then triangulated as in Lemma~\ref{cubetolinkmaplem}.  For each vertex $(x_1,\dots, x_k) \in q_\sigma$, if $x_i = m_i$, we will attach an edge from $(x_1,\dots, x_k)$ to the vertex $v_i$.  Let $\kappa_i$ be flag complex defined by the union of $\sigma$, $q_\sigma$ and all these edges.  (In other words, we fill in every length-three loop of edges with a triangle, and so on for each dimension.)  The reader can check that this complex is a cone over $\sigma$. 

For each face $\tau$ of $\sigma$, there is a natural inclusion map from the cone $\kappa_\tau$ into $\kappa_\sigma$.  In particular, this map sends the face of $q_\tau$ into the face of $q_\sigma$ in with the coordinates corresponding to the vertices of $\tau \subset \sigma$ are zero.  The union of all the cones for the simplices of $S^n$, under this identification, forms a cone over $S^n$, i.e. a ball $B^{n+1}$.  We will call this ball the \textit{Bachman ball} for $S^n$.  (A very similar ball was constructed by Bachman in~\cite{bach:gordon}.  

Because $D$ cannot be made disjoint from $S$, there is a map $\Phi_\sigma : q_\sigma \rightarrow L_v$ promised by Lemma~\ref{cubetolinkmaplem}.  Because each vertex $u$ of $q_\sigma$ with $x_i = m_i$ corresponds to a representative for $D$ that is disjoint from $B_i$, the image $\Phi(u) \in L_v$ is a K-disk disjoint from $B_i$.  Thus $\Phi_\sigma$ extends to a map $\Phi_\sigma : \kappa_\sigma \rightarrow L_v$ that agrees with $\xi$ on $\sigma$.  

By construction, the map $\Phi_\sigma$ agrees with $\Phi_\tau$ for each subface $\tau$ of $\sigma$.  Thus the collection of maps $\{\Phi_\sigma\}$ determines a map $\Phi : B^{n+1} \rightarrow L_v$.  This map from the Bachman ball to the descending link will be called the \textit{Bachman map}.  The Bachman map agrees with $\xi$ on $S^n = \partial B^{n+1}$, so the map $\xi$ is homotopy trivial.  Since $S^n$ was arbitrary, every homotopy group of $L_v$ is trivial, so $S$ is floppy.
\end{proof}

\section{Disks in compression bodies}
\label{diskssect}

In order to prove that the complex of surfaces satisfies the Casson-Gordon axiom, we will need to know that equivalent paths in $\mS(M, \mT^1)$ correspond to the same K-compression body.  First note that if a surface $S$ in a compression body $H$ that results from K-compressing the positive boundary of $H$ and $S$ can be made disjoint from every K-disk for $H$ then $S$ must be parallel to the negative boundary of $H$.  This and Lemma~\ref{cgsteponelem} imply the following:

\begin{Lem}
\label{pathsincbodycomplexlem}
If $S$ is a surface in a compression body $H$ that results from K-compressing $\partial_+ H$ some number of times then either $S$ is floppy or $S$ consists of trivial spheres and a surface parallel to $\partial_- H$.
\end{Lem}

Lemma~\ref{pathsincbodycomplexlem} can be interpreted as saying that every path in the complex of surfaces for a K-compression that descends from a surface parallel to the positive boundary either ends at the negative boundary or at a surface that does not have well defined index.  This is relevant in light of the following Lemma:

\begin{Lem}
\label{ifnoindexonelem}
Let $\mS$ be an oriented height complex, $v_0$ an index-zero vertex in $\mS$ and $v_1$ a second vertex such that there is no decreasing path from $v_1$ to an index-one vertex in $\mS$. Then any two oriented decreasing paths from $v_1$ to $v_0$ are equivalent.  Similarly, any two reverse oriented decreasing paths are equivalent.
\end{Lem}

\begin{proof}
Let $v_0$ be an index-zero vertex in $\mS$ and let $v_1$ be another vertex such that there are directed paths $E$, $E'$ from $v_0$ to $v_1$.  Construct an unoriented path $E^*$ by concatenating the path $E$ with the reverse of the path $E'$.  This path $E^*$ is a loop that starts and ends at $v_0$.  By Lemma~\ref{unorientedthinninglem}, this path can be thinned to a path in which every maximum has index one.  Every maximum in a path resulting from thinning $E$ is connected to $v_1$ by a decreasing path and by assumption no such vertices exist.  Thus $E^*$ can be thinned to a loop from $v_0$ to itself containing no maxima in its interior.  Because $v_0$ has index zero, the only such loop is the trivial loop from $v_0$ to itself, containing no edges.

The union $F$ of the faces along which we slide $E^*$ while thinning it is an immersed disk.  By construction, every edge in $F$ will be oriented to point towards its lower complexity endpoint.  The boundary of this disk is the union of the paths $E$ and $E'$.  The first edge of the path $E$ is a lower edge of a face in $F$ and the last edge is the upper edge of a face in $F$.  Thus either some consecutive pair of edges in $E$ are in the same diamond in $F$, or some edge in $E$ is the single edge on the side of a triangle or bigon.  In either case, the diamond, triangle or bigon defines a horizontal face slide that turns $E$ into a new path $E_1$.  There is an immersed disk with one fewer faces than $F$ whose boundary is $E_1 \cup E'$, and we can repeat the argument to find another horizontal edge slide.  Because there are finitely many faces in $F$, we will find a sequence of horizontal face slides that turn $E$ into $E'$.  An almost identical proof works in the case of two decreasing paths.
\end{proof}

We will say that two K-compression bodies with positive boundary $v$ are \textit{isotopic} if there is an isotopy that sends one compression body onto the other.  Because surfaces are only determined up to blind isotopy, a monotonic oriented path does not determine a unique isotopy class of K-compression bodies; the compression body will depend on the initial choice of representative for its positive boundary.  Once this surface is fixed, the negative boundary is only determined up to sphere-blind isotopies.  Thus we will say that two compression bodies with isotopic positive boundaries are \textit{sphere-blind isotopic} if they are related by sphere-blind isotopies of their negative boundaries.

\begin{Lem}
\label{equivalentcompbodylem}
Consider two strictly decreasing oriented or reverse oriented paths from $v$ to $v'$, defining K-compression bodies $H$, $H'$ whose positive boundaries are isotopic.  Then $E$ and $E'$ will be equivalent if and only if $H$ and $H'$ are sphere-blind isotopic.
\end{Lem}

\begin{proof}
First we will show that two equivalent, oriented decreasing paths from $v$ to $v'$ determine equivalent compression bodies in $M$.  If we can show that this is true when the paths are related by a single horizontal face slide, then the full result will follow by induction.  Let $q$ be the face along which we are sliding, $w$ its maximal vertex and $w'$ its minimal vertex.  

The edges of $q$ adjacent to $w$ determine disjoint K-disks $D_1$, $D_2$ for the surface $S$ corresponding to $w$.  If we compress along both disks simultaneously, we get a surface $S'$ representing $w'$.  Because both edges point away from $v$, the K-disks are on the same side of $S$, so the compressions determine a compression body $H$  between $S$ and $S'$.  The two edge paths in $q$ from $v$ to $v'$ correspond to compressing along $D_1$, $D_2$ in different orders, but for either order, the union of the two compression bodies determined by the two edges, is isotopic to $H$, as shown in Figure~\ref{horizontalfig}.
\begin{figure}[htb]
  \begin{center}
  \includegraphics[width=3.5in]{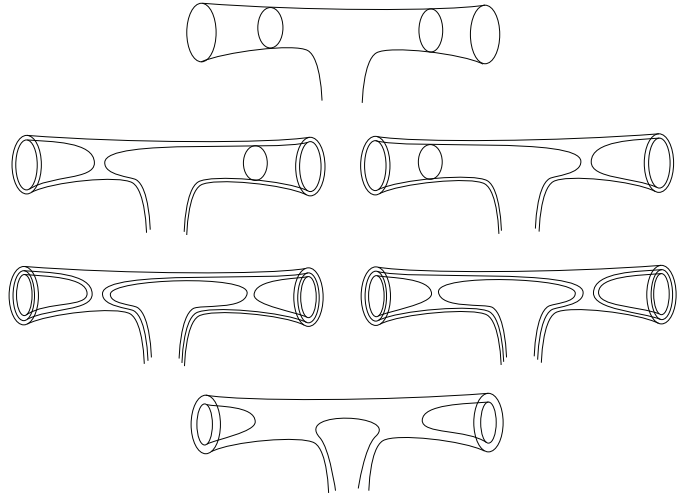}
  \caption{The two pairs of edges involved in a horizontal slide determine the same K-compression body.}
  \label{horizontalfig}
  \end{center}
\end{figure}

Conversely, if two decreasing paths determine sphere-blind isotopic compression bodies then we can assume they determine the same compression body $H$.  The two paths determine paths in the complex of surfaces for this $H$ starting at its negative boundary and ending at its positive boundary.  Since the paths have the same endpoints, Lemma~\ref{ifnoindexonelem} implies that they are equivalent in the complex of surfaces for $H$.  Every face in the complex of surface for $H$ corresponds to a face in the complex of surface for $M$, so the two paths are also equivalent in the complex of surfaces for $M$.  An identical proof works in the case of two increasing paths.
\end{proof}

\section{The Casson-Gordon axiom for surfaces}
\label{pathlinksect}

\begin{Lem}
\label{cgforsurfaceslem}
The complex of surfaces satisfies the Casson-Gordon axiom.
\end{Lem}

\begin{proof}
Let $v$ be a local maximum of an oriented edge path $E = \{e_i\}$ and let $v^-$, $v^+$ be the minima before and after $v$, respectively.  Let $H^-$ be the compression body determined by the decreasing path from $v$ to $v^-$ and let $H^+$ be the compression body determined by the path to $v^+$.  

Let $e$ be an edge below $v_-$, corresponding to a K-disk $D$ for $S^-$.  This K-disk is independent of the path $E$, so it may intersect $S$ and $S^+$.  If $D \cap S^+$ contains a loop that is trivial in $S^+$ then by compressing $D$ at this loop, we can replace $D$ with a K-disk for $S^-$ that intersects $S^+$ in fewer curves.  Thus we can assume $D \cap S^+$ contains no loops that are trivial in $S^+$.

If, after minimizing the intersection, $D \cap S^+$ is non-empty then there is a loop or arc of $D \cap S^+$ bounding a disk $D' \subset D$ whose interior is disjoint from $S^+$ and $S^-$.  Because $\partial D'$ is essential in $S^+ \setminus \mT^1$, it is a K-disk for $S^+$.  If $D'$ is on the positive side of $S^+$ then it determines a positive edge descending from $S^+$, satisfying the second possible conclusion of the axiom.  Otherwise, it determines a K-disk in $H^- \cup H^+$.  This case is equivalent to having $D$ disjoint from $S^+$ (but with $S^+$ and $S^-$ interchanged) so we will focus, without loss of generality on the case when $D$ is disjoint from $S^+$.

Let $M^*$ be the result of removing $H^- \cup H^+$ from $M$ and capping off any sphere components in the boundary that are disjoint from $\mT^1$.  Every K-disk for $S \subset M$ determines a K-disk for $S \subset M^*$.  By Lemma~\ref{equivalentcompbodylem}, we can choose a path that starts with any such disk, so the descending link for $S \subset M^*$ is canonically isomorphic to the path link for $S \subset M$.  Because the boundary of $M^*$ is K-compressible, Lemma~\ref{cgsteponelem}, implies that every homotopy group of its descending link is trivial.  Thus every homotopy group of the path link $L^E_v$ is trivial. Since $L^E_v$ is a simplicial complex with all its homotopy groups trivial, $L^E_v$ is contractible.
\end{proof}

\section{Iterated thin position}
\label{iteratedsect}

In this section, we will construct a complex whose vertices are equivalence classes of paths in an oriented height complex $\mS$.  There is an obvious choice for what the edges of such a complex will be; they will correspond to vertical slides in $\mS$.  Since we always consider face slides that fix the endpoints of the given path, we will now fix a pair of endpoints and restrict our attention to paths between those endpoints.

\begin{Def}
Given an oriented height complex $\mS$ and a specified pair of index zero vertices $v^-$, $v^+$, a \textit{splitting path} is an oriented path in $\mS$ from $v^-$ to $v^+$.
\end{Def}

For the complex of surfaces $\mS(M, \mT^1)$, $v^-$ and $v^+$ can be chosen in a fairly straightforward fashion.  If $M$ is a closed 3-manifold then $v^-$ and $v^+$ will be the two vertices representing empty surfaces.  There are two ways to orient the empty surface:  The orientation in which $M$ is labeled $+$ will define the vertex $v^-$ and the orientation in which $M$ is labeled $-$ will represent $v^+$.  (This may seem like a terrible convention, but if we did it the other way, there would be no oriented paths from $v^-$ to $v^+$.)  

If $M$ has non-trivial boundary then there are multiple choices for $v^-$ and $v^+$.  Let $\partial_- M$ and $\partial_+ M$ be disjoint subsets of $\partial M$ such that every component of $\partial M$ is contained in one set or the other.  Let $v^-$ be represented by a surface $S^-$ parallel to $\partial_- M$ such that the side of $S^-$ not containing $\partial_- M$ is labeled $+$.  Let $v^+$ be represented by a surface $S^+$ parallel to $\partial_+ M$ such that the side of $S^+$ not containing $\partial_+ M$ is labeled $-$.  (Again, this convention is chosen so that oriented paths exist.)

Given an oriented height complex $\mS$ with distinguished vertices $v_-, v_+$, the \textit{path complex} $\mP(\mS, v_-, v_+) = \mP(\mS)$ is a cell complex in which vertices are equivalence classes of splitting paths in $\mS$.  To define the cells of $\mP(\mS)$, we will define the edges and 2-cells, then define the higher dimensional cells by projecting edges across 2-cells as in Section~\ref{surfacecomplessect}.  \\

\noindent
\textbf{Edges:}
If $E$ is a splitting path in $\mS$ and $E'$ is the result of a vertical slide on $E$ then $E$ and $E'$ will determine distinct vertices of $\mP(\mS)$ and we will include in $\mP(\mS)$ an edge connecting these vertices.  \\

\noindent
\textbf{2-cells:}
There will be two types of 2-cells in $\mP(\mS)$.  First consider a path $E$ that has two distinct maxima that can be reduced by vertical slides.  Each vertical slide defines an edge in $\mP(\mS)$, and the resulting path still has a maximum that can be reduced by the other vertical slide.  Performing the vertical slides in either order produces the same path, so the two vertical slides determine four edges in $\mP(\mS)$ which form a closed loop in $\mP(\mS)$.  We will include in $\mP(\mS)$ a 2-cell bounded by this loop.

For the second type of 2-cell, recall that by the Morse axiom, every $n$-cell in $\mS$ is the quotient of an oriented $n$-cube $q$ under a map that crushes some edges and faces to points.  The oriented cube $q$ contains vertices $u_-$, $u_+$ such that every edge adjacent to $u_-$ points away from it and every edge adjacent to $u_+$ points towards it.  If $E$ is a path in $\mS$ such that a subpath of $E$, contained in $q$ passes from $u_-$ to $u_+$ then we can create a set of paths by replacing this subpath by one of the other subpaths.  Thus $E$ and $q$ determine a set of points in $\mP(\mS)$.  (Note that some of these paths may determine the same point in $\mP(\mS)$ because of the equivalence relation.)

Consider a 3-cube $q$ in $\mS$.  If no edges of $q$ are crushed then there are two possibilities: If $q$ is below $u_-$ or $u_+$ then any two paths in $q$ from $u_-$ and $u_+$ will be related by horizontal slides and will all define the same vertex in $\mP(\mS)$.  Otherwise, the cube will define a loop in $\mP(\mS)$ consisting of four vertical slides and two horizontal slides as in Figure~\ref{cellcycle1fig}.  Because the horizontal slides do not change the equivalence class of the path, the cube defines a diamond in $\mP(\mS)$.  If the $q$ is a crushed cube then it will still define a path, but some of the vertical slides may become trivial.  Thus $q$ will still determine a triangle or a bigon.  For each 3-cell $q$ and each path $E$ containing a subpath in $q$ from $u_-$ to $u_+$, we will include in $\mP(\mS)$ a 2-cell bounded by this path.  \\
\begin{figure}[htb]
  \begin{center}
  \includegraphics[width=3.5in]{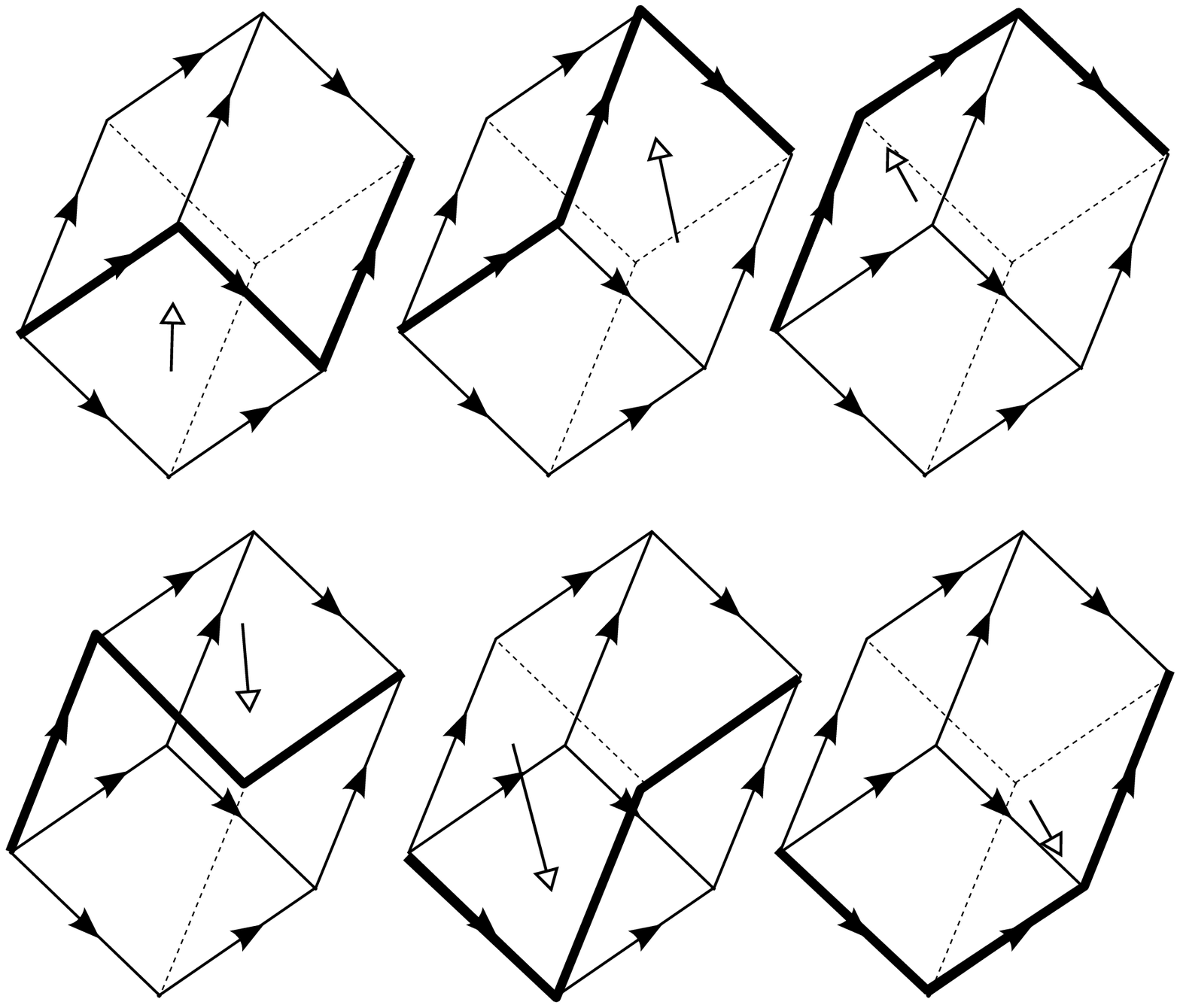}
  \caption{Faces in $\mP(\mS(M, \mT^1))$ defined by 3-cells in $\mS$.}
  \label{cellcycle1fig}
  \end{center}
\end{figure}

\noindent
\textbf{Higher-dimensional cells:}
Define the higher-dimensional cells by projecting the 2-cells as in Section~\ref{surfacecomplessect}.

\begin{Lem}
If $\mS$ is an oriented height complex then the path complex $\mP(\mS)$ satisfies the Morse axiom.
\end{Lem}

\begin{proof}
We noted during the construction of the 2-cells that every face is a diamond, triangle or bigon.  The higher dimensional cells are defined by projecting 2-cells, so to prove that $\mP(\mS)$ satisfies the Morse axiom, we need only check that given three edges such that any two bound a 2-cell, the projection of any two across the third will determine a face, an edge or a vertex. 

Let $e_1,e_2,e_3$ be edges in $\mP(\mS)$ such that any two of them determine a face.  We will show, without loss of generality, that the projections of $e_2$ and $e_3$ across $e_1$ determine a vertex, an edge or a face in $\mP(\mS)$.  Each of the three edges determines a vertical slide of a path $E$.  If the vertical slides represented by $e_2$ and $e_3$ occur at different maxima from each other then they will still occur at different maxima after the vertical slide defined by $e_1$.  If both of these vertical slides are at the same maximum, but the vertical slide corresponding to $e_1$ is at a different maximum of $E$ then the first two will not be affected by the vertical slide corresponding to $e_1$.  

Thus the only non-trivial case is when all three vertical slides occur at the same maximum $v$ of $E$. Each vertical slide corresponds to an edge in the path link of $v$ with one endpoint in the positive path link and the other in the negative path link.  Each pair of edges is contained in a triangle in the path link and the union of the three triangles is a surface with boundary.  Because the path link of $v$ is a simplicial complex, each boundary component of this surface has at least two edges.  Since there are nine edges in the three triangles and at least three pairs are identified, the surface must have a single boundary component containing three edges.  This surface cannot be a Mobius band because each internal edge has one endpoint in the positive path link and the other in the negative path link.  Thus the surface must be a disk.  

The boundary of this disk in the path link consists of three edges, so there is a fourth triangle whose union with this disk forms the boundary of a tetrahedron.  This tetrahedron has one vertex in the positive link of $v$ and three in the negative line, or vice versa.  Assume, without loss of generality, that the three vertices are in the negative path link.  Let $u_i$ be the vertex in the negative path link corresponding to $e_i$.  Let $v'$ be the lower vertex of the edge in $L_v$ corresponding to $u_i$.  

After a vertical slide corresponding to $e_i$, the new path will have a maximum at $v'$.  Because $e_1$ and $e_i$ determine a 2-cell in $\mP$ for $i = 2,3$, each of the arcs corresponding to $e_2$, $e_3$ projects across $e_1$ to an arc in the path link of $v$.  Because $\mS$ satisfies the Morse axiom and the initial two arcs are contained in a triangle in $L_v$, their projections determine a face in $L_{v'}$.  Thus the edges $e_2$ and $e_3$ project to either a vertex, a single edge or a pair of edges that determine a face.
\end{proof}

\begin{Lem}
The path complex $\mP(\mS)$ satisfies the Net Axiom.
\end{Lem}

\begin{proof}
To show that the path complex satisfies the Net Axiom, we will assign to each splitting path $E$ in $\mP$ an integer $n$ as follows: Along each successive edge in $E$, the complexity of the vertex either increases or decreases.  We will write $n$ in base-two so that the $i$th digit from the left is a one if the $i$th edge in $E$ increases and a zero of the $i$th edge decreases.  Since the path $E$ is finite, the number $n$ is finite.  Horizontal slides do not change the number associated to a path, but vertical slides do.  In particular, any vertical slide that reduces the complexity of the path replaces a pair of digits $10$ with something.  By the Morse axiom, this new string has length at most two.  If it is shorter than two then the associated number decreases.  If it is exactly two then the digits $10$ are replaced by $01$, and again the number decreases.  Since we must end with a positive number, the number of vertical that can be done to reduce the complexity of a splitting path is bounded above by the number associated to it.
\end{proof}

Because the path complex satisfies the Morse axiom and the net axiom, it is by definition a height complex:

\begin{Coro}
If $\mS$ is an oriented height complex then the path complex $\mP(\mS)$ is a (non-oriented) height complex.
\end{Coro}

\section{Projecting between complexes}
\label{projectionsect}

The path complex is not oriented, so to understand paths in $\mP(\mS)$, we return to the unoriented thin position defined in Section~\ref{thinpathsect}.  By Lemma~\ref{unorientedthinninglem}, every path in $\mP$ can be thinned to produce a path $E$ in which the maxima have index one.  Because $\mP$ is not an oriented height complex, this does not immediately imply that the minima are incompressible.  However, each minimum is connected to an index-zero vertex by a decreasing path.  We can insert into $E$ a path from each minimum down to an index zero vertex, then back to the original minimum.  The resulting path will have index-one maxima and index-zero minima.

A vertex in $\mP$ has index zero if no path in the equivalence class it represents can be thinned.  In other words, every path in $\mS$ represented by an index-zero vertex in $\mP$ is thin.  We would like to characterize paths in $\mS$ represented by index-one vertices $\mP$.

\begin{Lem}
\label{criticalpathonlyonewrvertexlem}
If $E = \{e_i\}$ is a splitting path representing a vertex $w$ in $\mP$ with index one then all but one of the maxima of $E$, are strongly irreducible.
\end{Lem}

\begin{proof}
An edge below $w$ in $\mP$ corresponds to a 2-cell $q$ in $\mS$ whose maximum vertex $v$ is a maximum of $E$ and such that, after a sequence of horizontal slides, the edges before and after $v$ are edges of $q$.  If $v'$ is a second maximum of $E$ and $v'$ is weakly reducible then $E$ is also equivalent to a path in which the edges before and after $v'$ are in a common 2-cell.  Turning the first of these paths into the second involves horizontal slides of edge that are in the segments of the path strictly descending from $v'$, so they do not affect the segments descending from $v$.  Thus $E$ is equivalent to a path in which the edges before and after $v$ and $v'$, respectively, simultaneously bound disjoint faces.

By this argument, every edge of $\mP$ corresponding to a vertical slide of $E$ across a square below a different maximum of $E$ is connected in the descending link of $w$ to the edge corresponding to the slide across $q$.  If there are such squares below two distinct maxima of $E$ then the descending link of $w$ is connected (in fact, it has diameter at most two) so $w$ does not have link index one.
\end{proof}

\begin{Lem}
\label{critpathcritmaxlem}
If $E$ is a splitting path representing a vertex in $\mP$ with link index one then the one weakly reducible maximum $v$ of $E$ has path index two.
\end{Lem}

\begin{proof}
Let $w$ be the vertex of $\mP$ representing $E$ and let $v$ be the single maximum in $E$ that does not have path-index one.  We will show that if the path link of $v$ is simply connected then the descending link of $w$ is connected.

Let $f$, $f'$ be any two edges descending from $w$.  By assumption, these correspond to faces $q$, $q'$ in $\mS$ below $v$.  Each face below $v$ determines an arc in the descending link of $v$, from a vertex in the positive path link to a vertex in the negative path link.  Let $p$, $n$ be the vertices in the path link defined by $q$ and let $p'$, $n'$ be the vertices determined by $q'$.  (We will let $p$ and $p'$ be in the positive path link and $n$, $n'$ in the negative path link.)  Because $p$ and $p'$ come from equivalent paths in $\mS$, there is a path in the positive path link from $p$ to $p'$.  Similarly, there is a path in the negative edge link from $n$ to $n'$.  These two paths, along with the edges between $p$, $n$ and $p'$, $n'$ define a closed edge path $\ell$ in the path link of $v$.  

If $L^E_v$ is simply connected then there is simplicial map from a triangulated disk $D$ into $L^E_v$ whose boundary is sent onto $\ell$.  Let $D_+$ be the set of vertices in $D$ that are sent into the positive path link of $v$.  Let $D_-$ be the set of vertices that are sent into the negative path link.  These sets are disjoint and every vertex in $D$ is in one set or the other.  Exactly two edges in $\partial D$ have endpoints with opposite signs (the edges between $p$, $n$ and $p'$, $n'$).  Every triangle has three vertices, so these are either all in the same set or two are in one set and the third is in the other.  

For each of the latter type of triangles, consider an arc between the endpoints of the edges whose vertices are in different sets.  Each endpoint of such an arc will connect to the endpoints of other arcs in adjacent triangles, so these arcs form a properly embedded, one-dimensional set in $D$.  It has exactly two endpoints (in the edges between $p$, $n$ and $p'$, $n'$) so there is a single arc in this set.  This arc defines a sequence of triangles, each with exactly one vertex in $D_+$ or one vertex in $D_-$, such that the first contains the edge between $p$, $n$ and the last contains the edge between $p'$, $n'$.

Let $\tau_1,\dots, \tau_k$ be this sequence of triangles in the path link and let $\alpha_0,\dots,\alpha_k$ be the edges spanning the positive and negative path links such that $\tau_i$ contains $\alpha_{i-1}$ and $\alpha_i$ for each $i$.  Each arc $\alpha_i$ determines a face in $\mS$ along which there is a vertical slide, so it determines an edge $e_i$ in $\mP$ descending from $w$.  Because $\alpha_{i-1}$ and $\alpha_i$ are edges of the same triangle $\tau_i$, these edges come from a 3-cell in $\mS$.  This 3-cell determines a face in $\mP$ below $w$ containing $e_{i-1}$ and $e_i$ in its boundary.  Thus the points corresponding to $e_{i-1}$, $e_i$ in the descending link of $w$ are connected by an arc.  This is true for each $i$, so there is a path in the descending link of $w$ connecting $e = e_0$ to $e' = e_k$.  Because this is true for any pair of edges $e$, $e'$ below $w$, the descending link is connected.
\end{proof}

This completes the characterization of the maximal vertices of any path in $\mS$ that represents an index-one vertex in $\mP(\mS)$.  To characterize the minima of such paths, we will employ the Casson-Gordon axiom.

\begin{Lem}
\label{cg1lem}
If $\mS$ satisfies the Casson-Gordon axiom and $E$ is a splitting path representing an index-one vertex in $\mP$ then every minimum in the path $E$ has index zero.
\end{Lem}

\begin{proof}
If $E$ is a splitting path corresponding to an index-one vertex in $\mP(\mS)$ then by Lemma~\ref{critpathcritmaxlem}, every maximum of $E$ has index one or two.  Because every maximum has well defined path index and the endpoints have index zero, Lemma~\ref{cgapplem} implies that the minima of $E$ are incompressible.
\end{proof}

\section{Links and Path links}
\label{linksect}

We have described the properties of thin and index-one paths in $\mS$ in terms of their path links.  The path link is a subset of the descending link of the vertex, but in many cases we will be interested in the whole descending link.  In particular, this brings up a subtle distinction in the terminology of Heegaard splittings that is often misunderstood as duplicate terminology.  A thick surface in a generalized Heegaard splitting is called \textit{weakly reducible} if there are disks in the adjacent compression bodies whose boundaries are disjoint, i.e. when the path link is connected.  A general two-sided surface is called \textit{strongly compressible} if it admits compressing disks on opposite sides whose boundaries are disjoint, i.e. its descending link is connected.  Because the path link for a generalized Heegaard splitting is not necessarily equal to its descending link (in particular when the adjacent thin surfaces are compressible) this is an important distinction.

In this section, we introduce an axiom that will help to compare the path link and the entire descending link. \\

\noindent
\textbf{The barrier axiom}:  Given any vertex $v \in \mS$, there are vertices $v_-$, $v_+$ and paths $E_-$, $E_+$ starting in $v_-$ and ending in $v_+$, respectively such that the following hold: Any oriented path descending from $v$ can be extended to a decreasing path ending in $v_+$ that is equivalent to $E_+$.  Any reverse oriented path decreasing from $v$ can be extended to a decreasing path ending at $v_-$ that is equivalent to $E_-$. \\

\begin{Def}
For any vertex $v$, the vertices $v_-$ and $v_+$ guaranteed by this axiom will be called \textit{barrier vertices}.
\end{Def}

\begin{Lem}
Assume $\mS$ satisfies the barrier axiom.  If $v$ is a maximum in a path $E$ such that the minima before and after $v$ are index-zero then the path link of $v$ is equal to its descending link.
\end{Lem}

\begin{proof}
Let $v_-$ and $v_+$ be the vertices before and after $v$, respectively.  Because these vertices are index-zero, the descending paths from $v$ to $v_\pm$ cannot be extended further.  Since any descending path can be extended to the barrier vertices of $v$, the vertices $v_-$ and $v_+$ must be the barrier vertices.  Any edge $e$ below $v$ determines an oriented, length one path, so this can be extended to a path from $v_-$ or a path to $v_+$ that is equivalent to the original.  Thus every edge below $v$ is in the path link of $v$.
\end{proof}

\begin{Coro}
\label{linkconnectedcoro}
If $\mS$ satisfies the barrier axiom then the positive and negative links of every vertex are connected.
\end{Coro}

\begin{proof}
Let $v$ be a vertex in $\mS$, let $v_-$, $v_+$ be its barrier vertices and let $E$ be a directed path from $v_-$ to $v_+$ to $v$.  Because $v_-$ and $v_+$ are barrier vertices, the path link of $v$ is equal to its link, so the positive and negative path links are equal to the positive and negative links, respectively.  As noted above, the positive and negative path links are connect, so this implies that the positive and negative links are connected.
\end{proof}

To show that the complex of surfaces satisfies the barrier axiom, we will need the following lemma:

\begin{Lem}
\label{incompsincompbodieslem}
Every K-incompressible surface in a K-compression body $H$ is boundary parallel or empty (modulo blind isotopy).
\end{Lem}

The proof of this Lemma follows almost immediately from the fact that there is a collection of K-disks that cut $H$ into a manifold homeomorphic to $\partial_- H \times [0,1]$.  Every K-incompressible surface is disjoint from all these K-disks.  The only incompressible surfaces in $\partial_- H \times [0,1]$ are boundary parallel, so the only incompressible surfaces in $H$ must be parallel to the negative boundary.

\begin{Coro}
\label{barrieraxiomlem}
The complex of surfaces $\mS(M, \mT^1)$ satisfies the barrier axiom.
\end{Coro}

\begin{proof}
Let $v$ be a vertex in $\mS(M, \mT^1)$ representing a surface $S \subset M$.  Let $E$ be a (possibly empty) directed path descending from $v$ such that $E$ cannot be extended.  Such a path exists by the net axiom.  By lemma~\ref{decpathkbodylem}, the path $E$ determines a K-compression body $H$ in $M$.  Because $E$ is maximal, the negative boundary $\partial_- H$ is K-incompressible to the positive side.  Because this surface is the negative boundary of a K-compression body, $\partial_- H$ is K-incompressible to both sides in the complement of $S$.

Let $E' = \{e'_1,\dots,e'_k\}$ be a second oriented path descending from $v$, and assume (via the net axiom) that we have extended $E'$ as far as possible.  Let $D$ be the K-disk on the positive side of $S$ representing $e'_1$.  Because $\partial_- H$ is incompressible in the complement of $S$, every loop of $D \cap \partial_- H$ must be trivial in both surfaces, so $D$ can be blindly isotoped into $H$.  Thus the compression body represented by the path $\{e'_1\}$ can be isotoped into $H$.  Repeating this argument for each edge in the path $E'$ implies that the compression body represented by $E'$ can be isotoped into $H$.  

After this isotopy, $\partial_- H'$ is contained in $H$.  By Lemma~\ref{incompsincompbodieslem}, every surface in a K-compression body is either K-compressible or boundary parallel.  The surface $\partial_- H'$ is K-incompressible on the negative side because it cobounds a compression body with $S$, and it is K-incompressible on the positive side because we have extended it as far as possible.  Thus $\partial_- H'$ must be parallel to $\partial_- H$, so $H'$ is transversely isotopic to $H$.  The surfaces $\partial_- H'$ and $\partial_+ H$ represent the same vertex of $\mS(M, \mT^1)$ so the paths $E$, $E'$ end at the same vertex $v_+$.  By Lemma~\ref{equivalentcompbodylem}, the paths $E$ and $E'$ are equivalent.  

Because the path $E'$ was arbitrary, any path descending from $v$ can be extended to a path ending at $v_+$ that is equivalent to $E$.  This, along with an almost identical argument for reverse directed paths implies that $\mS(M, \mT^1)$ satisfies the barrier axiom.
\end{proof}

Before we end the section, here is a somewhat surprising result about index-one vertices:

\begin{Coro}
\label{directedpathcoro}
If $\mS$ satisfies the Casson-Gordon axiom and the barrier axiom then every descending path from an index-one vertex in $\mS$ is either directed or reverse directed.
\end{Coro}

\begin{proof}
Let $v$ be an index-one vertex in $\mS$ and assume for contradiction there is a descending path $E$ from $v$ that is neither directed nor reverse directed.  Assume without loss of generality that the first edge in $E$ points away from $v$ and let $E'$ be the largest directed subpath of $E$ containing this initial edge.  Because $E$ is not directed, there is an edge $e$ in $E$ after the last vertex $v'$ in $E'$.  

Let $v_-$, $v_+$ be barrier vertices for $v$ and let $E''$ be an increasing directed path from $v_-$ to $v$.  Let $E^*$ be the concatenation of $E''$ and $E'$.  The edge $E$ descends from the last vertex $v'$ of the path $E$ and points towards it.  Because the initial vertex $v_-$ has index-zero, the Casson-Gordon axiom implies that $v$ does not have a well defined path index.  In particular, its path link is connected so there is an arc from its positive path link to its negative path link.  By Corollary~\ref{linkconnectedcoro}, the positive and negative descending links of $v$ are connected, so this implies that the entire link is connected, contradicting the assumption that $v$ has index one.
\end{proof}

\section{Heegaard paths}
\label{splittingpathsect}

In the context of Scharlemann-Thompson thin position, a Heegaard splitting is a generalized Heegaard splitting with a single thick surface.  This notion can be translated directly into the axiomatic setting:

\begin{Def}
A \textit{Heegaard path} in $\mS$ is a splitting path with a single maximum.
\end{Def}

We would like to turn a splitting path into a Heegaard path by ``undoing'' a sequence of thinning moves.  To do this, we need the following axiom: \\

\noindent
\textbf{The translation axiom}: Let $q^+$ be an $n$-cell in $\mS$ such that the  edges of $q^+$ all point away from $v$ and let $q^-$ be an $m$-cell in $\mS$ such that the edges of $q^-$ all point towards $v$.  Then there is an $(n+m)$-cell $C$ isomorphic to $q^+ \times q^-$ such that $q^+ = q^+ \times \{v\}$ and $q^- = \{v\} \times q^+$. Moreover, $C$ is unique up to automorphisms of $\mS$ fixing the vertices of the complex. \\

\begin{Lem}
The complex of surfaces satisfies the translation axiom.
\end{Lem}

\begin{proof}
Let $q^+$ be an $n$-cell in $\mS$ such that the  edges of $q^+$ all point away from $v$ and let $q^-$ be an $m$-cell in $\mS$ such that the edges of $q^-$ all point towards $v$.  Let $v^+$, and $v^-$ be the vertices such that $q^+$ and $q^-$ are below $v^+$ and $v^-$, respectively.  Let $S^+$ represent $v^+$, let $S^-$ represent $v^-$ and let $S$ represent $V$.

Because there is a descending path from $v^+$ to $v$, we can isotope $S^+$ so that it coincides with $S$ away from a collection of disks in $S$ and such that between $S$ and $S'$, there is a collection of balls on the negative side of $S$ such that each is either disjoint from $\mT^1$ or intersects $\mT^1$ in an unknotted arc or a regular neighborhood of a vertex.  We can isotope $S^+$ to intersect $S$ similarly, so that between $S^+$ and $S^-$, there are a collection of balls.  Because any two embedded disks in a surface $S$ are ambient isotopic, we can assume that the disk(s) where $S^+$ misses $S$ are disjoint from the disk(s) where $S^-$ misses $S$.  

Let $S^*$ be the surface that results from taking the union of $S'$ and $S''$, then removing the collection of disks where they miss $S$.  Let $v^*$ be the corresponding vertex in $\mS(M, \mT^1)$.  The resulting set is a surface and its isotopy class is uniquely determined.  Such as surface is shown in Figure~\ref{amalgfig}.  The balls between $S^+$ and $S$ and between $S^-$ and $S$ determine K-compressions for $S^*$ and the corresponding edges descending from $v$ determine an $n$-cell $C$ in $\mS$, containing $q^+$ and $q^-$, whose minimum vertex is $v$.
\begin{figure}[htb]
  \begin{center}
  \includegraphics[width=2.5in]{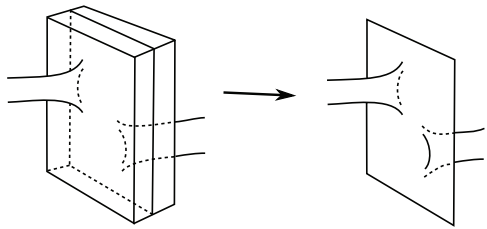}
  \caption{The surfaces $S^+$, $S$ and $S^-$, shown on the left, determine a surface $S^*$ on the right that can be compressed to produce any of the originals.}
  \label{amalgfig}
  \end{center}
\end{figure}

To see that $C$ is unique up to automorphisms of $\mS$ fixing the vertices, let $C'$ be a second $n$-cell containing $q^+$ and $q^-$. There are descending paths in $C'$ from the maximum vertex $v'$ in $C'$ to $v^-$ and $v^+$ so $v'$ is represented by a Heegaard surface $S'$ for the submanifold bounded by $S^-$ and $S^+$. 

There is a series of compressions on the positive side of $S'$ that turn $S'$ into a surface parallel to $S^-$. Every K-disk on the positive side of $S^-$ will intersect this parallel surface in a loop (in the case of a compressing disk), an arc (in the case of a bridge disk) or a collection of arcs (for a tree disk). In particular, this is the case for the collection of K-disks that define $q^-$. This implies that $S'$ also intersects these K-disks in loops and arcs. A similar argument shows that $S'$ intersects the K-disks defining $q^+$ in loops and arcs.

In the submanifold bounded by $S^-$ and $S^+$, the complement of the collection of K-disks defining $q^-$ and $q^+$ is homeomorphic to $S \times (0,1)$. The intersection of $S$ and $S'$ with the complement are unknotted and intersect the boundary in the same pattern, so $S$ and $S'$ are isotopic. This implies that the maximal vertices $v$ and $v'$ of $C$ and $C'$ are the same vertex. There may be different choices of isotopy from $S$ to $S'$, but each isotopy determines an automorphism of $\mS$ that permutes the descending link of $v$, but fixes the vertices and takes $C'$ to $C$.
\end{proof}

\begin{Lem}
\label{heegaardfromsplittinglem}
If $\mS$ is an oriented height complex satisfying the translation axiom  then every splitting path in $\mS$ is the result of thinning a Heegaard path in $\mS$.
\end{Lem}

\begin{proof}
The translation axiom implies that given two edges $e^+$, $e^-$ with lower endpoint $v$, one pointing towards $v$ and the other away from $v$, there is a unique diamond $q$ in $\mS$ whose minimum vertex is $v$ and such that the two edges adjacent to $v$ in $q$ are the edges $e^+$, $e^-$.  A vertical slide across $q$ is the opposite of a thinning move. 

Given a splitting path $\{e_i\}$, let $v_j$ be the first maximum in the path and let $v_k$ be the first minimum after $v_j$.  The edges $e_k$, $e_{k+1}$ before and after $v_k$ are above $v_k$.  One points towards $v_k$ and the other away so by the translation axiom, there is a square $q_1$ such that the edges of $q_1$ adjacent to its minimum are $e_k$ and $e_{k+1}$, as in Figure~\ref{simplefanfig}.
\begin{figure}[htb]
  \begin{center}
  \includegraphics[width=2.5in]{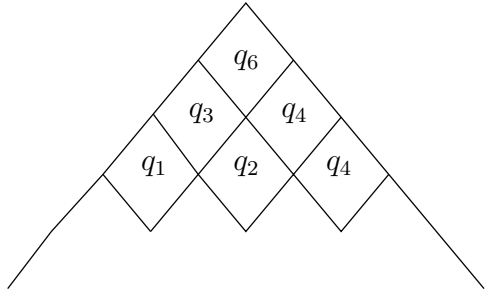}
  \put(-130,45){$q_1$}
  \put(-95,45){$q_2$}
  \put(-112,65){$q_3$}
  \put(-60,45){$q_4$}
  \put(-77,65){$q_4$}
  \put(-95,85){$q_6$}
  \caption{Every splitting path can be turned into a Heegaard path by a sequence of vertical face slides.}
  \label{simplefanfig}
  \end{center}
\end{figure}

If we slide these two edges up across $q_1$, the vertex $v_{k-1}$ will become a minimum between the edge $e_{k-1}$ and the edge of $q_1$ opposite $v_{k+1}$.  By the translation axiom, there is a diamond $q_2$ containing these edges and we can push the path up again.  We can continue in this fashion until we get to the edge $e_{j+1}$.  After we push this edge up across a diamond, the vertex $v_j$ will not be a maximum of the new path, but the next vertex in the new path will be a maximum.  Thus in the new path, the first maximum occurs one step later than in the original path, while the number of edges is the same.

We can continue pushing the minima up in this fashion.  After each step, the number of edges has stayed the same, but the index of the first maximum has increased by one.  Thus the process must eventually terminate, and this can only happen when there are no minima in the interior of the path, so the final path is a Heegaard path.
\end{proof}

\section{Thinning Heegaard paths}
\label{thinheegsect}

In this section, we will look more carefully at the process of producing a Heegaard path from a splitting path.  In particular, we will show that every splitting path determines a unique Heegaard path, up to horizontal slides.  In the case of the complex of surfaces for $\mT^1$ empty, this theorem has long been assumed, though it appears that no explicit proof was given until fairly recently by Lackenby~\cite[Proposition 3.1]{lacknb:1eff}.

Let $E$ be a Heegaard path and let $E'$ be a splitting path that results from thinning $E$.  By definition, there is a sequence $E = E_0, \dots, E_k = E'$ such that each $E_i$ is the result of a face slide on $E_{i-1}$ and the complexities of these paths are non-increasing.  Let $q_i$ be the face in $\mS$ that we slide $E_{i-1}$ across to produce $E_i$, as in Figure~\ref{fanfig}.
\begin{figure}[htb]
  \begin{center}
  \includegraphics[width=2.5in]{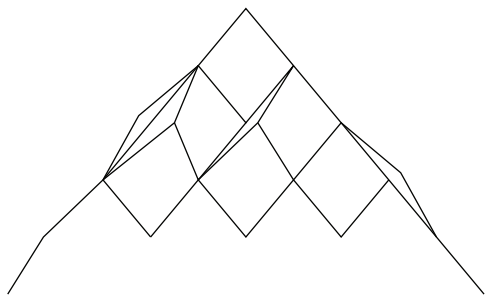}
  \caption{A \textit{fan} records a sequence of horizontal and vertical face slides that turn a Heegaard path into a splitting path.}
  \label{fanfig}
  \end{center}
\end{figure}

\begin{Def}
We will call the collection of faces $Q = \{q_i\}$ defined above a \textit{fan} from $E$ to $E'$.
\end{Def}

\begin{Lem}
\label{allverticallem}
Assume $\mS$ is an oriented height complex satisfying the translation axiom.  Let $E$ be a Heegaard path and $E'$ a splitting path that results from thinning $E$.  Then there is a Heegaard path $E^*$ equivalent to $E$ such that $E'$ is the result of thinning $E^*$ using only vertical slides.
\end{Lem}

\begin{proof}
Let $Q = \{q_i\}$ be a fan from $E$ to $E'$.  If there are no horizontal slides in $Q$ then $E'$ results from a sequence of vertical slides on $E$ and we're done.  Otherwise, we will replace $E$ with an equivalent Heegaard path $E''$ such that there is a fan from $E''$ to $E'$ with one fewer horizontal slide.  By induction, this will imply the result.

Let $q_j$ be the first face in $Q$ that corresponds to a horizontal slide.  Assume we have chosen the fan $Q$ so as to minimize the index $j$ over all fans with the same number of horizontal slides.  If $j = 1$ then this horizontal slide is the first move in the fan, so $E_1$ is equivalent to $E_0 = E$, and there is one fewer horizontal slides in the fan from $E_1$ to $E_0$.  If $j > 1$, we will show that we can choose a fan with the same number of horizontal slides in which $j$ is lower.

The face $q_j$ is a diamond, in which either both edges adjacent to the top vertex point towards it or both point away.  Two of the four edges are in the path $E_{j-1}$ and the other two are in $E_j$.  Let $e_1, e_2$ be the edges contained in $E_{j-1}$ and $e_3, e_4$ the edges in $E_j$ such that $e_1$ and $e_3$ are adjacent to the maximal vertex in $q_j$.  If the face $q_{j-1}$ does not contain one of the edges $e_1$, $e_2$ then the horizontal slide corresponding to $q_{j-1}$ can be performed before the vertical slide corresponding to $q_j$ and we can choose a fan in which $j$ is lower.  Thus $q_{j-1}$ must contain $e_1$ or $e_2$.
\begin{figure}[htb]
  \begin{center}
  \includegraphics[width=3.5in]{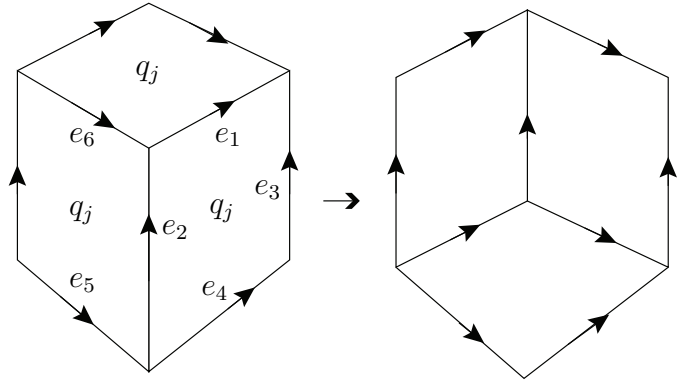}
  \put(-175,90){$e_1$}
  \put(-195,55){$e_2$}
  \put(-160,70){$e_3$}
  \put(-180,32){$e_4$}
  \put(-230,35){$e_5$}
  \put(-230,90){$e_6$}
  \put(-177,63){$q_j$}
  \put(-230,62){$q_j$}
  \put(-205,115){$q_j$}
  \caption{Two vertical slides followed by a horizontal slide can be replaced by a horizontal slide followed by vertical slides.}
  \label{cubefig}
  \end{center}
\end{figure}

Because the move corresponding to $q_{j-1}$ is vertical, the path $E_{j-1}$ contains the two edges of $q_{j-1}$ adjacent to its minimal vertex.  If $e_1$ is one of these edges, then the edge $e_2$ cannot be adjacent to $e_1$ in $E_{j-1}$.  Thus the face $q_{j-1}$ must contain the edge $e_2$ and the faces $q_j$ and $q_{j-1}$ have the same minimum vertex $v$.  Let $e_5$ be the other edge in $q_{j-1}$ adjacent to the minimum.  If $e_2$ and $e_4$ point towards $v$ then $e_5$ will point away, and vice versa.  The edges $e_2$ and $e_4$ are contained in the face $q_j$, so by the translation axiom, there is a 3-dimensional cube $C$ containing the edges $e_2$, $e_4$, $e_5$ and with $v$ as its minimum. 

Let $e_6$ be the edge above $e_2$ in the face $q_{j-1}$.  In the path $E_{j-2}$, the edges $e_6$ and $e_1$ are consecutive.  Because the complexity decreases across $e_6$ and increases across $e_1$ (or vice versa), the Heegaard path $E_0$ cannot contain $e_6$ and $e_1$, so the face of $C$ containing $e_6$ and $e_1$ must appear in the fan $Q$.  By reordering the faces in $Q$, we can assume that this is the face $q_{j-2}$.

The faces $q_{j-2}$, $q_{j-1}$, $q_j$ are three of the six faces in the cube $C$.  The other three faces are bounded by the same path in the 1-skeleton of $C$, but they correspond to a horizontal slide followed by two vertical slide.  Thus if we replace the faces $q_{j-2}$, $q_{j-1}$, $q_j$ in the fan $Q$ with the other three faces of $C$, the resulting fan will have the same number of horizontal slides, but the first horizontal slide with have index two less than the original.  By induction, we can repeat this process until $j = 1$, then replace $E$ with an equivalent Heegaard path in which fewer horizontal slides are required to produce $E'$.
\end{proof}

\begin{Lem}
\label{equivfillingsthenequivlem}
Assume $\mS$ is an oriented height complex satisfying the translation axiom.  If two Heegaard paths in $\mS$ can be thinned to produce equivalent splitting paths then they are equivalent.
\end{Lem}

\begin{proof}
Let $E$ and $F$ be Heegaard paths and let $E'$, $F'$ be splitting paths that result from thinning $E$ and $F$, respectively.  If $E'$ and $F'$ are equivalent then $E'$ is the result of a sequence of horizontal moves on $F'$.  If we append these to the fan from $F$ to $F'$ we get a fan from $F$ to $E'$, so $E'$ is also the result of thinning $F$.  By Lemma~\ref{allverticallem}, there are paths $E^*$, $F^*$ that are equivalent to $E$ and $F$, respectively, such that $E'$ results from thinning each of $E^*$, $F^*$ by only vertical moves.

Let $Q$ be a fan from $E^*$ to $E'$ and let $R$ be a fan from $F^*$ to $E'$ such that each fan contains only vertical moves.  Let $q_k$ be the last face in the fan $Q$ and let $e_1$, $e_2$ be the edges of $q_k$ adjacent to the minimal vertex of $q_k$.  The edges $e_1$ and $e_2$ cannot be in $F^*$ because $F^*$ is a Heegaard path, so $e_1$ and $e_2$ must be contained in a face $r_j$ of the fan $R$.  We can reorder the fan $R$ so that $r_j$ is the last face in the sequence.  Because one of $e_1$, $e_2$ points towards the minimum of the face $q_k$ and the other points away, the translation axiom implies that there is a unique face containing both $e_1$ and $e_2$, so $q_k = r_j$.  Thus the second-to-last paths defined by the fans $Q$ and $R$ are the same.  If we repeat this process, induction implies that the initial paths must be the same, completing the proof.
\end{proof}

The Heegaard path that results from a splitting path by vertical slides along diamonds will be called the \textit{Heegaard path associated to} the splitting path.  Lemma~\ref{equivfillingsthenequivlem} implies that if two splitting paths are equivalent then their associated Heegaard paths are equivalent.  In fact, something slightly stronger is true.

\begin{Lem}
If two splitting paths are related by a sequence of horizontal face slides and vertical face slides across diamonds then their associated Heegaard paths are equivalent.
\end{Lem}

\begin{proof}
By inducting on the number of face slides, we need only show that if two paths are related by a single face slide across a diamond then their associated Heegaard paths are equivalent.  If the face slide is a horizontal slide then the paths are equivalent, so Lemma~\ref{equivfillingsthenequivlem} implies that the associated Heegaard paths are equivalent.  If the slide is a vertical slide across a diamond then we can turn the lower path into a Heegaard path by first undoing this vertical slide.  Since the sequence of paths leading to the Heegaard path passes through the other path, the associated Heegaard paths are the same.
\end{proof}

\section{Heegaard splittings}
\label{heegsect}

We now return to the case when $\mS = \mS(M, \mT^1)$.  Define the \textit{genus} of a vertex $v \in \mS(M, \mT^1)$ to be the sum of the genera of the components of a surface representing $v$.  Note that the genus of a vertex does not depend on how the surface intersects $\mT^1$, so it does not change along an edge determined by a bridge disk or tree disk.  The genus is also constant along an edge corresponding to a compressing disk whose boundary separates a component of the surface.

The \textit{jump genus} of an edge in $\mS(M, \mT^1)$ is the (absolute value of the) difference in genus between its endpoints.  The \textit{genus} of an oriented edge path in $\mS(M, \mT^1)$ is one half the sum of the jump genera of its edges plus half the sum of the genera of its endpoints.

\begin{Lem}
The maximal vertex in a Heegaard path in $\mS(M, \mT^1)$ is a Heegaard surface whose genus is equal to the genus of the path.
\end{Lem}

\begin{proof}
By Lemma~\ref{alignkbodieslem}, the two monotonic arcs in a Heegaard path determine a pair of K-compression bodies that coincide along their positive boundary.  These compression bodies determine a Heegaard splitting for $M$.  The genus of the positive boundary of a compression body defined by an edge in $\mS$ is the genus of its negative boundary plus the jump genus.  By induction, the same is true for a compression body defined by any monotonic path.  Thus the genus of the single maximum in a Heegaard path is equal to the genus of either of its endpoints plus sum of the jump genera of the edges connecting it to that vertex.  Since there are two such paths, the total is twice its genus, so we divide by two.
\end{proof}

Note that opposite edges of any diamond in $\mS(M, \mT^1)$ have the same jump genera.  Thus if we change a path by a slide across a diamond or a triangle, the genus of the path does not change.  When we slide across a bigon, the genus of the path will change if and only if the jump genus of each of the edges is one.  Recall that given two Heegaard splittings for a 3-manifold, the \textit{stable genus} of the pair of Heegaard splittings is the genus of the smallest common stabilization.  The rest of this section will be devoted to proving the following:

\begin{Lem}
\label{sisimplyconnectedlem}
Every Heegaard splitting of $M$ is represented by a Heegaard path in $\mS(M, \mT^1)$.  Any two splitting paths are related by some sequence of face slides such that the genus of any intermediate is at most the (oriented) stable genus of the pair of Heegaard splittings.
\end{Lem}

Note that if two Heegaard splittings have stable genus equal to their own genus then they two splittings are isotopic.  Thus Lemma~\ref{sisimplyconnectedlem} also applies to isotopies between Heegaard surfaces.

In the case when $\mT^1$ is empty, this is almost immediate:  Every Heegaard surface can be compressed to a sphere in either direction, so this determines a path in which the maximal vertex represents the Heegaard surface.  By Reidemeister~\cite{reid} and Singer~\cite{sing}, any two Heegaard surfaces are related by stabilization, so there is a sequence of faces slides across bigons after which any two paths will have the same maximum.  Lemma~\ref{equivalentcompbodylem} implies that the two paths are related by horizontal slides.  In the case when $\mT^1$ is non-empty, we need to consider how this graph intersects the compression bodies of the Heegaard splitting.

\begin{Def}
Given a Morse function $f : M \rightarrow \mathbf{R}$, the pair $(f, \mT^1)$ will be called a \textit{Morse pair} if the restriction of $f$ to $\mT^1$ is Morse on the interior of each edge or loop in $\mT^1$, each vertex is a local minimum of $f|_{\mT^1}$, and any level contains at most one critical point of $f$, critical point of $f|_{\mT^1}$ or vertex of $\mT^1$.  
\end{Def}

Every Heegaard splitting for $M$ can be represented by a Morse function.  (This follows from the fact that a Morse function can be constructed on any handlebody containing an index-zero critical point and some number of index-one critical points.)  A general position argument implies the following:

\begin{Lem}
\label{gpformorselem}
If $f$ is a Morse function on $M$ then $f$ can be isotoped by an arbitrarily small amount so that $(f, \mT^1)$ is a Morse pair.
\end{Lem}

This gives us the first half of Lemma~\ref{sisimplyconnectedlem}:

\begin{Lem}
\label{morsethenpathlem}
If $(f, \mT^1)$ is a Morse pair then there is a splitting path $E$ in $\mS(M, \mT^1)$ such that each vertex of $E$ is represented by a level set $f^{-1}(t)$ for some $t$.  Conversely, for any splitting path $E$, there is a Morse pair $(f, \mT^1)$ such that every vertex of $E$ is represented by a level set of $f$ and vice versa.
\end{Lem}

\begin{proof}
Let $t_1,\dots, t_k$ be the values such that $f^{-1}(t_i)$ contains either an index-one or index-two critical point of $f$, a critical point of $f|_{\mT^1}$ or a vertex of $\mT^1$.  Assume the values $\{t_i\}$ are labeled so that $t_{i+1} > t_i$ for each $i$.  Let $s_0, \dots, s_k$ be regular values of $f$ such that $t_i < s_i < t_{i+1}$ for each $i$.  For each $i$, $f^{-1}(s_i)$ is a surface in $M$ transverse to $\mT^1$.  Because each surface is a level set of $f$, it is strongly separating and there is an induced transverse orientation.  Thus each $s_i$ determines a vertex $v_i$ in $\mS(M, \mT^1)$.  

Between each $s_{i-1}$, $s_i$, there may be one or more index-zero or index-three critical points, plus a single value $t_i$ at which there is a critical value of $f$, a critical value of $f|_{\mT^1}$ or a vertex of $\mT^1$.  The index-zero and index-three critical points of $f$ correspond to adding or removing trivial sphere components, so these do not change the corresponding vertices in $\mS(M, \mT^1)$.  The edge is determined by the other critical point in this interval.  In the first case, the surfaces representing $v_{i-1}$ and $v_i$ are related by a compression.  In the second case, the surfaces are related by a bridge compression, and in the last case they are related by a tree compression.  Each surface is on the positive side of the previous surface, so the Morse pair $(f, \mT^1)$ determines a splitting path.  

Conversely, for any splitting path $E$ in $\mS(M, \mT^1)$, Lemma~\ref{alignkbodieslem} implies that there is a family of K-compression bodies whose boundaries represent the maxima and the minima in the path.  We can define a Morse function on each compression body such that the critical points of the Morse function correspond to the edges in the path.  The union of these Morse functions is a function on all of $M$ representing $E$.
\end{proof}

We will now make a subtle but important switch from sphere-blind isotopies to isotopies.  In particular, because the complement of a Heegaard surface is irreducible, two Heegaard surfaces will be isotopic if and only if they are blind isotopic.

\begin{Lem}
Every splitting path in $\mS(M, \mT^1)$ determines a Heegaard splitting for $M$ that is unique up to isotopy transverse to $\mT^1$.
\end{Lem}

\begin{proof}
By Lemma~\ref{heegaardfromsplittinglem}, every splitting path determines a Heegaard path, and this path is unique up to horizontal slides.  In particular, the maximum of this path is uniquely determined.  This maximum corresponds to a blind isotopy class of surfaces in $M$, so this surface is only determined up to sphere tubing.  However, K-compression bodies are irreducible (in the complement of the graph) so any surface can only be tubed to a trivial sphere.  Tubing to a trivial sphere produces an isotopic surface, so the maximum of the Heegaard path determines a surface that is unique up to isotopy transverse to $\mT^1$.
\end{proof}

\begin{proof}[Proof of Lemma~\ref{sisimplyconnectedlem}]
Let $(\Sigma, H^-, H^+)$ be a Heegaard splitting for $M$ and let $f$ be a Morse function representing $(\Sigma, H^-, H^+)$.  By Lemma~\ref{gpformorselem}, we can perturb $f$ so that $(f, \mT^1)$ is a Morse pair, and by Lemma~\ref{morsethenpathlem}, there is a path $E$ in $\mS(M, \mT^1)$ such that each vertex is represented by a level set of $f$.  By Lemma~\ref{heegaardfromsplittinglem}, there is a Heegaard path $E^*$ such that $E$ results from thinning $E^*$.  Because the genera of the vertices in $E$ increase monotonically to a maximum, then decrease monotonically, the maximal vertex of $E^*$ is isotopic in $M$ to the maximum vertex of $E$ (though this isotopy may not be transverse to $\mT^1$.  Thus the Heegaard path $E^*$ also represents $(\Sigma, H^-, H^+)$.

Given two splitting paths $E$, $E'$ in $\mS(M, \mT^1)$, we can replace $E$ and $E'$ by Heegaard paths by Lemma~\ref{heegaardfromsplittinglem}.  In this case, the maxima of $E$ and $E'$ are Heegaard surfaces.  By Reidemeister and Singer's Theorem~\cite{reid}~\cite{sing}, we can stabilize each of these surfaces some number of times, (in a ball disjoint from $\mT^1$) after which they will be ambient isotopic.  These stabilizations correspond to sliding the paths $E$ and $E'$ across bigons in $\mS(M, \mT^1)$, so we can slide $E$ and $E'$ to Heegaard paths $E_H$, $E'_H$ such that their maxima represent isotopic surfaces in $M$ (though not necessarily transversely isotopic).

By Lemma~\ref{heegaardfromsplittinglem}, there is a Morse function $f$ on $M$ representing the path $E_H$, and there is a level in this Morse function representing the maximum vertex of $E_H$.  Because the maxima of $E_H$ and $E'_H$ are isotopic, there is a Morse function $h$ on $M$ that is isotopic to $f$ such that a level surface of $h$ represents the maximum vertex in $E'_H$.  We can isotope $h$ slightly so that $(h, \mT^1)$ is a Morse pair, defining a path $E''$.  Since the maxima of $E'_H$ and $E''$ are the same vertex in $\mS(M, \mT^1)$, Lemma~\ref{equivalentcompbodylem} implies that $E'_H$ and $E''$ are equivalent.  Thus we need only show that $E_H$ and $E''$ are related by face slides.

The isotopy from $f$ to $h$ determines a family $\{f_t\}$ of Morse functions, each isotopic to $f$.  If we choose a generic family, the critical points of $f_t$ will be disjoint from $\mT^1$ for every $t$.  Choose a ball neighborhood $N_v$ for every vertex $v$ of $\mT^1$ such that the critical points of each $f_t$ are also disjoint from $N_v$.  Let $N'_v$ be a smaller ball neighborhood within each $N_v$.  Because the critical points are disjoint from these balls, we can perturb the family of functions to a new family $\{f'_t\}$ so that within $N'_v$ the restriction of $f'_t$ is monotonically decreasing on each edge, and so that outside $N_v$, $f'_t$ agrees with $f_t$. 

The final function $f'_1$ will agree with $f_1$ outside each sphere-cross-interval region $N_v \setminus N'_v$.  On each sphere leaf in this product, the restriction of $f'_1$ or $f_1$ contains a single local maximum and a single local minimum.  The union of local minima forms an arc between the boundary components.  There is exactly one isotopy class of arcs connecting the boundary spheres of a sphere cross an interval, so there is an isotopy from $f'_1$ to $f_1$ that is the identity outside of each $N_v \setminus N'_v$.  Thus we can assume that $f'_1 = f_1$.  If we replace the family $\{f_t\}$ with $\{f'_t\}$ we will have found a family of Morse functions such that each vertex of $\mT^1$ is a local maximum of the restriction of each $f_t$ to $\mT^1$.

By general position, we can assume that the restriction of $f_t$ to $\mT^1$ is smooth for each $t$ and is Morse for all but finitely many $t$.  At the non-Morse values, the function $f_t|_{\mT^1}$ will change in one of two ways: two vertices or two critical points in edges of $\mT^1$ may pass through the same level, or two critical points in an edge may cancel/uncancel.  These changes correspond to face slides in the paths determined by the Morse pairs, so there is a sequence of face slides that turn $E_H$ into $E''$.
\end{proof}

\section{The derived complex}
\label{derivedsect}

In order to sum up the picture that comes from this axiomatic framework, we will construct the complex $\mD(\mS)$ that was promised in the introduction.  Theorem~\ref{mainthm} states a certain subcomplex $\mD^2(\mS)$ of $\mD(\mS(M, \mT^1))$ encodes all the combinatorics need to understand $\mS$, and we will present the proof at the end of this section.  

Assume $\mS$ is an oriented height complex satisfying the Casson-Gordon axiom, the barrier axiom and the translation axiom. 

\begin{Def}
The \textit{derived complex} $\mD = \mD(\mS)$ is the complex constructed as follows:
\begin{enumerate}
    \item The vertices of $\mD(\mS)$ are the rigid vertices of $\mS$, i.e. those with well defined index.
    \item For each equivalence class of oriented or reverse oriented paths in $\mS$ decreasing from an index-$n$ vertex to a lower index vertex, we will include in $\mD(\mS)$ a single edge connecting the corresponding vertices in $\mD$.
    \item Any decreasing path in the 1-skeleton of $\mD$ passes through vertices whose indices decrease and any two vertices in such a path are connected by a single path (possibly outside the path.)  We will include in $\mD$ a simplex spanned by these vertices.  
\end{enumerate}
\end{Def}

\begin{Def}
For each positive integer $i$, define $\mD^i(\mS)$ as the subcomplex of $\mD(\mS)$ spanned by vertices of index at most $i$.
\end{Def}

In particular, the complex $\mD^2(\mS(M, \mT^1))$ is a two-dimensional simplicial complex whose vertices are index-zero, -one and -two surfaces and whose faces are determined by triples in which a given index-two surface can be K-compressed to the index-one surface and then to the index-zero surface.

Recall that Theorem~\ref{mainthm} states that any path in $\mS$ can be thinned to a path that is represented in $\mD^1(\mS)$ and any sequence of face slides in $\mS$ are represented by face slides in $\mD^2(\mS)$.

\begin{proof}[Proof of Theorem~\ref{mainthm}]
First we will show that every path in $\mS$ can be thinned to a path that is represented in $\mD(\mS)$.  Because $\mS$ satisfies the net axiom, every splitting path can be weakly reduced to a thin path. By Lemma~\ref{thinpathsimaxlem}, the locally maximal vertices in a thin path are strongly irreducible and by Lemma~\ref{cgforsurfaceslem}, the locally minimal vertices are index-zero.  Because $\mS$ satisfies the barrier axiom and the minima are index-zero, the path link of each vertex is equal to the descending link, so each vertex is, in fact, index-one.  Thus a thin path in $\mS$ consists of a sequence of monotonic sub-paths between index-zero and index-one vertices. 

Given paths $E$ and $E'$ that are represented in $\mD$, let $E = E_0, \dots, E_k = E'$ be a sequence of paths in $\mS$, related by face slides.  This sequence determines a path $F$ in the path complex $\mP(\mS)$.  By Lemma~\ref{unorientedthinninglem}, there is a thin path $F'$ in $\mP(\mS)$ between the vertices representing $E$ and $E'$ that results from thinning $F$.  The local maxima of $F'$ will be index-one vertices of $\mP(\mS)$ and the local minima will be index-zero.  This translates to a sequence of paths $E = E'_0, \dots, E'_k = E'$ in $\mS$. 

The minimal vertices of $F'$ are thin paths in $\mS$, so these are represented in $\mP(\mS)$ as in the above argument.  By Lemma~\ref{critpathcritmaxlem}, each maximal vertex of $F'$ represents a path with a single path-index-two maximum and some number of index-one maxima.  By Lemma~\ref{cg1lem}, this implies that each maximum represents a path with index-zero minima, so by the barrier axiom, the maxima of a path representing a maximum of $F$ consist of one index-two vertex and some number of index-one vertices.  Since the maxima are index-one or index-two and the minima are index-zero, this is represented by a path in $\mD(\mS)$.

Thus a sequence of paths in $\mS$ from $E$ to $E'$ is represented by a sequence of paths in $\mD$.  To see that these are related by face slides in $\mD$, note that each monotonic path in $\mP(\mS)$ corresponds to thinning the path in $\mS$ corresponding to the local maximum at the end of the monotonic path.  Because the local maximum has a single weakly reducible maximum, there is exactly one maximum that can be thinned.  Thus each monotonic segment in the path $F'$ in $\mP(\mS)$ corresponds to sliding a path across a face in $\mD(\mS)$.
\end{proof}

\section{Surfaces and the derived complex}
\label{surfderivedsect}

In this section, we will interpret thin position in terms of surfaces in a 3-manifold $M$ in order to prove Theorem~\ref{mainthm1}.  We will prove each of the statements in the theorem separately, beginning with the claims about incompressible surfaces.  Let $\mD = \mD(M, \mT^1)$ be the derived complex $\mD(\mS(M, \mT^1))$ for the complex of surfaces.

\begin{Lem}
\label{compcriterialem}
A vertex $v$ in $\mD$ represents a compressible surface if and only if there is a path in $\mD$, starting at this vertex, and passing through index-zero and index-one vertices such that the genus does not increase, but the final vertex has genus lower than that of $v$.
\end{Lem}

(Note that this is the contrapositive of the statement in Theorem~\ref{mainthm1}.)

\begin{proof}
An edge in $\mD$ in which the genus does not increase represents an isotopy of the surface through $\mT^1$.  If such a path exists, then the initial edges, in which the genus is constant, represent an isotopy of a surface representing $v$.  The first edge along which the genus decreases represents a compression of this surface, so $v$ represents a compressible surface.

Conversely, assume $S$ represents a compressible surface and let $D$ be a compressing disk for $S$.  This disk may intersect $\mT^1$, but there is an isotopy of $S \cup D$ after which it is disjoint from $\mT^1$.  This isotopy is represented by a path in $\mS(M, \mT^1)$ starting at $v$.  After this isotopy, compress $S$ across $D$, then K-compress $S$ as far as possible with respect to $\mT^1$.  This sequence of K-compressions is also represented by a path in $\mS$, and the final vertex has index zero in $\mS$.  By Theorem~\ref{mainthm}, we can thin this path so that it is represented in $\mD$.  The genera of the vertices are never above that of $v$, but the final vertex is below that of $v$, so this path characterizes $v$ as compressible.
\end{proof}

\begin{Lem}
Two index-zero vertices $v, v' \in \mD$ representing incompressible surfaces will represent isotopic surfaces if and only if there is a path from one to the other, passing through index-one and index-zero vertices in which the genus is constant.
\end{Lem}

\begin{proof}
Because both vertices represent incompressible surfaces, any path in which the genus does not increase must have constant genus.  If there is a constant genus path from one to the other then each edge determines an isotopy, and the sequence of isotopies takes one surface to the other.  Conversely, if the surfaces representing $v$ and $v'$ are isotopic, then the isotopy from one surface to the other determines a path connecting the vertices in $\mS$.  We can thin this path to a path that is represented in $\mD$, and this path will have constant genus.
\end{proof}

\begin{Lem}
\label{splittingindlem}
Every splitting path in $\mD$ represents a Heegaard splitting for $M$ amd every unstabilized Heegaard splitting for $M$ is represented by a splitting path in $\mD$ passing through index-one and index-zero vertices.
\end{Lem}

\begin{proof}
Every splitting path in $\mD$ determines a splitting path in $\mS$, which determines a Heegaard splitting for $M$ by Lemma~\ref{sisimplyconnectedlem}.  Conversely, every Heegaard splitting for $M$ is represented by a path in $\mS$ by Lemma~\ref{sisimplyconnectedlem}, and this path can be thinned to a path in $\mD$.  By Lemma~\ref{equivfillingsthenequivlem}, this thinned path represents the same Heegaard path in $\mS$, and therefore the same Heegaard splitting.
\end{proof}

\begin{Lem}
\label{stabilizedindlem}
A genus $g$ splitting path in $\mD$ represents a stabilized Heegaard splitting if and only if there is a sequence of face slides, passing through index-zero, index-one and index-two vertices, in which the genera of the intermediate paths does not increase, but the genus of the final path is less than $g$.
\end{Lem}

\begin{proof}
If such a sequence of slides exists in $\mD$ then every slide in which the genus remains constant represents an isotopy of the Heegaard splitting.  The only way in which the genus of the path can decrease is a destabilization, so the first slide in which the genus decreases shows that the original Heegaard splitting is stabilized.

Conversely, assume the path $E$ represents a stabilized Heegaard surface $\Sigma$.  Let $\Sigma'$ be the result of destabilizing $\Sigma$ as much as possible.  The Heegaard surface $\Sigma'$ is unstabilized so it is represented by a path $E'$ in $\mD$ by Lemma~\ref{splittingindlem}.  The paths $E$ and $E'$ represent paths in $\mS$.  Their stable genus is $g$ because $\Sigma$ is  a common stabilization for $\Sigma$ and $\Sigma'$.  Thus Lemma~\ref{sisimplyconnectedlem} implies that they are related by a sequence of paths in which the genera of the intermediate paths is at most $g$.  By Theorem~\ref{mainthm}, this sequence of slides is represented in $\mD$, so there is a sequence of slides in $\mD$ from $E$ to $E'$ in which the genus does not increase, but the final genus is strictly less than $g$.
\end{proof}

\begin{Lem}
Two splitting paths represent Heegaard splittings with stable genus at most $g$ if and only if the paths are related by face slides in $\mD$ such that the genus of any intermediate path is at most $g$.  In particular, two paths representing unstabilized Heegaard splittings will be isotopic if and only if there is such a sequence of face slides in which the genus is constant.
\end{Lem}

\begin{proof}
As in Lemma~\ref{stabilizedindlem}, a sequence of face slides in $\mD$ in which the genus increases then decreases represent a squence of stabilizations and destabilizations that turn one Heegaard splitting into the other, so the stable genus is at most the maximum genus of the intermediate paths.

Conversely, if two paths represent Heegaard splittings with stable genus $g$ then Lemma~\ref{sisimplyconnectedlem} implies that they are represented by a sequence of face slides in $\mS$, and Lemma~\ref{mainthm1} implies that they are represented by face slides in $\mD$ in which the intermediate paths have genus at most $g$.
\end{proof}

\section{Odds and ends}
\label{oddsect}

In this section, we prove two Lemmas that, while not directly applicable in this paper, are generalizations of results that have proved useful in other contexts.  The first Lemma is specifically about the complex of surfaces defined above.

\begin{Lem}
\label{disksetcontractible}
The positive and negative links of any vertex in the complex of surface $\mS(M, \mT^1)$ is contractible.
\end{Lem}

This follows almost immediately from work of Sangbum Cho, who has developed a very effective method for showing that complexes related to the disk complex for a 3-manifold are contractible.  The most general statement of his method is in Proposition 3.1 of~\cite{cho}, but we will use the following corollary, which is also stated in his paper.  Below, $H$ is a compression body and $D(H)$ is the disk complex for $H$.  Recall that a subcomplex $L$ of a simplicial complex is \textit{full} if whenever the corners of a simplex are in $L$, the entire simplex is in $L$.

\begin{Thm}[Theorem 4.2 in~\cite{cho}]
If $L$ is a full subcomplex of $D(H)$ satisfying the following condition, then $L$ is contractible:

Suppose $E$ and $D$ are any two disks in $H$ representing vertices of $L$ such that $E \cap D$ is non-empty.  If $C \subset D$ is a disk cut off from $D$ by an outermost arc $\alpha$ of $D \cap E$ in $D$ such that $C \cap E = \alpha$, then at least one of the disks obtained from surgery on $E$ along $C$ also represents a vertex of $L$.
\end{Thm}

(Cho uses the symbol $K$ where we use $L$.)  Thus to show that the positive or negative link is contractible, we need only check that it satisfies the condition in Cho's Theorem.

\begin{proof}[Proof of Lemma~\ref{disksetcontractible}]
If $(H, K)$ is a connected K-compression body then the complement $H'$ in $H$ of a regular neighborhood of $K$ is a compression body.  The intersection of any K-disk for $H$ with $H'$ is a compression disk for $H$ and we will let $L$ be the set of compression disks for $H'$ induced by K-disks for $H$.  If $E$ and $D$ are two K-disks for $(H, K)$ and $C \subset D$ is an outermost disk as in the condition in Cho's Theorem, then one of the disks that results from surgery on $E$ across $C$ is also a K-disk for $(H, K)$.  The result of surgery in $H'$ is the image of the result of the surgery in $H$, so the resulting disk is in $L$.  Thus by Cho's theorem, the complex of K-disks for a connected K-compression body is contractible.  The complex of K-disks for a disconnected K-compression body is the join of the complexes for its components, (i.e. the union of the two complexes with edges connecting every vertex of one to every vertex of each other, plus the induced higher dimensional cells) so this set is also contractible.
\end{proof}

The second result is about height complexes in general.  We begin with the following definition, which translates the notion of ``amalgamation'', introduced by Schultens~\cite{sch:cross}, into the axiomatic setting.

\begin{Def}
A Heegaard path $E$ is an \textit{amalgamation} along a vertex $v$ if it can be thinned to a path $E'$ that passes through $v'$.
\end{Def}

In particular, if $E$ can be thinned to a splitting path $E'$ then $E$ will be an amalgamation along each of the locally minimal vertices in $E'$.

\begin{Lem}
Assume $\mS$ is an oriented height complex satisfying the Casson-Gordon axiom.  If $v$ is the maximum of a Heegaard path $E$ and there is a (not necessarily oriented) path decreasing from $v$ to $v'$ then $E$ is an amalgamation along $v'$.
\end{Lem}

This Lemma is used, in the context of generalized Heegaard splittings, in~\cite{me:stabs}.  In fact, the version here is slightly more general because the original version required that the manifold be irreducible.

\begin{proof}
Let $E$ be an oriented path with a single maximum $v$ and let $P$ be a (not necessarily oriented) decreasing path from $v$ to a vertex $v'$.  Without loss of generality, assume that the initial edge of $P$ points away from $v$.  Let $P_1$ be the longest directed subpath of $P$ starting at $v$.  Let $Q_1$ be the result of extending $P_1$ as far as possible.  Because $\mS$ satisfies the barrier axiom, $Q_1$ ends at the same vertex as $E$ and is related to the second leg of $E$ by horizontal slides.  Let $E_1$ be the path consisting of the first leg of $E$ and the path $Q_1$.  This path is equivalent to $E$ and contains the path $P_1$ as a subpath.  Let $v_1$ be the endpoint of $P_1$.

If $P_1 = P$ then $E$ is equivalent to a path passing through $v' = v_1$ as promised.  Otherwise, let $P_2$ be the longest reverse oriented subpath of $P$ starting at $v_1$.  Let $E'_2$ be the result of weakly reducing the subpath of $E$ from $v_-$ to $v_1$ as far as possible.  The last edge of $P_2$ is below $v_1$ and points towards it.  If the last extremal vertex of $E'_1$ is a maximum then the Casson-Gordon axiom would imply that this maximum is weakly reducible.  Thus the last extremum of $E'_2$ must be an incompressible local minimum.  Extend $P_2$ to this minimum (by the Barrier axiom) and then let $E_2$ be the result of a sequence of horizontal slides on $E'_2$, after which it contains the path $P_2$.

If the endpoint $v_2$ of $P_2$ is the endpoint of $P$ then $v' = v_2$, so we have thinned $E$ tp a path passing through $v'$.  Otherwise, we can repeat the argument, weakly reducing $E_2$ to a path that contains the endpoint of the next directed subpath $P_3$ of $P$.  By continuing in this manner, we will eventually weakly reduce $E$ to a path containing the lower endpoint of $P$.
\end{proof}

\section{Categories}
\label{catsect}

In the hopes that axiomatic thin position will eventually be applied in other settings, it might be useful to take a brief look at the situation from the point of view of category theory.  We would like to define a \textit{theory of thin position} as a functor from a given category to a category of height complexes.  To define this category of height complexes, we need to determine what our functors will be.

\begin{Def}
A \textit{homomorphism} between height complexes $\mS$ and $\mS'$ is a map from the 2-skeleton of $\mS$ to the 2-skeleton of $\mS'$ with the following properties:
\begin{itemize}
    \item Each vertex of $\mS$ is sent to a vertex of $\mS'$.
    \item Each edge of $\mS$ is sent to either a vertex of $\mS'$ or to an edge of $\mS'$ such that the endpoints are ordered in the same way.
    \item Each face $f$ of $\mS$ is sent to a vertex of $\mS'$, an edge of $\mS'$, or to a face $f'$ of $\mS'$ such that the edges of $f$ adjacent to the top vertex are sent to the edges of $f'$ adjacent to the top vertex.
\end{itemize}
\end{Def}

While we have only defined a homomorphism between the 2-skeletons of the height complexes, the Morse axiom implies that the higher dimensional cells are completely determined by the 2-cells, so one can always extend the map to the entire complexes.  Define $\mathcal{H}$ as the category whose objects are height complexes and whose functors are height complex homomorphisms.

For the complex of surfaces defined above, consider the category $\mathcal{MG}$ of pairs $(M, \mT^1)$ where $M$ is a compact, connected, orientable 3-manifold and $\mT^1 \subset M$ is a properly embedded graph in $M$.  If $M$ has boundary, we will fix a partition $\{\partial_- M, \partial_+ M\}$ of its boundary components, though we will not notate this partition in the ordered pair. A functor $f : (M, \mT^1) \rightarrow (M', \mT'^1)$ is an embedding (possibly a homeomorphism) of $M$ into $M'$ that sends $\mT^1$ onto $\mT'^1$ and such that the frontier of each component of $M' \setminus M$ is contained in either  $\partial_- M$ or $\partial_+ M$.

Define $STH : \mathcal{MG} \rightarrow \mathcal{H}$ as the function that associates to each pair $(M, \mT^1)$ in $\mathcal{MG}$, its complex of surface $\mS(M, \mT^1)$.  

\begin{Lem}
Every morphism from $(M, \mT^1)$ to $(M', \mT'^1)$ induces a height complex homomorphism from $\mS(M, \mT^1)$ to $\mS(M', \mT'^1)$.  This makes the map $STH$ a functor from $\mathcal{MG}$ to $\mathcal{H}$.
\end{Lem}

\begin{proof} 
Let $f : M \rightarrow M'$ be an embedding that sends $\mT^1$ onto $\mT'^1$ such that the frontier of each component of $M' \setminus M$ is contained in either  $\partial_- M$ or $\partial_+ M$.  If $S$ is a surface in $M$ that is transverse to $\mT^1$ then $f(S)$ will be an embedded surface transverse to $\mT'^1$.  If $S$ is strongly separating in $M$ then the condition on the frontier of $M' \setminus M$ implies that the image of $S$ will be strongly separating in $M$.  Thus $f$ induces a map from the vertices of $\mS(M, \mT^1)$ into the vertices of $\mS(M', \mT'^1)$.  Every edge in $\mS$ corresponds to a K-compression of a surface $S$ and the reader can check that if $S'$ is the result of a K-compression of $S$ in $M$ then $f(S')$ is either sphere-blind isotopic to $f(S)$ or the result of a K-compression.  This implies that edges are sent to vertices or edges.  Finally, if two K-disks for $S$ are disjoint in $M$ then their images in $M'$ will be disjoint, so faces are sent to faces.
\end{proof}

For example, the identity map on $M$ defines a morphism from $(M, \mT^1)$ to $(M, \emptyset)$ that induces a height complex homomorphism from $\mS(M, \mT^1)$ to $\mS(M, \emptyset)$.  This homomorphism sends each edge of $\mS(M, \mT^1)$ that corresponds to a bridge compression or tree compression to a vertex in $\mS(M, \emptyset)$.  In other words, we can construct $\mS(M, \emptyset)$ from $\mS(M, \mT^1)$ by taking the quotient that crushes all these edges to points.  (This is implicit in the proof of Theorem~\ref{mainthm1}.

Functors between height complexes can also be used to compare different theories of thin position.  For example, given a knot $K$ in $S^3$, let $\mS_G(K)$ be the height complex whose vertices are transversely oriented spheres in $S^3$ transverse to $K$ and whose edges are bridge disks with respect to $K$.  This complex encodes the original thin position for knots defined by Gabai~\cite{gabai}.  The reader can check that it satisfies all the axioms defined so far.  There is a natural embedding of $\mS_G(K)$ into $\mS(S^3, K)$ because every sphere in $S^3$ is strongly separating.  Every vertex of $\mS_G(K)$ defines a vertex of $\mS(M, K)$ and every edge defines an edge.  There is thus a one-to-one height complex homomorphism from $\mS_G(K)$ into $\mS(M, K)$.

\begin{Def}
We will say that a theory of thin position $B : \mathcal{B} \rightarrow \mathcal{H}$ \textit{generalizes} a theory $A : \mathcal{A} \rightarrow \mathcal{H}$ if $\mathcal{A}$ is a subcategory of $\mathcal{B}$ and for every $X \in \mathcal{A}$, there is a canonical injective height homomorphism of $A(X)$ into $B(X)$.  
\end{Def}

Under this definition, the Hayashi-Shimokawa thin position defined by the complex of surfaces generalizes both Gabai's thin position for knots and Scharlemann-Thompson thin position for 3-manifolds (without knots).  Tomova's thin position with cut disks~\cite{tom:brcompare} and Taylor-Tomova's thin position for graphs~\cite{tomovataylor} are further generalizations of Hayashi-Shimokawa's thin position under this definition.

\section{Gordon's conjecture}
\label{gordonsect}

One goal in developing axiomatic thin position was to generalize the methods developed by Bachman in his proof of the Gordon conjecture~\cite{bach:gordon}.  (The conjecture was also proved independently by Qiu~\cite{qiuscharl}, using vastly different methods.)  In this section, we present an exposition of Bachman's proof in terms of the complex of surfaces.  This is essentially the same proof with different terminology, but we believe that this exposition makes the proof clearer. 

Let $M_1$, $M_2$ be 3-manifolds and let $M_1 \# M_2$ be their connect sum, i.e. the result of removing a ball from each manifold and gluing the manifolds along the resulting sphere boundary components.  Given Heegaard surfaces $\Sigma_1 \subset M_1$, $\Sigma_2 \subset M_2$, we can choose the balls that we remove to be in the complements of these surfaces, so that both become closed surfaces in the connect sum.  These two surfaces and the connect sum sphere define the two maxima and one minima, respectively, for a splitting path in the complex of surfaces for $M_1 \# M_2$.  In fact, there are four such paths depending on which side of each surface we remove the balls from.  These different choices may produce different Heegaard splittings, but we will choose one of them.  The splitting path defines a (unique) Heegaard splitting, with Heegaard surface $\Sigma$, by Lemma~\ref{heegaardfromsplittinglem}.  

\begin{Thm}[Gordon's conjecture~\cite{bach:gordon},~\cite{qiuscharl}]
\label{gordonconj}
The Heegaard surface $\Sigma$ will be stabilized if and only if one of $\Sigma_1$, $\Sigma_2$ is stabilized.
\end{Thm}

To prove this, we will show that the derived complex of the connect sum is the 2-skeleton of the direct product of the derived complexes of $M_1$ and $M_2$.  We can then analyze Heegaard splittings of $M_1 \# M_2$ by projecting splitting paths into the derived complexes of $M_1$ and $M_2$.  Here, we will take $\mT^1$ to be the empty graph, and will shorten our notation to be $\mS(M) = \mS(M, \emptyset)$.

\begin{Lem}
\label{disjointfromcnctspherelem}
Every rigid surface in $M_1 \# M_2$ can be isotoped disjoint from the connect sum sphere.
\end{Lem}

\begin{proof}
Let $v$ be a vertex in $\mS(M_1 \# M_2)$ representing a rigid surface $S$.  Let $v_-$, $v_+$ be the barrier vertices for $v$ promised by the Barrier axiom.  (The complex of surfaces satisfies this axiom by Lemma~\ref{barrieraxiomlem}.)  Each of the surfaces $S_-$, $S_+$ representing these vertices is incompressible, so each can be isotopied disjoint from the connect sum sphere.  The two surfaces bound a submanifold $M_S$ of $M_1 \# M_2$ and $S$ is a Heegaard surface for $M_S$.  If the connect sum sphere is disjoint from $M_S$ then it is disjoint from $S$ and we are done.  Otherwise, if the connect sum sphere is contained in $M_S$ then Lemma~\ref{cgsteponelem} implies that $S$ is blind isotopic to a surface disjoint from the connect sum sphere, since $S$ is topologically minimal.  This occurs if and only if the connect sum sphere is parallel to a boundary component of $M_S$, in which case $S$ is again disjoint from the connect sum sphere.
\end{proof}

This implies that every rigid surface $S$ in $M_1 \# M_2$ is the union of the images of a surface $S_1$ in $M_1$ and a surface $S_2$ in $M_2$.  The descending link of $S$ is the join of the descending links of $S_1$ and $S_2$ so if either of these descending links is contractible, the descending link of $S$ is contractible.  Since $S$ is rigid, the surfaces $S_1$, $S_2$ must be rigid as well.  Let $\rho_i: \mD(M_1 \# M_2) \rightarrow \mD(M_i)$ be the map that sends a vertex representing a rigid surface $S \subset M_1 \# M_2$ to $S_i = S \cap M_i$.  Note that this result does not hold for the $\mS(M, \emptyset)$.

\begin{Lem}
Every path in $\mD^2(M_1 \# M_2)$ projects to a pair of paths in $\mD^2(M_1)$, $\mD^2(M_2)$.  Every face slide in $\mD^2(M_1 \# M_2)$ projects to either a face slide in $\mD^2(M_1)$, a face slide in $\mD^2(M_2)$, or the identity on both.
\end{Lem}

\begin{proof}
The index of a surface in $M_1 \# M_2$ is the sum of the indices of the surfaces in $M_1$ and $M_2$.  Thus every index-one vertex in $\mD^2(M_1 \# M_2)$ projects to an index-one vertex and an index-zero vertex.  An index-two vertex projects to either two index-one vertices, or an index-two and an index-zero vertex.  Every edge in $\mD^2$ goes from an index-one surface or index-one, to an index-one to index-sero surface.  If the indexin $\mD(M_1 \# M_2)$ changes by one then the index only changes in $M_1$ or $M_2$ so all the compressions take place in one of the summands.  If the index changes by two in $\mD(M_1 \# M_2)$ then either it changes by two in one of the summands, or it changes by one in each summand.  In any of these cases, the edge in $\mD(M_1 \# M_2)$ projects to an edge or a vertex in each summand.

A face in $\mD^2(M_1 \# M_2)$ is defined by three edges.  If these edges correspond to compressions in different manifolds then the face corresponds to changing the order of the compressions.  In this case, the face slide projects to the identity in each summand.  Otherwise, the slide occurs inside one of the summand and thus corresponds to a face slide in $\mD^2(M_1)$ or $\mD^2(M_2)$. 
\end{proof}

\begin{proof}[Proof of Theorem~\ref{gordonconj}]
A Heegaard splitting will be stabilized if and only if the Heegaard path defined by the Heegaard splitting can be slid, without changing its genus, to a splitting path such that there is a bigon below one of its maxima.  Consider such a sequence of face slides in $\mD^2(M_1 \# M_2)$.  This projects to sequences of paths in each of $\mD^2(M_1)$ and $\mD^2(M_2)$.  Since the sequence of paths in $\mD^2(M_1 \# M_2)$ does not involve a stabilization, neither of the projected sequences involves a stabilization.  However, the original sequence of paths does involve a destabilization.  The pair of disks defining this bigon below the maximum in the path in $\mD^2(M_1 \# M_2)$ determine a pair of disks in a maximum of one of the final projected paths.  Thus one of the projected sequence of paths ends with a stabilization.  Since the sequence of projected paths does not involve a stabilization, but ends with a destabilization, the original Heegaard splitting must have been stabilized.
\end{proof}

\section{Normal Surfaces}
\label{normalsect}

We now switch to the task of proving Theorem~\ref{mainthm2}.  In this and the next two sections, we define normal, almost normal and index-two normal surfaces, then show that the vertices of $\mD(M, \mT^1)$ are represented by such surfaces.  In Section~\ref{bentsurfacesect}, we describe an algorithm for calculating the edges and faces spanned by these vertices.  Then in Section~\ref{thm2sect}, we combine these results to prove Theorem~\ref{mainthm2}.

\begin{Def}
A \textit{triangle complex} is an embedded two-dimensional cell complex $\mT^2 \subset M$ such that the 2-cells of $\mT^2$ are triangles.
\end{Def}

In other words, a triangle complex is a 2-dimensional simplicial complex in which we do not require that edges have distinct endpoints.  We will often arrange that the complementary regions of a triangle complex are tetrahedra, but we expect there will be applications of these results in more general situations.  For this reason, we will carry out our analysis in full generality.

A \textit{3-cell} of a triangle complex is the closure of a component of the complement $M \setminus \mT^2$.  If $S \subset M$ is a surface transverse to $\mT^2$ and $C$ is a 3-cell of $\mT^2$ then each component of $S \cap C$ will be called a \textit{piece} of $S$.  A K-disk for a piece $F$ of $S$ is a K-disk for $S$ that is contained in $C$ and whose intersection with $S$ is contained in $F$.  The \textit{descending link} for a piece $F$ is the subcomplex of the descending link for $S$ spanned by the K-disks for $F$.

K-disks in distinct pieces of $S$ will be disjoint or isotopic, so there is a projection map from the join of the descending links for the pieces of $S$ into the descending link for $S$.  If every piece has a well defined index, then so will their join, whose index will be the sum of the indices of the pieces.

Consider a boundary loop $\ell$ of an index-$n$ normal piece in a 3-cell $\sigma$.  If any arc of $\ell$ in the complement of the 1-skeleton of $\partial \sigma$ has both endpoints in the same edge then this arc determines a K-disk for the normal piece that is disjoint from all other K-disks for this piece.  Thus every arc of $\ell$ in the complement of the edges of $\partial \sigma$ must have its endpoint in different edges.  We will call such an arc a \textit{straight arc}.  An arc with its endpoints in different edges will be called a \textit{bent arc}.  A loop made up of straight arcs will be called a \textit{straight arc}.  From the above argument, it follows that every boundary component of an index-$n$ normal piece must be a straight loop.

\begin{Def}
A \textit{index-$n$ normal surface} with respect to a triangle complex $\mT^2$ is a surface $S$ such that $S \cap \mT^2$ does not contain a simple closed curve in a triangle, each piece of $S$ has well defined index and the sum of these indices is $n$.
\end{Def}

We will see that when the 3-cells of $\mT$ are tetrahedra, this corresponds to the usual definition of normal and almost normal surfaces.

\begin{Lem}
\label{normalizerlem}
If $\mT$ is a simple 2-complex in $M$ then every index-$n$ vertex in $\mS(M, \mT^1)$ is represented by an index-$n$ normal surface with respect to $\mT$.
\end{Lem}

\begin{proof}
If $S$ has index zero then $S$ has no K-disks so no piece of $S$ can contain a K-disk for $S$.  If $S \cap \mT^2$ contains a simple closed curve in a triangle then this curve bounds a disk in the triangle, and therefore this loop must be trivial in $S$.  A sphere blind isotopy of $S$ will remove this loop of intersection, and by continuing this process, we can ensure that $S \cap \mT^2$ will contain no simple closed curves.

For $n > 0$, let $v$ be an index-$n$ vertex in $\mS(M, \mT^1)$ and let $S$ be a surface representing $v$.  Because $v$ has index $n$, there is a map $\xi$ from an $n$-sphere $S^n$ into the descending link $L_v$ such that $\xi$ is not homotopy trivial.  As in the proof of Lemma~\ref{cgsteponelem}, define a triangulated Bachman ball $B^n$ whose boundary is $S^n$ such that every vertex corresponds to a representative for $v$ transverse to $\mT^2$ and between any two vertices that are connected by an edge in $B^n$, there is an isotopy between the associated surfaces that is tangent to a 2-cell in $\mT^2$ at exactly one point.  The construction is nearly identical to that in Lemma~\ref{cgsteponelem}, but with $D$ replaced by the triangles in $\mT^2$.  Because K-disks for $S$ only intersect the interiors of the triangles, the argument is identical.

Let $V(B^n)$ be the subcomplex of $B^n$ spanned by the set of vertices whose associated surfaces have non-trivial visible link.  In other words, this is the set of surfaces whose intersection with $\mT^2$ contains the boundary of a K-disk.  As in the proof of Lemma~\ref{cgsteponelem}, the descending links define a continuous Bachman map $\Psi$ from $V(B^n)$ into $L_v$ that agrees with $\xi$ on $S^n$.  If $V(B^n) = B^n$ then $\Psi$ defines a homotopy from $\xi$ to the constant map, contradicting the assumption that $\xi$ is homotopy non-trivial.  Thus there is at least one vertex $a \in B^n$ that is not contained in $V(B^n)$.

Let $S_a$ be a surface representing this vertex $a$.  Because the visible link of $S_a$ is empty, every arc of intersection between $S_a$ and a 2-cell of $\mT^2$ has its endpoints in different edges of the triangle.  Moreover, any loop of intersection between $S_a$ and a 2-cell must be trivial in $S_a \setminus \mT^1$.  

Define $L^a_v \subset L_v$ to be the sub-complex spanned by the set of K-disks for $S_a$ whose interiors are disjoint from the 2-cells in $\mT^2$.  We will define a projection $\pi : L^a_v \rightarrow L_v$ as follows:  Given a K-disk $D$ representing a vertex $u \in L_v$, if $u$ is already in $L^a_v$ then we define $\pi(u) = u$.  If $u$ is not in $L^a_v$ then $D$ must intersect the 2-cells of $\mT^2$ in a collection of arcs and loops.  We can remove all loop intersections so that $D \cap \mT^2$ consists of arcs with boundary in $S_a$.  

Let $O$ be an outermost disk cut off by one of these arcs.  The boundary of $O$ consists of an arc in $S_a$ and an arc $\alpha$ in a 2-cell $C$ of $\mT^2$.  The endpoints of $\alpha$ are contained in straight arcs $\gamma_1$, $\gamma_2$ of $S_a \cap C$.  Because $C$ is a simple cell, a zero-surgery in $C$ of $\gamma_1 \cup \gamma_2$ produces either a trivial loop or a bent arc in $C$, bounding a compressing disk or a bridge disk, respectively, in $S_a$.  The disk $D$ is disjoint from this K-disk and we will have $\pi$ send $D$ to one of the K-disks defined in this way.  In particular, if we choose an ordering on the vertices of $L_v$, we will have $\pi$ send $u$ to the K-disk with lowest order, defined in this way.

Note that if K-disks $D$, $D'$, representing $u$, $u'$, are disjoint then the K-disks representing $\pi(u)$ and $\pi(u')$ will be disjoint.  Thus the map $\pi$ extends to a continuous map between the simplicial complexes $L_v$ and $L^a_v$.  This map is a retraction so if $L^a_v$ has a non-trivial homotopy group of dimension less than $n$ then so does $L_v$.  Because we assumed $S$ has index $n$, the surface $S_a$ cannot have normal index less than $n$.  Since the map $\xi : S^n \rightarrow L_v$ is non-trivial and $L^a_v$ is contained in $L_v$, the map $\pi \circ \xi$ must be homotopy non-trivial in $L^a_v$, implying that $L^a_v$ has a non-trivial homotopy group in dimension $n$.  Thus $S_a$ has normal index $n$.
\end{proof}

\section{Normal disks in tetrahedra}
\label{triangsect}

In this section we classify the low index normal pieces in tetrahedra that are incompressible.  Because a tetrahdron is a ball, the only incompressible surfaces it contains are disks.  In fact, every loop $\ell$ in the boundary of a tetrahedron $\sigma$ bounds a unique (up to isotopy fixing $\ell$) disk, so we must classify the loops that bound index-$n$ normal disks.  Since we are only interested in the complex $\mD^2(M, \mT^1)$, we will restrict our attention to pieces with index at most two.  As noted above, the boundary of an index-$n$ normal piece must be a straight loop.  Thus we will consider all straight loops in the boundary of a tetrahedron and determine which bound index-$n$ normal disks for $n \leq 3$.

To classify straight loops, we will consider the complement in $\partial \sigma$ of its four vertices.  Given any loop $\ell \subset \partial \sigma$ in the complement of the four vertices, if $\ell$ contains a bent arc then there is an isotopy of $\ell$ that removes this bent arc.  (This isotopy will not be transverse to the 1-skeleton of $\partial \sigma$.)  If we isotope $\ell$ so as to minimize its intersection with the edges, the result will be a straight loop that is isotopic to $\ell$ in the complement of the vertices of $\sigma$.  Thus for each isotopy class of loops in the complement of the vertices, there is a unique straight loop up to transverse isotopy, and this loop bounds a unique disk in $\sigma$  

Around each vertex there is a straight loop bounding a disk that contains only that vertex.  The normal loop in each of these isotopy classes consists of three normal arcs.  A disk bounded by this loop will be called a \textit{normal triangle}.  The four normal disks in a tetrahedron are shown on the left in Figure~\ref{triandquadfig}.
\begin{figure}[htb]
  \begin{center}
  \includegraphics[width=4.5in]{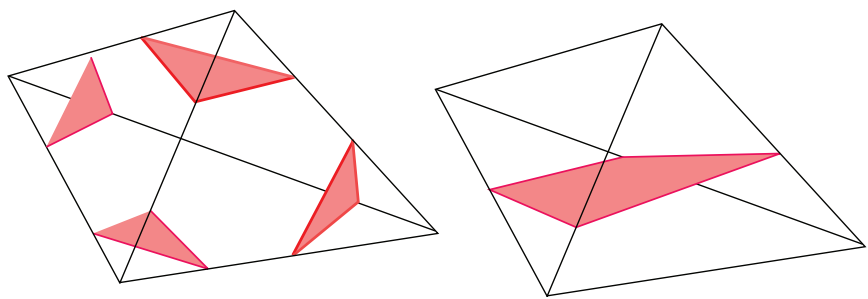}
  \caption{The two types of index-zero normal pieces are triangles and quadrilaterals.}
  \label{triandquadfig}
  \end{center}
\end{figure}

Let $\ell$ be a normal loop in one of the remaining isotopy classes.  The complement $\partial \sigma \setminus \ell$ consists of two disks, $D^+, D^-$ each containing two vertices.  In each disk, there is up to isotopy a unique arc $\alpha^+$, $\alpha^-$ connecting the two vertices.  If one of these arcs is isotopic to an edge of the triangulation of $\partial \sigma$ then the loop $\ell$ consists of four straight arcs.  The disk bounded by this $\ell$ is called a \textit{normal quadrilateral}.  For each pair of vertices, there is a unique edge connecting them and thus a unique normal quadrilateral.  One of the three quadrilaterals is shown on the right in Figure~\ref{triandquadfig}.

Normal triangles and quadrilaterals are incompressible.  A bridge disk for a normal piece intersects an edge of $\partial \sigma$ in two distinct points contained in the boundary of the normal piece.  The boundary of a normal triangle or quadrilateral intersects each edge of $\sigma$ at most once so it has no bridge disks.  Thus normal triangles and quadrilaterals are index-zero normal pieces.

Consider a straight loop $\ell$ defined by an $\alpha$ that intersects the 1-skeleton of $\sigma$ in exactly one point.  Any pair of vertices is connected by a unique such $\alpha$ and the corresponding straight loop consists of eight normal arcs.  A disk bounded by such an $\ell$ will be called a \textit{flat octagon}.  One such disk is shown on the left in Figure~\ref{octagonfig}, and is shown twisted in Figure~\ref{octbridgesfig} so that the bridges are clearer.
\begin{figure}[htb]
  \begin{center}
  \includegraphics[width=4.5in]{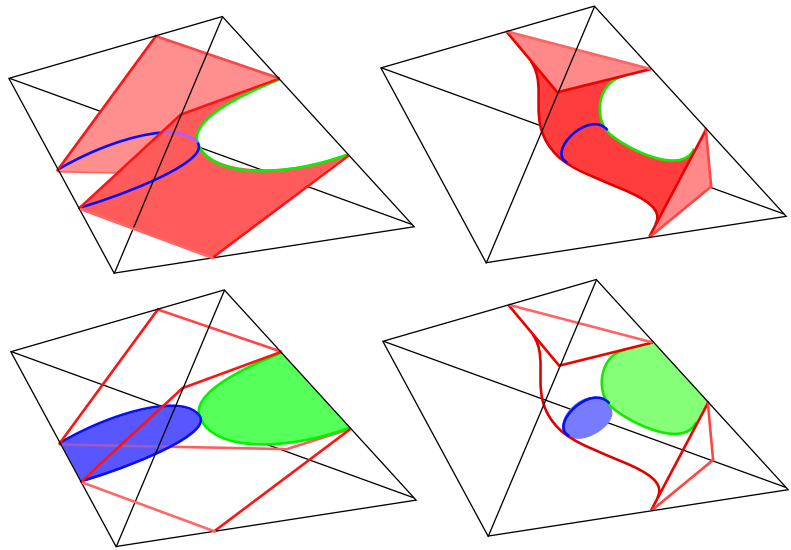}
  \caption{The two types of index-one pieces and their compressing disks.}
  \label{octagonfig}
  \end{center}
\end{figure}
\begin{figure}[htb]
  \begin{center}
  \includegraphics[width=4.5in]{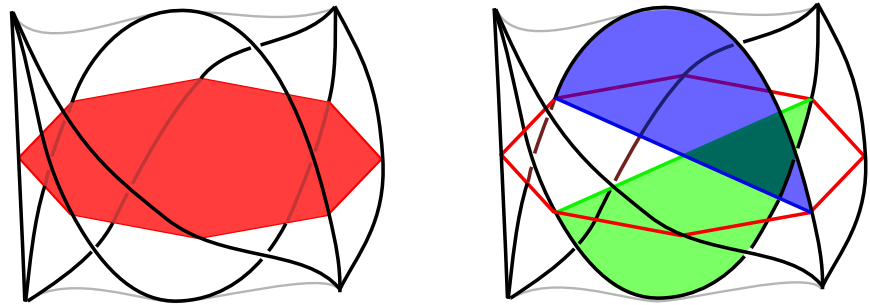}
  \caption{A normal octagon in a tetrahedron, drawn so that the octagon is horizontal and the edges form bridges.}
  \label{octbridgesfig}
  \end{center}
\end{figure}

The boundary of a normal octagon intersects two edges in two points and intersects the remaining four edges in one point each.  Because a bridge disk requires two points of intersection, there are exactly two bridge disks.  These two bridge disks intersect in the interior of the normal disk, so they are not connected by an edge in the descending link for the associated surface.  The zero homotopy group of an octagon is thus non-trivial, making the normal octagon an index-one normal piece for the tetrahedron.

Next consider the case when the arc $\alpha$ intersects the 1-skeleton of $\sigma$ in more than one point.  No arc between distinct vertices intersects exactly two edges so the next type of arc intersects three edges and the corresponding loop $\ell$ consists of twelve normal arcs and bounds a disk shown in Figure~\ref{dodecagonfig} that we will call a \textit{normal dodecagon}.  .  Again, there is a unique such arc and a unique dodecagon for each pair of vertices.
\begin{figure}[htb]
  \begin{center}
  \includegraphics[width=4.5in]{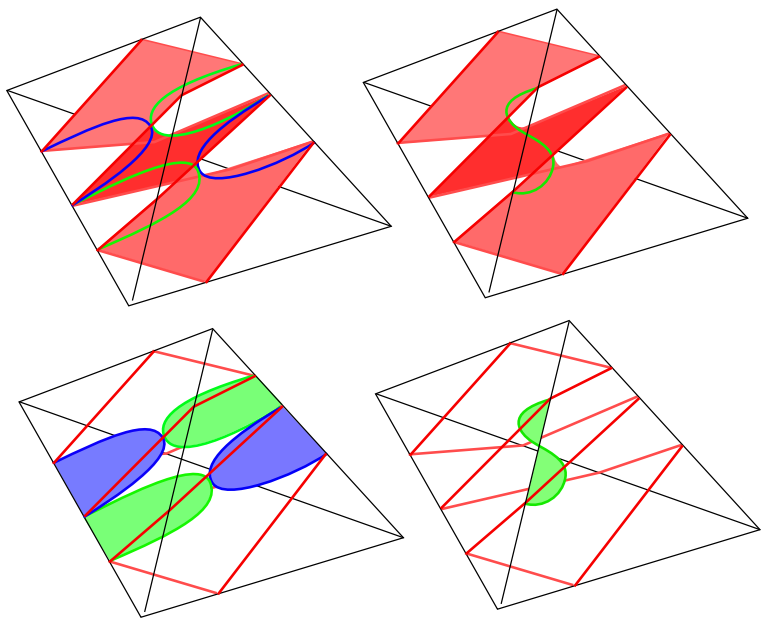}
  \caption{The two types of almost normal pieces and their compressing disks.}
  \label{dodecagonfig}
  \end{center}
\end{figure}

The boundary of a normal dodecagon intersects two opposite edges of $\sigma$ in three points, two edges in two points and two edges in one point.  There are two bridge disks intersecting each of the edges with three intersections and one bridge disk in each of the edges with two, for a total of six bridge disks.  The subcomplex of the descending link spanned by these six disks is connected, but contains a non-trivial loop of length four, so a normal dodecagon is an index-two normal piece in the tetrahedron.

All that remains is to show that disks with straight boundaries that intersect the 1-skeleton of $\sigma$ in more than 12 points are not normal disks of index two or less.  A complete classification of the indices of disks in tetrahedra has been worked out by Dave Bachman, but has not been published at the time of this writing.  The following proof came out of a conversation with Ryan Derby-Talbot using some of the ideas in Bachman's proof.  It is included here with the permission of both Ryan and Dave.

\begin{Lem}
\label{normaloctbridgeslem}
Every index-zero normal disk is a normal triangle or quadrilateral.

Every index-one normal disk is a normal octagon.

Every index-two normal disk is a normal dodecagon.
\end{Lem}

\begin{proof}
The indices of triangles, quadrilaterals, octagons and dodecagons were calculated above, so the statement will follow from showing that if $D$ is a disk with straight boundary in a tetrahedron $\sigma$ and $\partial D$ intersects the 1-skeleton of $\partial \sigma$ in more than 12 points then $D$ does not have index zero, one or two.  In particular, we must show that the descending link for such a disk $D$ is connected and simply connected.

The boundary of $D$ is determined by arcs $\alpha^+$ and $\alpha^-$ as above.  Let $k$ be the number of points where the interior of $\alpha^+$ intersects the 1-skeleton of $\partial \sigma$.  (This $k$ will always be an odd number.)  Then $\partial D$ will intersect the 1-skeleton of $\partial \sigma$ in $n = 2k + 6$ points, and we will call these points the \textit{corners of $D$}.

For every pair of corners of $D$ that are adjacent in a single edge of $\partial \sigma$, there is a single intersection of either $\alpha^+$ or $\alpha^-$ with that edge.  Thus the bridge disks on the positive/negative side of $D$  correspond to points where the interior of $\alpha^+$/$\alpha_-$, respectively, crosses an edge in $\partial \sigma$.  In particular, on either side the bridge disks intersect $D$ in a collection of pairwise-disjoint, parallel arcs.  Let $\beta_1,\dots,\beta_k$ be the arcs coming from disks on the positive side and let $\gamma_1,\dots,\gamma_k$ be the arcs coming from disks on the positive side.  There are exactly six corners of $D$ that are not endpoints of some $\beta_i$ and six corners that are not endpoints of some $\gamma_i$.  Each set of six consists of two clumps of three.

The set of bridge disks on each side of $D$ form a simplex in the descending link $L_D$.  Let $c_-$, $c_+$ be these simplices.  The entire link consists of $c_-$, $c_+$ and a collection of cells spanned by edges between $c_-$ and $c_+$.  In particular, two vertices in opposite simplices will be connected by an edge if and only if the corresponding arcs $\beta_i$, $\gamma_j$ have disjoint interiors in $D$.  (If  $\beta_i$, $\gamma_j$ intersect in an endpoint, there will be an edge between them.)

Assign a linear metric to $L_D$ such that every edge has length one and let $C \subset L_D$ be the set of points that are distance $\frac{1}{2}$ from both $c_-$ and $c_+$.  The complement of this set is the union a regular neighborhood of $c_-$ and a regular neighborhood of $c_+$.  Each of these regular neighborhoods is contractible since each of $c_-$, $c_+$ is a simplex.  Thus we can calculate the homotopy type of $L_D$ from the homotopy type of $C$ using Van Kampen's Theorem.  In particular, $L_D$ will be connected and simply connected if and only if $C$ is non-empty and connected.

The set $C$ is not a simplicial complex, but it is a cell complex and every vertex is given by a pair $\beta_i$, $\gamma_j$ with disjoint interiors.  The only way this set will be empty is if the endpoints of all the arcs $\{\gamma_i\}$ are contained in the six points that are not endpoints of $\{\beta_i\}$ and vice versa.  This is only possible if there are at most twelve corners of $D$.  Thus if $n > 12$ then $C$ is non-empty.

If $n > 12$ then $\beta_0$ will be disjoint from either $\gamma_0$ or $\gamma_k$, and $\beta_k$ will be disjoint from one or both of them.  Assume we have chosen labels so that $\beta_0$ is disjoint from $\gamma_0$ and $\beta_k$ is disjoint from $\gamma_k$.  For any other disjoint pair $\beta_i$, $\gamma_j$, note that $\beta_i$ separates $D$ so $\gamma_j$ will also be disjoint from either $\beta_0$ or $\beta_k$.  Assuming without loss of generality that it is disjoint from $\beta_0$, there will br an edge in $C$ from the vertex defined by $\beta_i$, $\gamma_j$ to the vertex defined by $\beta_0$, $\gamma_j$.  Repeating the argument with $\beta_0$, $\gamma_j$ switched, we see that every vertex in $C$ is connected to either the vertex defined by $\beta_0$, $\gamma_0$ or to the vertex defined by $\beta_k$, $\gamma_k$.

If $\beta_0$ is disjoint from both $\gamma_0$ and $\gamma_k$ then there is an edge between these two vertices, so $C$ is connected and we're done.  If there is a different arc $\gamma_i$ that is disjoint from both $\beta_0$ and $\beta_k$ then these vertices will be connected by a path of length three and $C$ will again be connected.  Otherwise, every arc $\gamma_i$ must intersect either $\beta_0$ or $\beta_k$.  Because there are three vertices between the endpoints of $\beta_0$, there are at most three arcs $\gamma_i$ that intersect the interior of $\beta_0$.  There are at most three more arcs that intersect the interior of $\beta_k$, so $k \leq 6$.  In fact, $k$ must be an odd number and there is (up to symmetries) exactly one pair of arcs $\alpha_+$, $\alpha_-$ in $\partial \sigma$ with $k = 5$.  The reader can check that for this disk, the arc $\gamma_3$ is disjoint from both $\beta_0$ and $\beta_5$.  Thus $C$ is connected for every $k > 3$, so there are no index-zero, -one or -two disks for $n > 12$.
\end{proof}

\section{Normal tubes in tetrahedra}
\label{ntubessect}

In this section, we classify the index-one and index-two normal pieces in a tetrahedron that are compressible.  A compressible normal piece must have index at least one since it will have a non-empty descending link.  

\begin{Lem}
\label{compdiskintetlem}
If $F$ is an index-$n$ normal piece of $S$ in a tetrahedron $\sigma$ and $F$ is not a disk then the subcomplex of the disk complex for $F$ spanned by compressing disks is contractible with diameter at most 2.
\end{Lem}

\begin{proof}
The boundary of $F$ is a collection of simple closed curves in the sphere boundary of $\sigma$.  Each of these loops is isotopy non-trivial in $F$ and bounds a disk in $\partial \sigma$.  An innermost loop of $\partial F$ in $\partial \sigma$ bounds a disk in the boundary of the tetrahedron, which is parallel to a compressing disk $D$ for $F$.  Any other compressing disk for $F$ has boundary disjoint from $\partial S$, so its boundary can be isotoped out of any regular neighborhood of $\partial S$.  In particular, any compressing disk can be isotoped disjoint from $\partial D$.  Thus every vertex of the descending link for $F$ represented by a compressing disk is connected to the vertex for $D$ by an edge.  This implies that this subcomplex is contractible with diameter at most two.
\end{proof}

\begin{Lem}
Every index-one normal piece $S$ in a tetrahedron is either a normal octagon or the result of attaching an unknotted tube between two index-zero disks.
\end{Lem}

\begin{proof}
By Lemma~\ref{compdiskintetlem}, if $S$ is not a disk, then the subcomplex of the descending link spanned by compressing disks is connected.  Thus if the disk complex for $S$ is disconnected then there must be a bridge disk $D'$ for $S$ that intersects every compressing disk.

Let $\alpha$ be the arc $D' \cap S$ and let $N$ be a closed regular neighborhood in $S$ of $\alpha$ and the boundary component or components of $S$ containing the endpoints of $\alpha$.  Let $\ell$ be the boundary of $N$ in the interior of $S$.  This may be one or two loops, each of which bounds a disk in $\sigma$ disjoint from $D'$.  Let $E$ be one of these disks.  If $\partial E$ is essential in $S$ then $E$ is a compressing disks for $S$, contradicting the fact that $D'$ intersects every compressing disk.  Thus each component of $\ell$ must be trivial in $S$.  This implies that $S$ is either a disk (in the case there $\ell$ is two loops) or an annulus.  In the later case, $S$ has a single compressing disk $D$ that intersects $D'$ in a single point, so the annulus is unknotted, as in Figure~\ref{octagonfig}.
\end{proof}

\begin{Lem}
Every index-two normal piece $S$ in a tetrahedron is either a normal dodecagon, the result of attaching a tube from an index-one disk to an index-zero disk, or the result of attaching two unknotted tubes between three index-zero disks.
\end{Lem}

\begin{proof}
By the construction in the previous Lemma, if $S$ is not a disk or an unknotted tube then every bridge disk $D$ for $S$ is disjoint from some compressing disk $C(D)$ and any compressing disk disjoint from $D$ is disjoint from $C(D)$.  By Lemma~\ref{compdiskintetlem}, any essential loop $E$ in the disk complex for $S$ must pass through at least one bridge disk.  If $E$ passes through three bridge disks in a row, we can insert between the two edges connecting them a pair of edges to a compressing disk, then back to the bridge disk.  Thus we can assume that $E$ passes through at most two bridge disks in a row.

If there is a single bridge disk $D$ in $E$ such that the vertices before and after $D$ are compressing disks then $C(D)$ is disjoint from both these disks, so we can isotope the two edges to a path that passes through $C(D)$ instead.  Thus we can assume that whenever $E$ passes through a bridge disk, it passes through two in a row.

Let $D$, $D'$ be the consecutive bridge disks in $E$.  We can isotope $E$ so that the vertex before $D$ is $C(D)$ and the vertex before $D'$ is $C(D')$.  If there is a compression disk $D''$ that is disjoint from both $D$ and $D'$ then this disk will be disjoint from both $C(D)$ and $C(D')$.  Since $D''$ is disjoint from all four disks, there are three triangles spanned by $D''$ and the three edges of the path $E$ from $C(D)$ to $C(D')$.  These triangles define a homotopy of the three edge of $E$ into the subcomplex of the disk complex spanned by compressing disks.  By assumption this cannot happen, so every compressing disk must intersect one of $D$ or $D'$.

Let $\alpha$, $\alpha'$ be the arcs or intersection $D \cap S$, $D' \cap S$.  Let $N$ be a regular neighborhood of these arcs and the boundary components of $S$ containing their endpoints.  The intersection of $\partial N$ with the interior of $S$ consists of one, two or three loops, each of which bounds a disk in $\sigma$ disjoint from $D$ and $D'$.  Because there are no compressing disks disjoint from $D$ and $D'$, these disks must have trivial boundary in $S$.  If $\alpha$ and $\alpha'$ are both non-separating in $S$ then $S$ is a pair of pants consisting of three index-zero normal disks connected by unknotted tubes.  If one is separating and the other non-separating then $S$ is the result of attaching a tube to a normal octagon.  If both are separating, then $S$ is incompressible, so $S$ is a normal dodecagon.
\end{proof}

\section{Bent surfaces}
\label{bentsurfacesect}

Now that we have characterized the vertices of $\mD^2(M, \mT^1)$ in terms of normal surfaces, we will turn to the task of calculating the edges and faces of $\mD^2$.

\begin{Def}
A \textit{bent surface} with respect to a triangulation $\mT$ is a surface $S$ that is transverse to the 1-skeleton $\mT^1$, intersects the 2-skeleton in a collection of bent and normal arcs, and intersects the interior of each tetrahedron in a collection of (open) disks.
\end{Def}

Note that every normal surface is a bent surface and an index-one or index-two normal surface will be a bent surface if and only if it has no tube pieces.  If $S$ is an index-one or index-two normal surface with one or two tube pieces, then we can isotope it to a bent surface by pushing the tube or tubes into faces of the 2-skeleton.  

Unlike a normal surface, a bent surface may contain extraneous pieces.  A \textit{skewered sphere} is a sphere bounding a ball $B$ such that $B \cap \mT^1$ is unknotted.  Such a sphere can be a component of a bent surface.  We will say that a bent surface $S'$ represents a normal surface $S$ if after removing all skewered spheres, either $S'$ is isotopic to the union of $S$, transverse to $\mT^2$ or $S'$ is the result of isotoping one or more tube pieces of $S$ into triangles of $\mT^2$.

The \textit{edge weight vector} of a bent surface $S$ is the vector of integers given by the intersection of $S$ with each edge in the triangulation.  Note that removing a skewered sphere reduces one of the coordinates of the edge weight vector by two.

\begin{Lem}
\label{finitebentsurfaceslem}
For each vector $w$ of edge weights there are finitely bent surfaces with edge weight vector $w$.
\end{Lem}

\begin{proof}
A bent surface is determined, up to isotopy, by its intersection with the 2-skeleton $\mT^2$.  For any finite collection of (an even number of) points in the boundary of a triangle, there are finitely many ways to connect these points with bent and normal arcs.  Thus for each vector of edge weights, there are finitely many ways of connecting the points in all the triangles of the 2-skeleton, and thus finitely many bent surfaces.
\end{proof}

Define $G(w)$ to be the graph whose vertices are tubed bent surfaces and edges are defined as follows:  Consider a bent surface $S$ representing a vertex in $G(w)$ and an arc $\alpha$ in a triangle $\tau$ of $\mT^2$ whose endpoints are contained in arcs of $S \cap \mT^2$ and whose interior is disjoint from $\mT^2$.  If we pinch the two arcs in $S \cap \mT^2$ across $\alpha$, the resulting collection of arcs in $\mT^2$ will define a new bent surface $S'$.  (This pinching is sometimes called zero surgery.)  We will call this construction a \textit{pinch move}.

Note that vertices of $G(w)$ that are connected by an edge may not correspond to isotopic surfaces.  In particular, one of the surfaces may be the result of compressing the other along a disk that intersects the 2-skeleton $\mT^2$ in the arc $\alpha$.  We will therefore keep track of the complexity of the surfaces represented by the vertices of $G(w)$.

\begin{Lem}
\label{indexonealglem}
Consider an index-one normal surface $S$ represented by a vertex $v \in G(w)$ and an index-zero normal surface $S'$.  Then there will be a vertex $v'$ in $G(w)$ representing $S'$ and a path in $G(w)$ from $v$ to $v'$ along which the complexity is non-increasing  if and only if there is a decreasing path in the complex of surfaces from $S$ to $S'$.
\end{Lem}

This is not the easiest way to calculate the two edges below an index-one vertex, but it is in the same line as the method we will use to calculate the edges below an index-two vertex in the next Lemma.

\begin{proof}
First note that a path in $G(w)$ along which the complexity does not increase determines a sequence of isotopies and compressions from $S$ to $S'$.  Thus the implication from a path in $G(w)$ to a path in $\mS$ is immediate.  The converse direction will be the focus of the remainder of the proof.

If $S$ is an index-one surface and $S'$ is an index-zero surface such that $S'$ is the result of K-compressing $S$ then by Lemma~\ref{barrieraxiomlem} any maximal collection of compressions on the same side of $S$ as the compressions producing $S'$ will also produce $S'$.  Let $D_0$ be a K-disk on the side of $S$ along which we would like to compress.  We can isotope $D_0$ so that $D_1 \cap \mT^2$ is a collection of arcs.  Let $E \subset D_0$ be an outermost disk cut off by an arc in $D_1 \cap \mT^2$.  The intersection $\partial E \cap \mT^2$ is an arc connecting two arcs of $S \cap \mT^2$.  The disk $E$ defines an isotopy of $S$ that is tangent to $\mT^2$ at one point and induces an arc pinch on the intersection $S \cap \mT^2$.  If the resulting surface is not flat then it is compressible on the same side as $D_1$ and we will replace $S$ with the bent surface $S_1$ that results from this compression.  Otherwise, if the resulting surface is bent, we will let $S_1$ be this bent surface.

In the latter case, when $S_1$ is isotopic to $S$, the isotopy takes $D_0$ to a compressing disk $D_1$ for $S_1$ that intersects $\mT^2$ in one fewer arcs than $D_0$.  If we continue the process, the disk will eventually become disjoint from $\mT^2$, so at some point the disk will define an arc pinch corresponding to a compression.  If the resulting bent surface $S_j$ is K-incompressible then it represents $S'$.  Otherwise, we will choose a K-disk $D_j$ for $S_j$ and continue the process.  Because each compressing disk intersects $\mT^2$ in a finite number of points and there is a bound on the number of times we can compress $S$, the process must terminate with a vertex $v' \in G(w)$ representing $S'$.  The sequence of surfaces induces a path in $G(w)$ and by construction, the complexity is non-increasing along this path.
\end{proof}

We will next show that $G(w)$ also determines all paths from an index-two to a lower index surface.  In this case, however, we cannot perform every compression that arises, so the argument is much more delicate.

Given an arc $\alpha$ along which we would like to pinch a bent surface $S$, there is a relatively simple criteria to determine whether the pinch will correspond to an isotopy, a compression, or adding a tube.  Let $\sigma_+$, $\sigma_-$ be the tetrahedra adjacent to the triangle containing $\alpha$.  The intersection $S \cap \partial \sigma_\pm$ is a collection of simple loops in the sphere boundary and $\alpha$ either connects two of these loops or connects a loop to itself.  The pinch will correspond to a compression if $\alpha$ connects a loop to itself in both tetrahedra.  It will correspond to a tubing if it attaches distinct loops in both tetrahedra, and it corresponds to an isotopy otherwise.

\begin{Lem}
\label{indextwoalglem}
Consider an index-two normal surface $S$ represented by a vertex $v \in G(w)$ and a lower index normal surface $S'$.  Then there will be a vertex $v'$ in $G(w)$ representing $S'$ and a path in $G(w)$ from $v$ to $v'$ along which the complexity is non-increasing  if and only if there is a decreasing path in the complex of surfaces from $S$ to $S'$.
\end{Lem}

\begin{proof}
As in the previous proof, a path in $G(w)$ along which the complexity does not increase determines a sequence of isotopies and compressions from $S$ to $S'$.  The implication from a path in $G(w)$ to a path in $\mS$ is immediate so the converse direction will be the focus of the remainder of the proof.

First assume that $S'$ is an octagon index-one normal surface.  Then there is a pair of bridge disks $D'_1$, $D'_2$ on opposite sides $S'$ that are contained in a tetrahedron $\sigma$ and intersect in a single point.  Because $S'$ is the result of K-compressing $S$, there are bridge disks $D_1$, $D_2$ for $S$ such that after the compressions, $D_1$ and $D_2$ become isotopic to $D'_1$, $D'_2$, respectively.  The sequence of K-compressions from $S$ to $S'$ is defined by a family of pairwise disjoint compressing disks $D_3,\dots, D_n$ for $S$ that are disjoint from $D_1$ and $D_2$.  If some $D_i$ (for $i \geq 3$) is a bridge disk, we will replace it with the compressing disk surrounding the bridge disk, so we can assume $D_3,\dots,D_n$ are compressing disks.

To get from $S$ to a representative for $S'$, we must isotope $S$ so that each $D_i$ is contained in a tetrahedron, then compress along $D_3,\dots,D_n$.  We can isotope the disks $\{D_i\}$ so that for each $i$, $D_i \cap \mT^2$ is a collection of arcs in $D_i$.  There will be an outermost arc $\alpha$ cutting off a disk $E \subset D_i$ such that $\partial E$ consists of $\alpha$ and an arc in $S$.  If the second arc of $\partial E$ is boundary parallel in $S \setminus \mT^2$ then we can isotope $D_i$ so as to eliminate $\alpha$, while fixing $S$.  Otherwise, the disk $E$ determines an isotopy of $S$ that is tangent to $\mT^2$ at exactly one point, and such that the image of $D_i$ after the isotopy intersects $\mT^2$ in one fewer arcs.  If we repeat this process with each disk in the set $\{D_i\}$, we will find a sequence of surfaces $S_0,\dots,S_m$ such that $S_0 = S$ and $S_m$ is the result of attaching tubes to $S'$ in the complement of $\mT^2$.  

Each $S_i$ intersects the triangles of $\mT^2$ in a collection of straight arcs, bent arcs and simple closed curves.  The collection of bent arcs and straight arcs determine a bent surface $B_i$ and we will define the complexity $c_i$ to be the complexity of this bent surface $B_i$.  Each $B_i$ is the result of compressing $S_i$ maximally in the complement of $\mT^2$, so if we can choose the sequence $\{S_i\}$ so the the complexity is non-increasing then the corresponding sequence of bent surfaces will define a sequence of isotopies and compressions from $S$ to $S'$.

The sequence $\{S_i\}$ is determined by the collection of disks $\{D_i\}$ and the order in which we isotope across outermost disks.  If there is a step where the complexity increases, let $k$ be the first value where $c_{k+1} > c_k$.  Let $D_i$ be the disk containing the outermost disk $E$ that defines the isotopy from $S_k$ to $S_{k+1}$.

In this case, there is a disk $E'$ whose boundary consists of an arc in $\mT^2$ and an arc in $S_k$ such that an isotopy across $E'$ takes $S_k$ back to $S_{k-1}$.  If the isotopy from $S_{k-1}$ to $S_k$ is defined by a disk on the side opposite $D_i$, the disks $E$ and $E'$ will be disjoint.  Because the isotopy defined by $E$ increases the complexity, $E$ is contained in a non-disk component of $S_k$.  If we isotope along $E'$ to get back to $S_{k-1}$ and the image disk $E$ still has this configuration then reversing the orders of the moves will make the move that increases complexity happen earlier.  Otherwise, the isotopy defined by $E'$ must turn the non-disk component touching $E$ into a disk component, so that the isotopy across $E$ disk not increase the complexity.  In this case, the move from $S_{k-1}$ to $S_k$ must decrease the complexity by pulling out the compression that is removed by $E$.  In this case, switching the order creates two moves along which the complexity stays the same.  So, switching the orders either eliminates a move in which the complexity increases or makes such a move happen earlier, without creating any new such moves.

We can thus assume that the isotopy from $S_{k-1}$ to $S_k$ is defined by an outermost disk on the same side as $D_i$.  A similar argument applies to each of the isotopies leading up to $S_k$, so we can assume that each of these isotopies was defined by an outermost disk on the same side as $D_i$.  If an isotopy defined by an outermost disk in $D_i$ pulls a compression disk for $S$ into the complement of $\mT^2$ then this compression disk must be on the same side as $D_i$.  If such an isotopy pulls a compression disk into the 2-skeleton then it must be on the opposite side.  Thus a sequence of compressions defined by the same $D_i$ cannot both create a compression in the complement of $\mT^2$ and then eliminate it.  This contradiction implies that we must have gotten rid of the isotopy that increased the complexity already.

We can thus find a sequence of surfaces $\{S_i\}$ such that the complexities of these surfaces form a non-increasing sequence.  If we compress the final $S_i$ along all the disks $D_3,\dots,D_k$, the resulting surface is isotopic to $S'$ and has a bridge disk on each side contained in a tetrahedron.  Because $S'$ is isotopic to an octagon normal surface, the only representative with this property is the bent surface.  This implies that the bent surfaces $\{B_i\}$ corresponding to $\{S_i\}$ define a path in $G(w)$ from a bent surface representing $S$ to a bent surface representing $S'$ along which the complexity does not increase, as promised by the Lemma.

In the case when $S'$ is an index-zero surface, the proof is almost identical, but without the disks $D_1, D_2$.  In particular, we let $D_1,\dots, D_n$ be compressing disks that we compress across to get $S'$, then isotope $S$ along outermost disks in $D_1,\dots,D_n$ in a way that the complexity is non-increasing.

For the final case, assume $S'$ is a tube index-one normal surface.  Let $D_1$ be the image in $S$ of the compressing disk for $S'$ dual to the tube and let $D_2$ be the image in $S$ of a bridge disk that intersects this tube in a single point and is contained in a tetrahedron.  (Note that the disk $D_2$ may not be contained in a tetrahedron, but after the isotopy to $S'$, it will.)  Assume that we have isotoped $\partial D_1$ so that the one point of intersection between $D_1$ and $D_2$ is adjacent in $\partial D_2$ to the arc $\partial D_2 \cap \mT^1$.  Let $D_3,\dots,D_k$ be a complete collection of compressing disks, disjoint from $D_1$, $D_2$, that turn $S$ into $S'$.

We will define a sequence of surfaces $S_1,\dots,S_n$ with $S = S_1$, by isotoping across outermost disks in $\{D_i\}$, as with the octagon normal surface $S'$, but with one difference:  In this sequence, we will only isotope $D_1$ to a disk that intersects $\mT^2$ in a single arc, adjacent to the point of intersection with $D_2$.  By repeating the argument above, we can choose this sequence of surfaces so that the complexity forms a non-increasing sequence.

If we push the tube dual to $D_1$ into the tetrahedron containing $D_2$ then as with the octagon surface, we will get a tube normal surface representing $S'$.  We thus only need to check that $S'$ is the result of pushing this tube into a triangle of $\mT^2$.  The only way it will not is if the tube dual to $D_2$ is already contained in the tetrahedron, i.e. the disk $D_2$ can be isotoped away from $\mT^2$ without changing $S_n$.  In this case there is one or two isotopies that push the tube back into the triangle.  If we append these to the original, the complexities will no longer be non-increasing.  However, these moves are disjoint from all the original disks, so they commute with all of the previous moves.  Thus we can again rearrange the sequence so that it the complexity is non-increasing and the final surface represents $S'$.
\end{proof}

Finally we note that an index-one or index-two surface with tubes will, in general, be represented by more than one bent surface, depending on which triangle in $\mT^2$ we push the tube(s) into.  We therefore must determine which of the vertices in the graph $G(w)$ represent the same surface.

\begin{Lem}
\label{bentsurfaceisotopiclem}
Two bent surfaces with weight vector $w$ will represent the same index-zero, -one or -two surface (i.e. up to isotopy transverse to $\mT^1$) if and only if they are connected by a path in $G(w)$ with constant complexity.
\end{Lem}

\begin{proof}
Any path in $G(w)$ where the complexity is constant represents an isotopy transverse to $\mT^1$, so one direction is immediate.  For the other direction, assume $S$ and $S'$ are bent surfaces representing isotopic index-zero, -one or -two surfaces.

Because an index-zero surface intersects the triangles of $\mT$ in straight arcs and the tetrahedra in disks, it is determined entirely by its intersection with the 1-skeleton, so two index-zero normal surfaces are isotopic transverse to the 1-skeleton $\mT^1$ if and only if they are isotopic relative to the 2-skeleton $\mT^2$.  A similar argument applies to index-one and index-two normal surfaces without tubes.  

For surfaces with tubes, we will follow the proof of Lemma~\ref{indextwoalglem}.  For each tube in $S$, there is a bridge disk $D_1$ for $S$ contained in a triangle of $\mT$ and a compressing disk $D_2$ that intersects $\mT^2$ in a single arc.  The isotopy from $S$ to $S'$ takes $D_1$, $D_2$ to a pair of K-disks $D'_1$, $D'_2$ for $S'$.  (If $S$ has two tubes there will a second, disjoint pair of disks as well.)  If we isotope $S'$ across outermost disks in $D'_1$, $D'_2$, we will find a sequence of bent surfaces from $S'$ to $S$.  Following the argument in the proof of Lemma~\ref{indextwoalglem}, we can order the sequence so that the complexity of the intermediate bent surfaces is non-increasing.  This determines a sequence of edges in $G(w)$ with constant complexity from $S'$ to $S$.
\end{proof}

\section{The proof of Theorem~\ref{mainthm2}}
\label{thm2sect}

In this section, we prove Theorem~\ref{mainthm2}, which states that given a partially flat angled ideal triangulation $\mT$ and an integer $g$, we can calculate the subcomplex of $\mD^2(M, \mT^1)$ spanned by vertices representing surfaces with genus at most $g$.  (For a disconnected surface, we mean that the sum of the genera of the components is at most $g$.)

\begin{proof}[Proof of Theorem~\ref{mainthm2}]
The vertices of $\mD(M, \mT^1)$ correspond to index-zero, -one and -two vertices in the complex of surfaces $\mS(M, \mT^1)$.  By Lemma~\ref{normalizerlem}, these vertices correspond to index-zero, -one and -two normal surfaces with respect to $\mT$.  

By Theorem 4.3 in~\cite{lacknb:1eff}, there is for any integer $n$, a finite collection of connected index-zero and -one normal (i.e. normal and almost normal) surfaces with respect to $\mT$ whose genus is at most $n$, and there is an algorithm to compute them.  Every disconnected surface is a union of normal and almost normal surfaces with positive genus, so there is a finite number of (possibly disconnected) index-zero and index-one normal surfaces such that the sum of the genera of the components is at most $n$.  

The argument in~\cite{lacknb:1eff} uses standard linear programming techniques combined with the fact (Theorem 2.1 in~\cite{lacknb:1eff}) that in a partially flat angled ideal triangulation, the only non-negative Euler characteristic normal surface is a a vertex linking torus.  These techniques are equally applicable to index-two normal surfaces, though we will leave the details to the reader.  Lemma~\ref{bentsurfaceisotopiclem} shows that there is an algorithm to determine when two index-zero, -one or -two surfaces represent the same vertex in $\mD^2$, so there is an algorithm to compute the finitely many index-zero, -one or -two normal surfaces of genus at most $g$.  This algorithm constructs the vertices of $\mD^2(M, \mT^1)$ with genus at most $g$.

Next we must compute the edges and faces of $\mD(M, \mT^1)$.  Because there are finitely many index-one and index-two vertices, it will be sufficient to compute the edges and faces below each one.  Every edge below an index-one or index-two vertex corresponds to a sequence of K-compressions that end at a surface lower index.  In Section~\ref{bentsurfacesect}, we defined the graph of bent surfaces for a given edge weight vector.  By Lemma~\ref{indexonealglem}, every such path below an index-one surface is represented by a path in this graph of bent surfaces.  Similarly, Lemma~\ref{indextwoalglem} shows that every such path below an index-two normal surface corresponds to a path in its bent surface graph.  By Lemma~\ref{finitebentsurfaceslem}, each of these graphs is finite and constructible, so there is an algorithm to find all such paths.

Once we have calculated the edges, we will attach a triangle to every loop of edges whose vertices have distinct indices.  By construction these are the only faces in $\mD^2(M, \mT^1)$ and every such loop bounds a face.  Since there are finitely many edges, there are finitely many such paths, so this final part of the construction is algorithmic as well.
\end{proof}

\appendix
\section{Reference list of axioms}
\label{axiomsect}

\noindent
\textbf{The Morse axiom}: Every 2-cell in $\mS$ is a diamond, a triangle or a bigon as in the construction of $\mS(M, \mT^1)$.  Given three edges such that any two bound a 2-cell, the projection of any two across the third will determine a face, an edge or will project to a point.  Every $n$-cell $C$ is defined by mapping the boundary of an $n$-cube in the $n-1$-skeleton via projections. \\

\noindent
\textbf{The net axiom}: For any vertex $v \in \mS$, there is an integer $\ell(v)$ such that every edge path starting at $v$, along which the complexity strictly decreases, has length at most $\ell(v)$. \\

\noindent
\textbf{The parallel orientation axiom}: For any 2-cell $q$ in $\mS$, the orientations on the edges of $q$ make it a parallel-oriented diamond, triangle or bigon. \\

\noindent
\textbf{The Casson-Gordon axiom}: Let $v$ be a maximum in an oriented path $E$, and let $v_-$, $v_+$ be the minima of $E$ right before and after $v$, respectively.  If $v_-$ is compressible to the positive side then either the path link of $v$ is contractible or $v_+$ is compressible to the positive side.  Similarly, if $v_+$ is compressible to the negative side then either the path link of $v$ is contractible or $v_-$ is compressible to the negative side.  \\

\noindent
\textbf{The barrier axiom}:  Given any vertex $v \in \mS$, there are vertices $v_-$, $v_+$ and paths $E_-$, $E_+$ starting in $v_-$ and ending in $v_+$, respectively such that the following hold: Any directed path descending from $v$ can be extended to a decreasing path ending in $v_+$ that is equivalent to $E_+$.  Any reverse-directed path descending from $v$ can be extended to an decreasing path ending at $v_-$ that is equivalent to $E_-$. \\

\noindent
\textbf{The translation axiom}: Let $q^+$ be an $n$-cell in $\mS$ such that the  edges of $q^+$ all point away from $v$ and let $q^-$ be an $m$-cell in $\mS$ such that the edges of $q^-$ all point towards $v$.  Then there is an $(n+m)$-cell $C$ isomorphic to $q^+ \times q^-$ such that $q^+ = q^+ \times \{v\}$ and $q^- = \{v\} \times q^+$. Moreover, $C$ is unique up to automorphisms of $\mS$ fixing the vertices and edges of the complex. \\

\bibliographystyle{amsplain}
\bibliography{thin}

\providecommand{\bysame}{\leavevmode\hbox to3em{\hrulefill}\thinspace}
\providecommand{\MR}{\relax\ifhmode\unskip\space\fi MR }
\providecommand{\MRhref}[2]{%
  \href{http://www.ams.org/mathscinet-getitem?mr=#1}{#2}
}
\providecommand{\href}[2]{#2}
\begin{thebibliography}{10}

\bibitem{bachman:normal}
D.~Bachman, \emph{2-normal surfaces}, preprint (2003).

\bibitem{bach:gordon}
\bysame, \emph{{Connected sums of unstabilized Heegaard splittings are
  unstabilized}}, preprint (2004), math.GT/0404058.

\bibitem{bach:stabs}
\bysame, \emph{{Stabilizations of Heegaard splittings of sufficiently
  complicated 3-manifolds (Preliminary Report)}}, preprint (2008),
  arXiv:0806.4689.

\bibitem{bachman}
\bysame, \emph{{Stabilizations of Heegaard splittings of sufficiently
  complicated 3-manifolds (Preliminary Report)}}, preprint (2008),
  arXiv:0806.4689.

\bibitem{bdts:dehn}
D.~Bachman, R.~Derby-Talbot, and E.~Sedgwick, \emph{{Surfaces that become
  isotopic after Dehn filling}}, preprint (2010), arXiv:1001.4259.

\bibitem{cass:red}
A.~Casson and C.~Gordon, \emph{{Reducing Heegaard splittings}}, Topology Appl.
  \textbf{27} (1987), 275--283.

\bibitem{cho}
S.~Cho, \emph{{Homeomorphisms of the 3-sphere that preserve a Heegaard
  splitting of genus two}}, Proc. Amer. Math. Soc. \textbf{136} (2008),
  1113Ð1123.

\bibitem{gabai}
D.~Gabai, \emph{{Foliations and the topology of 3-manifolds. III.}}, J.
  Differential Geom. \textbf{26} (1987), no.~3, 479--536.

\bibitem{haken}
W.~Haken, \emph{Some results on surfaces in 3-manifolds}, {MAA Studies in
  Mathematics}, vol.~5, The Mathematical Association of America, 1968.

\bibitem{hayashimo}
C.~Hayashi and K.~Shimokawa, \emph{Thin position of a pair (3-manifold,
  1-submanifold)}, Pacific J. Math. \textbf{197} (2001), no.~2, 301--324.

\bibitem{me:stabs}
J.~Johnson, \emph{{Bounding the stable genera of Heegaard splittings from
  below}}, preprint (2008), arXiv:0708.2683.

\bibitem{lacknb:1eff}
M.~Lackenby, \emph{{An algorithm to determine the Heegaard genus of simple
  3-manifolds with non-empty boundary}}, preprint (2007), arXiv:0709.0376.

\bibitem{li:alg}
T.~Li, \emph{{An algorithm to determine the Heegaard genus of a 3-manifold}},
  preprint (2010), arXiv:1002.1958.

\bibitem{manjarrez}
F.~Manjarrez-Gutierrez, \emph{{Circular thin position for knots in the
  3-sphere}}, preprint (2008), arXiv:0810.3742.

\bibitem{qiuscharl}
Scharlemann~M. Qiu, R., \emph{{A proof of the Gordon conjecture}}, Adv. Math.
  \textbf{222} (2009), no.~6, 2085--2106.

\bibitem{reid}
K.~Reidemeister, \emph{{Zur dreidimensionalen Topologie}}, Abh. Math. Sem.
  Univ. Hamburg \textbf{11} (1933), 189--194.

\bibitem{rub:alnormal}
J.~H. Rubinstein, \emph{{Polyhedral minimal surfaces, Heegaard splittings and
  decision problems for 3-manifolds}}, 1997, pp.~1--20.

\bibitem{sch:thin}
M.~Scharlemann and A.~Thompson, \emph{Thin position for 3-manifolds},
  Contemporary Mathematics \textbf{164} (1994), 231--238.

\bibitem{sch:cross}
J.~Schultens, \emph{{The classification of Heegaard splittings for
  \textit{(compact orientable surface)} $\times S^1$}}, Bull. Lond. Math. Soc.
  \textbf{15} (1993), 425--448.

\bibitem{schl:vert}
\bysame, \emph{{Heegaard splittings of Seifert fibered spaces with boundary}},
  Trans. Amer. Math. Soc. \textbf{347} (1995), no.~7, 2533--2552.

\bibitem{sch:width}
\bysame, \emph{Width complexex of knots and 3-manifolds}, Pac. J. Math.
  \textbf{239} (2009), no.~1, 135--156.

\bibitem{sing}
J.~Singer, \emph{{Three-dimensional manifolds and their Heegaard diagrams}},
  Trans. Amer. Math. Soc. \textbf{35} (1933), 88--111.

\bibitem{stevens}
A.~Stevens, \emph{{K-stable equivalence for knots in Heegaard surfaces}},
  preprint (2009), arXiv:0902.3707.

\bibitem{stocking}
M.~Stocking, \emph{Almost normal surfaces in 3-manifolds}, Trans. Amer. Math.
  Soc. \textbf{352} (2000), no.~1, 171--207.

\bibitem{tom:brcompare}
M.~Tomova, \emph{Multiple bridge surfaces restrict knot distance}, preprint
  (2005), math.GT/0511139.

\bibitem{tomovataylor}
M.~Tomova and S.~Taylor, \emph{{Essential surfaces in (3-manifold, graph) pairs
  and leveling edges of Heegaard spines}}, preprint (2009), arXiv:0910.3251.

\end{thebibliography}

\end{document}